\documentclass[UTF-8,reqno]{amsart}
\usepackage{enumerate}
\usepackage{mhequ}
\usepackage[margin=1.0in]{geometry}
\usepackage{amssymb,url,color, booktabs,nccmath}
\usepackage{mathrsfs}
\usepackage{enumitem}
\usepackage{graphicx}
\usepackage{tikz}
\usetikzlibrary{shapes,snakes}
\usetikzlibrary{calc}
\usetikzlibrary{decorations.shapes}
\usepackage{color}
\usepackage[colorlinks=true]{hyperref}
\hypersetup{
    linkcolor=blue,          
    citecolor=red,        
    filecolor=blue,      
    urlcolor=cyan
}

\definecolor{darkergreen}{rgb}{0.0, 0.5, 0.0}

\numberwithin{equation}{section}

\newcommand{\be}{\begin{eqnarray}}
\newcommand{\ee}{\end{eqnarray}}
\newcommand{\ce}{\begin{eqnarray*}}
\newcommand{\de}{\end{eqnarray*}}
\newtheorem{theorem}{Theorem}[section]
\newtheorem{lemma}[theorem]{Lemma}
\newtheorem{remark}[theorem]{Remark}
\newtheorem{definition}[theorem]{Definition}
\newtheorem{proposition}[theorem]{Proposition}
\newtheorem{Examples}[theorem]{Example}
\newtheorem{corollary}[theorem]{Corollary}

\newenvironment{nouppercase}{%
  \renewcommand{\uppercasenonmath}[1]{}}{}

\def\[{{\Big[}}
\def\]{{\Big]}}
\def\<{{\langle}}
\def\>{{\rangle}}
\def\({{\Big(}}
\def\){{\Big)}}

\def\bx{{\mathbf{x}}}

\def\sgn{\mbox{\rm sgn}}

\def\min{{\mathord{{\rm min}}}}

\def\={&\!\!=\!\!&}

\def\1{{\mathbf{1}}}

\def\geq{\geqslant}
\def\leq{\leqslant}
\def\ge{\geqslant}
\def\le{\leqslant}

\def\\nabla\cdot {\mathord{{\rm \nabla\cdot }}}

\def\[{{\Big[}}
\def\]{{\Big]}}
\def\<{{\langle}}
\def\>{{\rangle}}
\def\({{\Big(}}
\def\){{\Big)}}

\def\bx{{\mathbf{x}}}

\def\sgn{\mbox{\rm sgn}}

\def\min{{\mathord{{\rm min}}}}

\def\={&\!\!=\!\!&}
\def\bt{\begin{theorem}}
\def\et{\end{theorem}}
\def\bl{\begin{lemma}}
\def\el{\end{lemma}}
\def\br{\begin{remark}}
\def\er{\end{remark}}
\def\bx{\begin{Examples}}
\def\ex{\end{Examples}}
\def\bd{\begin{definition}}
\def\ed{\end{definition}}
\def\bp{\begin{proposition}}
\def\ep{\end{proposition}}
\def\bc{\begin{corollary}}
\def\ec{\end{corollary}}

\def\geq{\geqslant}
\def\leq{\leqslant}
\def\ge{\geqslant}
\def\le{\leqslant}

\def\\nabla\cdot {\mathord{{\rm \nabla\cdot }}}

\def\<{\langle} \def\>{\rangle}

\allowdisplaybreaks

\tikzset{
        dot/.style={circle,fill=black,inner sep=0pt, outer sep=0.7pt, minimum size=1mm},
        Phi/.style={white!40!red,thick,snake=coil,segment amplitude=0.6pt, segment length=2pt},
         Z/.style={black!40!green,thick,snake=coil,segment amplitude=0.6pt, segment length=2pt},
        C/.style={thick,black!20!blue},
          Cr/.style={thick,black!20!red},
            Cg/.style={thick,black!20!green},
       }

\begin{document}
\title[McKean-Vlasov PDE and large deviations]{\LARGE McKean-Vlasov PDE with Irregular Drift and Applications to Large Deviations for Conservative SPDEs}

\author[Zhengyan Wu]{\large Zhengyan Wu}
\address[Z. Wu]{Department of Mathematics, University of Bielefeld, D-33615 Bielefeld, Germany.School of Mathematical Sciences, University of Chinese Academy of Sciences, China}
\email{zwu@math.uni-bielefeld.de}
\author{Rangrang Zhang}
\address[R. Zhang]{School of Mathematics and Statistics,
Beijing Institute of Technology, Beijing, 100081, China
}
\email{rrzhang@amss.ac.cn}

\begin{abstract}
Inspired by [Fehrman, Gess; Invent. Math., 2023], we provide a fine analysis of the McKean-Vlasov PDE with singular interactions and  drift terms of square root form. As the corresponding skeleton equation of Dean-Kawasaki equation with singular interactions (a stochastic, conservative
PDE), it determines the rate function of small noise large deviations. By imposing Ladyzhenskaya-Prodi-Serrin type conditions on the interaction kernel,
we establish the large deviations in the framework of stochastic renormalized kinetic solution, when the intensity and the correlation of the noise are simultaneously sent to $0$ under a suitable scaling. This result contributes to demonstrating
the consistency between the macroscopic fluctuation theory associated with singular interacting mean-field systems and fluctuating hydrodynamics related to the Dean-Kawasaki equation. As an application, we also obtain large deviations for the other stochastic conservative PDE called fluctuating Ising-Kac-Kawasaki dynamics. It is of great importance in exploring fluctuations of Kawasaki dynamical Ising-Kac model, since they formally exhibits the same key features in terms of Gaussian fluctuations, large deviations, and scaling limits near criticality.

%

\end{abstract}


\date{\today}

\begin{nouppercase}
\maketitle
\end{nouppercase}

\setcounter{tocdepth}{1}

\section{Introduction}
In this paper, we perform an analysis on the well-posedness and the stability of the McKean-Vlasov partial differential equations (PDEs) with an irregular drift
\begin{equation}\label{mean-fieldPDE}
\partial_t\rho=\Delta\rho-\nabla\cdot(\rho V\ast\rho)-\nabla\cdot(\sqrt{\rho}g), 	
\end{equation}
where $g\in L^2([0,T]\times\mathbb{T}^d;\mathbb{R}^d)$, the function $V$ is singular in the sense that it is only $L^p$-integrable and $\ast$ stands for spatial convolution. In addition to the mathematical significance due to the singular coefficients, (\ref{mean-fieldPDE}) is the skeleton equation corresponding to the dynamical zero noise large deviations of Dean-Kawasaki equation with singular interactions
\begin{align}\label{DK-1}
 \partial_t \rho^{\varepsilon,K(\varepsilon)}=\Delta \rho^{\varepsilon,K(\varepsilon)}-\nabla\cdot(\rho^{\varepsilon,K(\varepsilon)} V\ast \rho^{\varepsilon,K(\varepsilon)})-\sqrt{\varepsilon}\nabla\cdot(\sqrt{\rho^{\varepsilon,K(\varepsilon)}} \circ \xi^{K(\varepsilon)}),
 \end{align}
in a joint scaling regime $(\varepsilon,K(\varepsilon))\rightarrow(0,+\infty)$. Here, $\xi^K$ is a spatial regularization of space time white noise and $\circ$ denotes the Stratonovich integration. In fact, with (\ref{mean-fieldPDE}), the
large deviations of (\ref{DK-1}) can be described by the rate function
\begin{align}\label{k-7}
I(\rho)=\frac{1}{2}\inf\Big\{\|g\|^2_{L^2(\mathbb{T}^d\times[0,T];\mathbb{R}^d)}:\partial_t \rho=\Delta\rho-\nabla\cdot(\rho V\ast\rho)-\nabla\cdot(\sqrt{\rho} g)\Big\}.
\end{align}

The zero noise large deviations of Dean-Kawasaki equation (\ref{DK-1}) can contribute to demonstrating the consistency between the macroscopic fluctuation theory (MFT) associated with singular interacting mean-field systems and fluctuating hydrodynamics (FHD) related to the Dean-Kawasaki equation. Over the past two decades, significant strides have been made in understanding non-equilibrium diffusive systems through MFT, which promotes the understanding of processes far from equilibrium and refined near-equilibrium linear approximations (see Bertini et al. \cite{BDGJL}, Derrida \cite{Derrida}). The starting point of MFT is an ansatz on large deviations for interacting particle systems, which has aroused great interest from the mathematics community working on the large deviation principle \cite{CG22,IS}. The theory of FHD offers a framework for modeling microscopic fluctuations consistent with statistical mechanics and non-equilibrium thermodynamics, wherein a stochastic, conservative PDE is postulated to capture fluctuations in systems out of equilibrium (see \cite{LL87}, \cite{HS}). Hence, the fundamental ansatz of MFT can be obtained from the zero noise large deviations for this conservative stochastic PDE. For instance, as a typical model in FHD, the Dean-Kawasaki equation (\ref{DK-1}) can accurately predict fluctuations in mean-field systems, for details, see the subsection \ref{subsec-1.1}, (1).

The other motivation of studying (\ref{mean-fieldPDE}) comes from its  background from singular interacting mean-field systems.
Referring to \cite{WWZ22}, equation (\ref{DK-1}) is strongly related to the mean-field system:
\begin{align}\label{IPS}
dX_i(t)=\frac{1}{N}\sum^N_{j\neq i}V(X_i(t)-X_j(t))dt+\sqrt{2}dB_i(t),
\end{align}
where $V$ is a singular interaction kernel. Such a singular interacting mean-field system has a background in different fields, for example, in plasma physics \cite{D79}, in fluids \cite{JO04}, and  in biological sciences \cite{CCH14}. The large $N$ limit of (\ref{IPS}) leads to the so-called McKean-Vlasov PDE:
\begin{equation}\label{mean-fieldPDE-0}
\partial_t\rho=\Delta\rho-\nabla\cdot(\rho V\ast\rho).
\end{equation}
The McKean-Vlasov PDE (\ref{mean-fieldPDE-0}) has a rich structure, which has aroused great interest in the mathematical community and fruitful results have been achieved. For example, it can be taken as a nonlinear Fokker-Planck equation associated with distribution-dependent stochastic differential equations (SDEs), see \cite{RZ18,BR23,BR23spdeac} and references therein. Moreover, when the kernel is in a typical form $V=-\nabla \mathcal{U}$, it can be represented as a gradient flow on the Wasserstein space with respect to the energy functional:
\begin{equation}
\mathcal{E}(\rho)=\int_{\mathbb{T}^d}\rho\log\rho+\rho\mathcal{U}\ast\rho dx.
\end{equation}
For further feature of the McKean-Vlasov PDE as a gradient flow, we refer readers to \cite{JKO98,Otto01}.
Clearly, (\ref{mean-fieldPDE}) can be viewed as the McKean-Vlasov PDE driven by a highly nonlinear irregular force. The analysis of (\ref{mean-fieldPDE}) requires us to deal not only with the McKean-Vlasov PDE itself, but also with the difficulties caused by the irregular force term.

In the sequel, we always assume that $T\in(0,\infty)$ is a fixed time horizon. Closely following the well-posedness of (\ref{DK-1}), we impose the same conditions on the kernel $V$ as \cite{WWZ22}. Let the spatial dimension $d\geq 2$. The kernel
 $V:[0,T]\times\mathbb{T}^d\rightarrow\mathbb{R}^d$ is assumed to satisfy a Ladyzhenskaya-Prodi-Serrin (LPS) type condition:
\begin{description}
\item[{\bf Assumption (A1)}] $V\in L^{p^{*}}([0,T];L^{p}(\mathbb{T}^d))$ with $\frac{d}{p}+\frac{2}{p^{*}}\leq 1$, $2\leq p^{*}<\infty$ and $d<p\leq \infty$,
\item[{\bf Assumption (A2)}] $\nabla\cdot V\in L^{q^{*}}([0,T];L^{q}(\mathbb{T}^d))$ with $ \frac{d}{2q}+\frac{1}{q^{*}}\leq 1$, $1\leq q^{*}<\infty$ and $d/2<q\leq \infty$.
\end{description}
The LPS condition is a regularity criterion proposed by Prodi \cite{Pro59}, Serrin \cite{Ser62}, and Ladyzhenskaya \cite{Lad67} in studying the uniqueness of the 3D Navier-Stokes equations. Later, the work of \cite{CI07} leads to the study of stochastic differential equations (SDEs) with a singular drift that satisfies the LPS condition. Various literature in this direction includes \cite{KR05, Zha05,Zha11,XXZZ20,RZ18} and references therein.

Let $\rm{Ent}(\mathbb{T}^d)$ be the space of finite entropy dissipation functions given by
\begin{align*}
{\rm{Ent}}(\mathbb{T}^d):=\Big\{\rho\in L^1(\mathbb{T}^d): \rho\geq0\ a.e.\ {\rm{and}}\ \int_{\mathbb{T}^d}\rho\log\rho dx<\infty\Big\}.
\end{align*}
The well-posedness and stability of (\ref{mean-fieldPDE}) reads as follows.
\begin{theorem}\label{thm-1-main}(cf. Propositions \ref{L1uniq}, \ref{prp-2} and \ref{prp-3})
Assume that $ V$ satisfies Assumptions (A1) and (A2). For any
 $\rho_0\in \rm{Ent}(\mathbb{T}^d)$ and $g\in L^2([0,T]\times\mathbb{T}^d;\mathbb{R}^d)$, there exists a unique weak solution to (\ref{mean-fieldPDE}). Furthermore, for any $N>0$, assume that $\{g_n\}_{n\in\mathbb{N}_+}, g\subseteq L^2([0,T]\times\mathbb{T}^d;\mathbb{R}^d)$ satisfying
\begin{align*}
\sup_{n\geq 1}\int^T_0 \|g_n(s)\|^2_{L^2(\mathbb{T}^d)}ds\leq N,
\quad g_n\rightharpoonup g\ {\rm{weakly\ in\ }} L^2([0,T]\times\mathbb{T}^d;\mathbb{R}^d).
\end{align*}
For every $n\in \mathbb{N}$, let $\rho_n$ be the weak solution of the skeleton equation (\ref{skeleton}) with control $g_n$ and initial data $\rho_0$. Let $\rho$ be the weak solution of the skeleton equation to (\ref{skeleton}) with control $g$ and initial data $\rho_0$. Then
\begin{align*}
\rho_n\rightarrow\rho\ {\rm{strongly\ in\ }} L^1([0,T];L^1(\mathbb{T}^d)),
\end{align*}
as $n\rightarrow \infty$.
\end{theorem}

In the framework of large deviations, equation (\ref{mean-fieldPDE}) is the corresponding skeleton equation of (\ref{DK-1}). We emphasize that the uniqueness of the skeleton equation plays a key role in proving large deviations. To illustrate it, we mention several works. In the absence of uniqueness of the skeleton equation, \cite{Daniel23} proposed a counterexample to exhibit the violation of the lower bound of large deviations for the Boltzmann equation. \cite{GHW23} linked the full large deviations of the Landau-Lifshitz-Navier-Stokes equations to the open problems of the energy equality and uniqueness of the skeleton-Navier-Stokes equations. \cite{FG23} provided an approach to show the consistency of the rate function with its lower semi-continuous envelope by using the stability of the skeleton equation, and thereby, they established a full large deviations for the rescaled zero-range process. Thanks to Theorem \ref{thm-1-main}, we are able to establish a full dynamical large deviations of (\ref{DK-1}), which can be formulated as follows.
\begin{theorem}{(cf. Theorem \ref{thm-2})}
Assume that $V$ satisfies Assumptions (A1) and (A2). Let $\rho_0\in  {\rm{Ent}} (\mathbb{T}^d)$, $\varepsilon\in (0,1)$, and $K\in \mathbb{N}$ satisfying $K(\varepsilon)\rightarrow \infty$ and $\varepsilon K^{d+2}(\varepsilon)\rightarrow 0$. Then the solutions $\{\rho^{\varepsilon,K(\varepsilon)}(\rho_0)\}_{\varepsilon\in (0,1)}$ of (\ref{DK-1}) satisfy large deviation principles with rate function $I$ on $L^1([0,T];L^1(\mathbb{T}^d))$, where $I$ is given by (\ref{k-7}).
\end{theorem}

Due to structural similarity, we can also obtain large deviations of a conservative SPDE called  fluctuating Ising-Kac-Kawasaki dynamics
\begin{equation}\label{ising-kac-equation}
\partial_t\rho^{\varepsilon}=\Delta\rho^{\varepsilon}-\beta\nabla\cdot[(1-(\rho^{\varepsilon})^2) \nabla J\ast\rho^{\varepsilon}]-\varepsilon^{1/2}\nabla\cdot(\sqrt{1-(\rho^{\varepsilon})^2}\circ \xi^{K(\varepsilon)}),
\end{equation}
where $\beta$ is a parameter related to temperature and $J$ stands for Kac potential. This equation was
 proposed by \cite{GJE99} when considering nonlinear fluctuation of Ising-Kac-Kawasaki dynamics in a neighborhood of the critical temperature.
As explained in \cite{WWZ22}, this equation can be regarded as a continuum phenomenological model of the Kawasaki dynamical Ising-Kac spin system. Precisely, when the spatial dimension $d=1$, (\ref{ising-kac-equation}) shares the same fluctuations feature with Kawasaki dynamical Ising-Kac process in terms of Gaussian fluctuations, large deviations and scaling limit near criticality. This will be explained in details in subsection \ref{subsec-1.1}, (2).  In this paper, we will prove one of the above features rigorously:  the dynamical large deviations for (\ref{ising-kac-equation}) from its hydrodynamical limits.
In addition, (\ref{ising-kac-equation}) belongs to Ginzburg-Landau dynamics and satisfies fluctuation-dissipation relation, see \cite{HS}.

Define $\Psi(\xi)=\frac{1}{2}[(1+\xi)\log(1+\xi)-(\xi+1)+(1-\xi)\log(1-\xi)-(1-\xi)]$.
Let the initial data be subject to the space
\begin{align*}
\overline{\text{{\rm{Ent}}}}(\mathbb{T}^{d})=\Big\{\rho: -1\leq\rho\leq1\ a.e.,\ {\rm{and}}\ \int_{\mathbb{T}^{d}}\Psi(\rho(x))dx<\infty\Big\}.
\end{align*}
For simplicity, let the parameter $\beta\equiv1$. The result of large deviations of (\ref{ising-kac-equation}) reads as follows.
\begin{theorem}{(cf. Theorem \ref{thm-LDP-ising})}\label{thm-0}
Assume that $J\in C^{\infty}(\mathbb{T}^d)$. Let $\rho_0\in  \overline{\rm{Ent}} (\mathbb{T}^d)$, $\varepsilon\in (0,1)$, and $K\in \mathbb{N}$ satisfying $K(\varepsilon)\rightarrow \infty$ and $\varepsilon K^{d+2}(\varepsilon)\rightarrow 0$. Then the solutions $\{\rho^{\varepsilon,K(\varepsilon)}(\rho_0)\}_{\varepsilon\in (0,1)}$ of (\ref{ising-kac-equation}) satisfy large deviation principles with rate function $I$ on $L^2([0,T];L^2(\mathbb{T}^d))$, where $I$ is given by
\begin{align*}
I(\rho)=\frac{1}{2}\inf\Big\{\|g\|^2_{L^2(\mathbb{T}^d\times[0,T];\mathbb{R}^d)}:
\partial_t \rho=\Delta\rho-\nabla\cdot[(1-\rho^2) \nabla J\ast\rho]-\nabla\cdot(\sqrt{1-\rho^2} g),\  \rho(0)=\rho_0
\Big\}.
\end{align*}
\end{theorem}

\subsection{Applications.}\label{subsec-1.1}
The content is divided into two parts.
We first explain the application of well-posedness and stability of (\ref{mean-fieldPDE}) to singular interacting mean-field systems. Then, we explore the fluctuation behaviors of (\ref{ising-kac-equation}) and Kawasaki dynamical Ising-Kac model.

{\bf(1) Application of Theorem \ref{thm-1-main} to singular interacting mean-field systems.}
In \cite{BKL95}, the authors obtained a restricted large deviations of zero-range processes, that is, the upper bound matches the lower bound only on a class of regular fluctuations. Since then, how to achieve the full large deviations of zero-range processes becomes a longstanding open problem. Recently, a nice work from \cite{FG23} resolves it by showing that the rate function of large deviations for zero-range processes is consistent with its lower-semicontinuous envelope on every path with finite entropy dissipation. The key of the proof is the well-posedness and stability of the skeleton equation. For details, see \cite[Section 8]{FG23}. Based on this result, \cite{GH23} proved full large deviations for rescaled zero-range processes by a path-wise regularity argument. Thus, the large deviations of fluctuating hydrodynamic equations can contribute to the study of interacting particle systems.

Inspired by \cite{FG23} and \cite{GH23}, with the well-posedness and stability of (\ref{mean-fieldPDE}) in hand, we can expect that similar results hold for singular interacting mean-field systems.
 As explained in \cite{WWZ22}, the Dean-Kawasaki equation (\ref{DK-1}) exhibits the same fluctuation features with interacting mean-field systems (\ref{IPS}). There are a series of works on dynamical large deviations for mean-field systems, for example, \cite{DJ87,BDF12,CG22}. We mention the result of \cite{CG22}, wherein the interaction kernel $V=-\nabla^{\top}\mathcal{N}=(-\partial_{x_2}\mathcal{N},\partial_{x_1}\mathcal{N})$ on a two-dimensional torus and $\mathcal{N}$ can be Green function of the Laplacian or a logarithmic function. The authors of \cite{CG22} established the lower bound and upper bound separately, when restricting to the space of finite energy:
\begin{equation*}
Q_T(\rho):=\sup_{t\in[0,T]}\Big(\frac{1}{2}\int_{\mathbb{T}^2}\rho(t)\mathcal{N}\ast\rho(t) dx+\frac{1}{2}\int^t_0\|\rho(s)-1\|_{L^2(\mathbb{T}^2)}^2ds\Big)<\infty. 	
\end{equation*}
In particular, we point out that when proving the lower bound, \cite{CG22} employed an extension argument from a regular domain $\mathcal{J}$ (see \cite[Lemma 2.9]{CG22}) to the space of finite energy: $Q_T(\rho)<\infty$. In general, such an extension argument strongly relies on the consistency of the rate function and its lower semicontinuous envelope. The result of \cite[Section 8]{FG23} implies that the stability of the skeleton equation plays an important role in verifying the lower semicontinuous envelope of the rate function. Therefore, we try to apply Theorem \ref{thm-1-main} to the study of large deviations for mean-field systems with general interactions (for instance, LPS type interactions), especially in the 'extension' step for the lower bound.

{\bf(2) Fluctuation behaviors of (\ref{ising-kac-equation}) and Kawasaki dynamical Ising-Kac model.}
The Ising model with a Kac interaction was introduced to recover the van der Waals theory of phase transition \cite{HKU}. Two main dynamics for the Ising-Kac model are Glauber dynamics and Kawasaki dynamics. For the latter, the time evolution of the spin systems involves neighboring atoms swapping their spins with a frequency that depends on the current configuration and the temperature (see \cite{HS}).

Concerning the fluctuation behaviors of (\ref{ising-kac-equation}) and Kawasaki dynamical Ising-Kac model, it formally suggests that
 (\ref{ising-kac-equation}) exhibits the same three key features as the Kawasaki dynamical Ising-Kac model in the terms of Gaussian fluctuations, large deviations, and scaling limits near criticality.
Let $\sigma(\cdot,t)$ denote the Ising-Kac-Kawasaki process defined by \cite[(6.6)]{GJE99}, whose hydrodynamic limit $m^{(K)}$ fulfills that
\begin{align}\label{rr-1}
	\partial_t m^{(K)}=\Delta m^{(K)}-\beta\nabla\cdot[(1-(m^{(K)})^2) \nabla J\ast m^{(K)}].
\end{align}
Then,
for every $\phi\in C^{\infty}_c(\mathbb{R}^d)$, the fluctuation field given by
\begin{equation*}
X^{\varepsilon}(t,\phi):=\varepsilon^{d/2}\sum_{x\in\mathbb{Z}^d}\phi(\varepsilon x)[\sigma(x,\varepsilon^{-2}t)-m^{(K)}(\varepsilon x,t)]
\end{equation*}
is conjectured to converge to $\bar{\rho}^1(\phi)$, where $\bar{\rho}^1$ solves a linear SPDE \cite[(6.25)]{GJE99}:
\begin{align}\label{rr-2}
\partial_t\bar{\rho}^1=\Delta\bar{\rho}^1-&\beta\nabla\cdot[(1-(m^{(K)})^2)\nabla J\ast\bar{\rho}^1]\\
+&2\beta\nabla\cdot[\bar{\rho}^1\nabla J\ast m^{(K)}]-\nabla\cdot\Big(\sqrt{1-(m^{(K)})^2}\xi\Big)	,
\end{align}
as $\varepsilon\rightarrow0$, where $\xi$ is the space-time white noise.
For the fluctuation behavior of (\ref{ising-kac-equation}), it
 informally satisfies
\begin{equation*}
	\varepsilon^{-1/2}(\rho^{\varepsilon}-m^{(K)})\rightharpoonup\bar{\rho}^1,
\end{equation*}
as $\varepsilon\rightarrow0$, where $m^{(K)}$ and $\bar{\rho}^1$ are given by (\ref{rr-1}) and (\ref{rr-2}), respectively. This implies that (\ref{ising-kac-equation}) and Kawasaki dynamical Ising-Kac model have the same Gaussian fluctuations and large deviations.
In addition, in one spatial dimension, for any $a\in\mathbb{R}$, let $\beta=1+a\varepsilon^{2/3}$, the rescaled density field
\begin{equation*}
\rho_{\varepsilon}(x,t):=\varepsilon^{-1/3}\rho^{\varepsilon}(\varepsilon^{-1/3}x,\varepsilon^{-4/3}t)
\end{equation*}
is conjectured to converge to the conservative stochastic Cahn-Hilliard equation \cite{GJE99}
\begin{equation}\label{cahnhilliard}
\partial_tu=-\partial_{xx}^2(\partial_{xx}^2u-u^3+au)-\partial_x\xi,
\end{equation}
as $\varepsilon\rightarrow0$. Hence, (\ref{ising-kac-equation}) shares the same scaling limits near criticality as the discrete Ising-Kac-Kawasaki particle system as well. However, this conjecture still remains widely open both in continuous and discrete setting, only some progress in the discrete setting has been made by \cite{IM18}. Regarding the large deviations, we refer readers to \cite{ME07} for
Kawasaki dynamics, and the above Theorem \ref{thm-0} is the large deviations for (\ref{ising-kac-equation}). Remarkably, structural similarities between the two rate functions can be observed.
Motivated by these results, we devote to exploring fluctuations of (\ref{ising-kac-equation}) to understand the behavior of  Ising-kac kawasaki model and its continuous analogue.

\subsection{Key ideas and technical comments}
(1){\bf{ The McKean-Vlasov PDEs.}} The analysis of the McKean-Vlasov PDEs (\ref{mean-fieldPDE})
 involves three aspects: path-wise uniqueness, equivalence between different concepts of solution, and weak-strong continuity. Since the last one is a byproduct of proving the well-posedness, we will make comments on the first two aspects.

\noindent (i) The main obstacle of uniqueness arises from the square root coefficient and the singular nonlocal interaction kernel. The former has been overcome by \cite{FG23}, where the authors introduced cut-off functions,  then employed renormalized kinetic formulation and  $L^1$-theory. However, cut-off functions cannot control the nonlocal term. To prove the uniqueness, we employ the Fisher information of the solution to control the nonlocal term and utilize Gronwall's inequality, which is different from \cite{FG23}, where all terms vanish in the $L^1$-framework.

\noindent(ii) Equivalence between weak solutions and renormalized kinetic solutions. Inspired by the gradient flow structure of the Dean-Kawasaki equation, an entropy dissipation estimates holds that can provide compactness of weak solutions in $L^{1}([0,T];L^1(\mathbb{T}^d))$ (see Lemma \ref{lem-3}). Hence, we derive the existence of weak solutions of the skeleton equation by introducing two approximating equations: smooth kernel and smooth coefficients. Compared with \cite{FG23}, we need to pass to the limit of the singular interaction term, which is nontrivial. As stated in the above (i), the uniqueness is established for renormalized kinetic solutions. Therefore, a bridge between weak solutions and renormalized kinetic solutions has to be built. We emphasize that the regularity required to guarantee the well-definedness of weak solutions and renormalized kinetic solutions is different. Precisely, for any $\varphi \in C^{\infty}_c(\mathbb{T}^d\times(0,\infty))$, the interaction kernel term involved in the definition of renormalized kinetic solution (see Definition \ref{dfn-3} later on) is of the form:
\begin{align*}
\int_0^t\int_{\mathbb{T}^d}\nabla(\varphi(x,\rho))\cdot\rho V\ast \rho dxds=&\int_0^t\int_{\mathbb{T}^d}\nabla_x\varphi(x,\rho)\cdot\rho V\ast \rho dxds\\
&+\int_0^t\int_{\mathbb{T}^d}\partial_{\xi}\varphi(x,\rho)2\nabla\sqrt{\rho}\cdot\sqrt{\rho}\rho V\ast (\nabla\rho)dxds\\
\leq &C(\varphi)\int_0^t\int_{\mathbb{T}^d}|\nabla\sqrt{\rho} \cdot V\ast \rho| dxds.
\end{align*}
Therefore, it requires that $\nabla\sqrt{\rho}\cdot V\ast\rho\in L^1([0,T];L^1(\mathbb{T}^d))$. However, for any $\psi \in C^{\infty}(\mathbb{T}^d)$, the interaction kernel term appeared in the definition of weak solution (see Definition \ref{dfn-1} later) is
\begin{align*}
\int^t_0\int_{\mathbb{T}^d}\rho V\ast\rho\cdot\nabla\psi(x)dxds \leq C(\psi)\int^t_0\int_{\mathbb{T}^d}|\rho V\ast\rho|dxds.
\end{align*}
In this case, it requires that $\rho V\ast\rho\in L^1([0,T];L^1(\mathbb{T}^d;\mathbb{R}^d))$. As a result, we will discuss the equivalence under different regularity conditions. The main idea of this part basically follows from \cite{FG23}, where the authors mollify the equation and apply additional regularity (provided by the entropy dissipation estimate) to pass to the limits of the commutator. Compared to previous literature, we highlight the analysis of the singular interaction kernel term.

{\bf{ (2) The weak convergence approach to LDP of Dean-Kawasaki equation (\ref{DK-1}).}} To prove the Freidlin-Wentzell's LDP, an important tool is the weak convergence approach developed by Dupuis and Ellis in \cite{DE}. The key idea of this approach is to prove a certain variational representation formula about the Laplace transform of bounded continuous functionals, which then leads to the verification of the equivalence between the LDP and the Laplace principle. In particular, for Brownian functionals, an elegant variational representation formula has been established by Bou\'{e} and Dupuis in \cite{MP} and by Budhiraja and Dupuis in \cite{BD}.
In this paper, we adopt the sufficient condition for the large deviation criteria of Budhiraja, Dupuis, and Maroulas \cite{BDM11} for functionals of Brownian motions, which consists of two parts. The first part is to show the weak-strong continuity of the skeleton equation, which is ensured by the stability in (1).
The second part is to prove the weak convergence of the perturbations of (\ref{DK-1}) (it is called  stochastic controlled equation, see (\ref{s-3}) below) in the random directions of the Cameron-Martin space of the driving Brownian motions. The fundamental problem is to establish the well-posedness of the renormalized kinetic solutions to the stochastic controlled equation. The standard $L^p(\mathbb{T}^d)-$estimates for $p\in (1,\infty)$ are not available in the presence of non-local kernel term, which is compensated by $L^1(\mathbb{T}^d)$ preservation and entropy estimates. Combining techniques from \cite{FG23} and \cite{WWZ22}, the well-posedness can be obtained. Moreover, with the aid of the entropy dissipation estimates, it implies the tightness of the laws of renormalized kinetic solutions in $L^1([0,T];L^1(\mathbb{T}^d))$.
Also, we need to show that the renormalized kinetic solution of stochastic controlled equation converges in distribution to the solution of skeleton equation when the control term converges weakly. Since the skeleton equation can be equipped with weak solution and kinetic solution at the same time, we have two choices to achieve the second part. However, we are not able to prove that the renormalized kinetic solution of stochastic controlled equation converges to the kinetic solution of skeleton equation due to the difficulty arising from the term:
\begin{align}\label{w-1}
\lim_{\varepsilon\rightarrow 0}  \Big( 2\int^t_0\int_{\mathbb{T}^d}\tilde{\rho}^{\varepsilon}\nabla \sqrt{\tilde{\rho}^{\varepsilon}}P_{K(\varepsilon)}\tilde{g}^{\varepsilon}(x,s)\partial_{\xi}\psi(x,\tilde{\rho}^{\varepsilon})dxds\Big),
\end{align}
where the product of the weakly convergent gradient and control terms appears. As a result, we cannot identify the limit. The crucial observation is that this term does not appear in the weak solution of skeleton equation. Therefore, we follow the idea of \cite{FG23} to introduce test functions that are compactly in the velocity variable to recover the classic weak formulation of the skeleton equation.

Recently, Matoussi, Sabbagh and Zhang in \cite{MSZ21} proposed a modified sufficient condition to verify the large deviation criteria of Budhiraja, Dupuis and Maroulas for functionals of Brownian motions.
 This condition has been succeed in proving large deviations of scalar stochastic conservation laws by \cite{DWZZ20} and the other references \cite{LSZZ23,WZ22}.
 The characteristic of this sufficient condition is that the control terms of the stochastic controlled equation and the skeleton equation are the same (denoted by $g^{\varepsilon}$). Under this circumstance, only the weakly convergent gradient term appears in (\ref{w-1}). Inspired by \cite{DWZZ20}, we try to apply the doubling variables method to directly show that the kinetic solution of the stochastic controlled equation converges to the kinetic solution of the skeleton equation. Referring to \cite[Theorem 3.3, (3.13)]{FG24}, we need to handle one of the control terms:
\begin{align*}
Q:=4\int_{\mathbb{R}}\int_{(\mathbb{T}^d)^3}(P_{K(\varepsilon)}g^{\varepsilon}(x,t)-g^{\varepsilon}(x',t))
\sqrt{\rho^1}\sqrt{\rho^2}\cdot \nabla_x \sqrt{\rho^1}\bar{\kappa}^{\gamma,\delta}_{1,t}\bar{\kappa}^{\gamma,\delta}_{2,t}\zeta^M(\eta)dxdx'dyd\eta,
\end{align*}
where $\rho^1$ and $\rho^2$ are kinetic solutions of the stochastic controlled equation and the skeleton equation, respectively. The functions $\bar{\kappa}^{\gamma,\delta}_{1,t}=\kappa^{\gamma,\delta}(x,y,\rho^1,\eta)$, $\bar{\kappa}^{\gamma,\delta}_{2,t}=\kappa^{\gamma,\delta}(x',y,\rho^2,\eta)$ with $\kappa^{\gamma,\delta}$ being the standard convolution kernel on spatial and velocity variables. The continuous piecewise linear function $\zeta^M$ is compactly supported in $[1/M, M+1]$. $P_{K(\varepsilon)}$ is Fourier projection.
To prove the convergence, it requires to show that $Q\rightarrow0$ as $\gamma, \delta$ and $\varepsilon$ converge to $0$.
Note that
\begin{align*}
Q=&4\int_{\mathbb{R}}\int_{(\mathbb{T}^d)^3}(P_{K(\varepsilon)}g^{\varepsilon}(x,t)-g^{\varepsilon}(x,t))
\sqrt{\rho^1}\sqrt{\rho^2}\cdot \nabla_x \sqrt{\rho^1}\bar{\kappa}^{\gamma,\delta}_{1,t}\bar{\kappa}^{\gamma,\delta}_{2,t}\zeta^M(\eta)dxdx'dyd\eta\\
&+ 4\int_{\mathbb{R}}\int_{(\mathbb{T}^d)^3}(g^{\varepsilon}(x,t)-g^{\varepsilon}(x',t))
\sqrt{\rho^1}\sqrt{\rho^2}\cdot \nabla_x \sqrt{\rho^1}\bar{\kappa}^{\gamma,\delta}_{1,t}\bar{\kappa}^{\gamma,\delta}_{2,t}\zeta^M(\eta)dxdx'dyd\eta\\
=:& Q_1+Q_2.
\end{align*}
Clearly, $Q_2$ converges to $0$ when $\gamma\rightarrow 0$. However, we are not able to show that $Q_1$ converges to $0$ in the presence of the Fourier projection $P_{K(\varepsilon)}$, which results in the failure of the above approach.

\subsection{Comments on the literature}
The Dean-Kawasaki equation (together with several variants) has been an active research topic in various branches of non-equilibrium statistical mechanics over the last twenty years. In \cite{RK09}, von Renesse and Sturm constructed a process having the Wasserstein distance as its intrinsic metric, which can formally be seen as a solution to the Dean-Kawasaki equation with the kernel term replaced by a nonlinear operator. Lehmann, Konarovskyi and von Renesse \cite{KLR19} asserted that some correction term is in fact necessary for the existence of (nontrivial) solutions to these Dean-Kawasaki type models.
Konarovskyi, Lehmann and von Renesse \cite{KLR20} show the existence of
 measure-valued solutions to (\ref{DK-1}) driven by It\^o type space-time white noise in certain parameter regimes when $V$ is smooth. For more results on Dean-Kawasaki model, we refer the readers to \cite{AR10, CSZ19, KLR19} and references therein.

 Regarding the large deviations for SPDEs in singular limits, we refer to Faris and Jona-Lasinio \cite{FJ82}, Jona-Lasinio and Mitter \cite{JM90}, Cerrai and Freidlin \cite{CD19}, and Hairer and Weber \cite{HW}. Let us mention particularly the last work \cite{HW}, where the authors show that the renormalization constants for the stochastic Allen-Cahn equation appear in the rate function. This completely differs from the present paper, where possible renormalization constants vanish in the joint scaling regimes of the noise intensity and regularity of the space-time white noise.

Regarding the large $N$ behavior of singular interacting mean-field systems, we display the following literature. \cite{WZZ22} proved Gaussian fluctuations of interacting mean-field systems including Biot-Savart type singular kernel. \cite{HRZ22} established propagation of chaos for mean-field systems with $L^p$-type interactions by using Zvonkin's transformation. \cite{HHMT} proved  large deviations for mean-field systems with $L^p$-type interactions. Further analysis of more general singular kernels have been made by \cite{JW18, BJW20, Ser20, RS, WZZ22b}.

We display some literature relevant to dynamical Ising-Kac model as well. The conjecture concerning nonlinear fluctuations of the Glauber dynamics in 1, 2 and 3 dimensions is completely settled and the rescaled density field is shown to converge to dynamical $\Phi^4_d$ model. See \cite{BPRS94,FR95} for $d=1$, \cite{MW17} for $d=2$, and \cite{GMW23} for $d=3$. For the macroscopic limits of Kawasaki dynamics, we refer readers to \cite{GL97}. Further analysis of Kawasaki dynamics can be found in \cite{VY97,GL15,Gia91,LOP91}.

\subsection{Structure of the paper}
Some basic notations and assumptions on the interaction kernel are presented in Section \ref{sec-2}. Section \ref{sec-3} is devoted to showing well-posedness and stability of the McKean-Vlasov partial differential equations. Preliminaries on large deviations and the main result are presented in Section \ref{sec-4}. In Section \ref{sec-5}, well-posedness and tightness of renormalized kinetic solutions to the stochastic controlled equation are proved.
Section \ref{sec-6} contains the proof of large deviations for Dean-Kawasaki equation with non-local interactions in a joint scaling regime. An application to the large deviations for Ising-Kac-Kawasaki equation is presented in Section \ref{sec-7}.

\section{Preliminaries}\label{sec-2}
\subsection{Notations}
We closely follow the framework of \cite{WWZ22}. {\bf{Throughout the whole content, we always assume that $T\in(0,\infty)$ is a fixed time horizon.}}
Let $\nabla$ and $\nabla\cdot$ be the derivative operator and the divergence operator with respect to the space variable $x\in\mathbb{T}^d$, respectively.
We denote by $L^p(\mathbb{T}^d)$ the Lebesgue spaces for $p\in [1,\infty]$.
 Let $\|\cdot\|_{L^p(\mathbb{T}^d)}$ denote the norm of $L^p(\mathbb{T}^d)$. When $p=2$, the inner product of $L^2(\mathbb{T}^d)$ will be denoted by
$(\cdot,\cdot)$.
Let $C^{\infty}_c(\mathbb{T}^d\times (0,\infty))$ be the space of infinitely differentiable functions with compact support on $\mathbb{T}^d\times (0,\infty)$.
For nonnegative integers $k$ and $p\in [1,\infty]$, we use $W^{k,p}(\mathbb{T}^d)$ to denote the usual Sobolev space on $\mathbb{T}^d$ which consists of all functions $\rho:\mathbb{T}^d\rightarrow \mathbb{R} $ such that for each multiindex $\alpha$ with $|\alpha|:=\alpha_1+\cdots+\alpha_d\leq k$, the weak derivative $D^{\alpha}\rho$ exists
and belongs to $L^p(\mathbb{T}^d)$.
For a positive integer $a$, let
$H^a(\mathbb{T}^d)=W^{a,2}(\mathbb{T}^d)$ be the usual Sobolev space of order $a$, $H^{-a}(\mathbb{T}^d)$ stands for the topological dual of $H^a(\mathbb{T}^d)$ whose norm is denoted by $\|\cdot\|_{H^{-a}(\mathbb{T}^d)}$.
Let the bracket $\langle\cdot,\cdot\rangle$ stand for the duality between $C^{\infty}(\mathbb{T}^d)$ and $C^{\infty}(\mathbb{T}^d)^*$.

For the sake of simplicity, for $1\leq p\leq \infty$, we denote by $L^p(\mathbb{T}^d)=:L^p(\mathbb{T}^d;\mathbb{R})$, $C([0,T])=:C([0,T];\mathbb{R})$,
$L^p([0,T]\times\mathbb{T}^d)=:L^p([0,T];L^p(\mathbb{T}^d;\mathbb{R}))$ and $L^p([0,T]\times\mathbb{T}^d;\mathbb{R}^d)=:L^p([0,T];L^p(\mathbb{T}^d;\mathbb{R}^d))$. In the following, we also use $\nabla \rho$ to denote the weak derivative of $\rho$.
In the context, we employ the notation $a\lesssim b$ for $a,b\in \mathbb{R}$ which stands for $a\leq Cb$ for some constant $C>0$ independent of any parameters. Moreover,
we adopt the letters $c, C$ to denote generic positive and finite constants
and if deemed necessary, their (in-)dependence on parameters or functions will be specified.

In the following kinetic formulation, we will frequently encounter derivatives of functions
$\psi\in C^{\infty}_c(\mathbb{T}^d\times \mathbb{R})$  evaluated at the point $\xi=\rho(x,t)$. Denote by $\nabla \psi(x,\rho(x,t)):=\nabla \psi(x,\xi)|_{\xi=\rho(x,t)}$.

\subsection{Definition of the noise}
In the present paper, we consider the following type of noise.
Let $\{B^k,W^k\}_{|k|\leq K}$ be independent $d$-dimensional real-valued Brownian motions defined on a probability space $(\Omega, \mathcal{F}, \{\mathcal{F}_t\}_{t\in [0,T]}, \mathbb{P})$. We use $\mathbb{E}$ to denote the expectation with respect to $\mathbb{P}$. Let $(B^k,W^k)_{|k|\leq K}$ take values in the space
\begin{align*}
  \mathbb{R}^{\infty}=\{(x_i)_{i\in \mathbb{N}}: x_i\in \mathbb{R}^d, \forall i\in \mathbb{N}\}
\end{align*}
equipped with the metric topology of coordinate-wise convergence. For every $K\in \mathbb{N}$, we define the spectral approximation $\xi^K$ of the noise $\xi$ as follows:
\begin{align}\label{s-39-1}
  d\xi^K:=\sum_{|k|\leq K} \Big(\sin (k \cdot x)dB^k_t+\cos (k\cdot x)dW^k_t\Big).
\end{align}
Clearly, $\xi^K$ is a special case of the noise defined in \cite{WWZ22}.
In the case of $\xi^K$, we have
\begin{align}\label{ll-1}
 N_K:=\sum_{|k|\leq K}1\leq cK^d,\quad M_K:=\sum_{|k|\leq K}k^2\leq cK^{d+2}.
\end{align}

\subsection{Kinetic solution of the SPDE}
We consider the following Stratonovich equation
 \begin{align}\label{s-1}
  \left\{
    \begin{array}{ll}
      \partial_t \rho^{\varepsilon,K}=\Delta \rho^{\varepsilon,K}-\nabla\cdot  (\rho^{\varepsilon,K} V\ast \rho^{\varepsilon,K})-\sqrt{\varepsilon}\nabla\cdot  (\sqrt{ \rho^{\varepsilon,K}} \circ \xi^K), &  \\
    \rho^{\varepsilon,K}(0)=\rho_0. &
    \end{array}
  \right.
 \end{align}


The equation (\ref{s-1})
  can be written as the following It\^{o} formulation
  \begin{align}\label{s-2}
  \left\{
    \begin{array}{ll}
      \partial_t \rho^{\varepsilon,K}=\Delta \rho^{\varepsilon,K}-\nabla\cdot  (\rho^{\varepsilon,K} V\ast \rho^{\varepsilon,K})-\sqrt{\varepsilon}\nabla\cdot  (\sqrt{\rho^{\varepsilon,K}} \xi^K)+\frac{\varepsilon N_K}{8}\nabla\cdot ({(\rho^{\varepsilon,K})}^{-1}\nabla \rho^{\varepsilon,K}), & \\
     \rho^{\varepsilon,K}(0)=\rho_0. &
    \end{array}
  \right.
 \end{align}
 In the sequel, we will choose $K$ to be $K(\varepsilon)$. \textbf{For the sake of simplicity, denote by $\rho^{\varepsilon}=:\rho^{\varepsilon,K(\varepsilon)}$}.

Now, we state the definition of stochastic renormalized kinetic solution of (\ref{s-2}). Let the finite entropy dissipation space defined by
\begin{align}\label{finite-entropy}
{\rm{Ent}}(\mathbb{T}^d):=\Big\{\rho\in L^1(\mathbb{T}^d): \rho\geq0\ a.e.\ {\rm{and}}\ \int_{\mathbb{T}^d}\rho\log\rho dx<\infty\Big\}.
\end{align}
Firstly, we need the definition of kinetic measure.
\begin{definition}\label{dfn-r-1}
  Let $(\Omega,\mathcal{F},\mathbb{P})$ be a probability space with a filtration $(\mathcal{F}_t)_{t\in [0,T]}$. A kinetic measure is a mapping $q$ from $\Omega$ to the space of nonnegative, Radon measures on $\mathbb{T}^d\times (0,\infty)\times [0,T]$ denoted by $\mathcal{M}^+(\mathbb{T}^d\times (0,\infty)\times [0,T])$ that satisfies
  \begin{align*}
    (\omega,t)\in \Omega\times[0,T]\rightarrow
    \int^t_0\int_{\mathbb{R}}\int_{\mathbb{T}^d}\psi(x,\xi)dq(x,s,\xi)
  \end{align*}
  is $\mathcal{F}_t-$predictable, for every $\psi\in C^{\infty}_c(\mathbb{T}^d\times (0,\infty))$.
\end{definition}

 \begin{definition}\label{dfn-2}
Let $\rho_0\in  {\rm{Ent}} (\mathbb{T}^d)$, $\varepsilon\in (0,1), K\in \mathbb{N}$. A nonnegative function $\rho^{\varepsilon}\in L^{\infty}(\Omega\times[0,T];L^1(\mathbb{T}^d))$ is called a stochastic renormalized kinetic solution of (\ref{s-2}) with initial data $\rho_0$ if $\rho^{\varepsilon}$ is almost surely continuous $L^1(\mathbb{T}^d)-$valued $\mathcal{F}_t-$predicatable function and satisfies the following properties.
\begin{enumerate}
  \item Preservation of mass: almost surely for almost every $t\in [0,T]$, $\|\rho^{\varepsilon}(\cdot,t)\|_{L^1(\mathbb{T}^d)}=\|\rho_0\|_{L^1(\mathbb{T}^d)}$.
  \item Regularity:
  $\nabla\sqrt{\rho}\in L^2(\Omega\times[0,T];L^2(\mathbb{T}^d;\mathbb{R}^d))$, $V\ast\rho\in L^2(\Omega\times[0,T];L^2(\mathbb{T}^d))$.

  \item Regularity of kinetic measure: for all nonnegative $\psi\in C^{\infty}_c(\mathbb{T}^d\times (0,\infty))$, it holds that $q^{\varepsilon}(\psi)\geq o^{\varepsilon}(\psi)$, $\mathbb{P}-$a.s., where
 $o^{\varepsilon}: \Omega\rightarrow \mathcal{M}^+(\mathbb{T}^d\times (0,\infty)\times [0,T])$ is defined by
\begin{align*}
  o^{\varepsilon}(\psi):=4\int^T_0\int_{\mathbb{T}^d}\int_{\mathbb{R}}\psi(x,\xi)
\delta_0(\xi-\rho^{\varepsilon}) \xi |\nabla\sqrt{\rho^{\varepsilon}}|^2d\xi dxdt.
\end{align*}

\item Vanishing at infinity: we have
     \begin{align}\label{k-27}
       \liminf_{M\rightarrow \infty}\mathbb{E}\Big[q^{\varepsilon}(\mathbb{T}^d\times [0,T]\times [M,M+1])\Big]=0.
     \end{align}
\item The equation: the kinetic function $\chi^{\varepsilon}(x,\xi,t):=I_{0<\xi<\rho^{\varepsilon}(x,t)}$ and the kinetic measure $q^{\varepsilon} $ satisfy that, for all
  $\psi\in C^{\infty}_c(\mathbb{T}^d\times (0,\infty))$, for almost every $t\in [0,T]$, $\mathbb{P}-$a.s.,
     \begin{align}\notag
      &\int_{\mathbb{R}}\int_{\mathbb{T}^d}\chi^{\varepsilon}(x,\xi,t)\psi(x,\xi)dxd\xi
      =\int_{\mathbb{R}}\int_{\mathbb{T}^d}\bar{\chi}(\rho_0(x),\xi)\psi(x,\xi)dxd\xi
      +\int^t_0\int_{\mathbb{R}}\int_{\mathbb{T}^d}\chi^{\varepsilon}\Delta_x\psi dxd\xi ds\\ \notag
      +&\int^t_0\int_{\mathbb{T}^d}\nabla(\psi(x,\rho^{\varepsilon}))\cdot\rho^{\varepsilon} V\ast \rho^{\varepsilon}dxds
      -\frac{\varepsilon N_K}{8}\int^t_0\int_{\mathbb{T}^d}(\rho^{\varepsilon})^{-1}\nabla \rho^{\varepsilon}
\cdot \nabla \psi(x,\rho^{\varepsilon})dxds\\ \notag
-&\int^t_0\int_{\mathbb{R}}\int_{\mathbb{T}^d}\partial_{\xi}\psi(x,\xi)dq^{\varepsilon}
+\frac{\varepsilon M_K}{2}\int^t_0\int_{\mathbb{T}^d} \rho^{\varepsilon} \partial_{\xi}\psi(x,\rho^{\varepsilon})dxds\\ \label{s-4}
-&\sqrt{\varepsilon}\int^t_0\int_{\mathbb{T}^d}\psi(x,\rho^{\varepsilon})
\nabla\cdot(\sqrt{\rho^{\varepsilon}}d\xi^{K}).
\end{align}
\end{enumerate}
\end{definition}

  The well-posedness of (\ref{s-2}) has been proved by \cite{WWZ22}, which reads as follows.
  \begin{theorem}\label{thm-14}
Let $T\in (0,\infty)$, $\varepsilon\in (0,1)$ and $K\in \mathbb{N}$. Assume that $V$ satisfies Assumptions (A1) and (A2). For any $\rho_0\in {\rm{Ent}} (\mathbb{T}^d)$, (\ref{s-2}) admits a unique stochastic renormalized kinetic solution in the sense of Definition \ref{dfn-2}.

  \end{theorem}

\subsection{Assumptions}
We impose the same assumptions on the kernel $V$ as \cite{WWZ22}.
Let $d\geq2$. For the interaction kernel $V:[0,T]\times\mathbb{T}^d\rightarrow\mathbb{R}^d$, we assume that
\begin{description}
\item[{\bf Assumption (A1)}] $V\in L^{p^{*}}([0,T];L^{p}(\mathbb{T}^d))$ with $\frac{d}{p}+\frac{2}{p^{*}}\leq 1$, $2\leq p^{*}<\infty$ and $d<p\leq \infty$,
\item[{\bf Assumption (A2)}] $\nabla\cdot V\in L^{q^{*}}([0,T];L^{q}(\mathbb{T}^d))$ with $ \frac{d}{2q}+\frac{1}{q^{*}}\leq 1$, $1\leq q^{*}<\infty$ and $d/2<q\leq \infty$.
\end{description}
Define
 \begin{align}
 p'=\frac{2p}{p-d},\  q'=\frac{2q}{2q-d}.
 \end{align}
Clearly, we have $\frac{d}{p}+\frac{2}{p'}= 1$ and $\frac{d}{2q}+\frac{1}{q'}=1$, which imply that $p^{*}\geq p'$ and $q^{*}\geq q'$.

\subsection{A priori estimates}\label{sec-10}

We need the following regularity of $\sqrt{h}V\ast h$.
\begin{lemma}\label{kernelestimate}
 Assume that $V\in L^p(\mathbb{T}^d)$, $h\in L^1(\mathbb{T}^d)$, $\nabla\sqrt{h}\in L^2(\mathbb{T}^d;\mathbb{R}^d)$, for some $p>2$, it gives
\begin{align}\label{r-10}
\|\sqrt{h}V\ast h\|_{L^2(\mathbb{T}^d)}
\leq c(d)\|\nabla\sqrt{h}\|^{\frac{d}{p}}_{L^{2}(\mathbb{T}^{d})}
\|V\|_{L^p(\mathbb{T}^d)}\|h\|^{\frac{3}{2}-\frac{d}{2p}}_{L^1(\mathbb{T}^d)}.
\end{align}

\end{lemma}

\begin{proof}
Let $s$ be the conjugate number of $p/2$, by H\"older and convolution Young inequalities, we have
\begin{align*}
\|\sqrt{h}V\ast h\|_{L^2(\mathbb{T}^d)}
\leq\|\sqrt{h}\|_{L^{2s}(\mathbb{T}^d)}\|V\|_{L^p(\mathbb{T}^d)}\|h\|_{L^1(\mathbb{T}^d)}.
\end{align*}
Owing to Gagliardo-Nirenberg interpolation inequality in \cite{BM18}, it follows that
\begin{align}\label{r-12}
  \|\sqrt{h}\|_{L^{2s}(\mathbb{T}^d)}\leq c(d)\|\nabla\sqrt{h}\|^{\frac{d}{2}(1-1/s)}_{L^{2}(\mathbb{T}^{d})}
\|\sqrt{h}\|^{1-\frac{d}{2}(1-1/s)}_{L^{2}(\mathbb{T}^{d})},
\end{align}
then,
\begin{align*}
\|\sqrt{h}V\ast h\|_{L^2(\mathbb{T}^d)}\leq&c(d)\|\nabla\sqrt{h}\|^{\frac{d}{2}(1-1/s)}_{L^{2}(\mathbb{T}^{d})}
\|\sqrt{h}\|^{1-\frac{d}{2}(1-1/s)}_{L^{2}(\mathbb{T}^{d})}\|V\|_{L^p(\mathbb{T}^d)}\|h\|_{L^1(\mathbb{T}^d)},
\end{align*}
 Note that $1-1/s=2/p$, we get the desired result (\ref{r-10}).
\end{proof}

We also need the product rule for weak derivatives, whose proof can be found in \cite{WWZ22}.
\begin{lemma}\label{lem-r-1}
\begin{description}
\item[(i)]Let $f\in L^1(\mathbb{T}^d)$ be nonnegative and the weak derivative $\nabla\sqrt{f}\in L^2(\mathbb{T}^d;\mathbb{R}^d)$, then the weak derivative $\nabla f=2\sqrt{f}\nabla\sqrt{f}$ for almost every $x\in \mathbb{T}^d$. Moreover, we have $\nabla f\in L^1(\mathbb{T}^d;\mathbb{R}^d)$.
  \item[(ii)] Let $f\in L^1(\mathbb{T}^d)$ and $g\in L^1(\mathbb{T}^d)$ be nonnegative functions, and the weak derivatives $\nabla\sqrt{f}\in L^2(\mathbb{T}^d;\mathbb{R}^d)$ and $\nabla\sqrt{g}\in L^2(\mathbb{T}^d;\mathbb{R}^d)$. Assume that $V$ satisfies Assumption (A1), then the weak derivative $\nabla\cdot(fV\ast g)=\nabla f\cdot V\ast g+fV\ast(\nabla g)$ holds for almost every $x\in \mathbb{T}^d$, where $V\ast(\nabla g):=\int_{\mathbb{T}^d}V(y)\cdot\nabla_x g(x-y)dy$.
  \item[(iii)] Let $f\in L^1(\mathbb{T}^d)$ and $g\in L^1(\mathbb{T}^d)$ be nonnegative functions, and the weak derivative $\nabla\sqrt{f}\in L^2(\mathbb{T}^d;\mathbb{R}^d)$. Assume that $V$ satisfies Assumptions (A1) and (A2), then the weak derivative $\nabla\cdot(fV\ast g)=\nabla f\cdot V\ast g+f (\nabla \cdot V)\ast g$ holds for almost every $x\in \mathbb{T}^d$, where $(\nabla \cdot V)\ast g:=\int_{\mathbb{T}^d}(\nabla\cdot V(x-y)) g(y)dy$.
\end{description}
\end{lemma}

\section{Analysis of the McKean-Vlasov  PDE }\label{sec-3}

For every $g\in L^2([0,T];L^2(\mathbb{T}^d;\mathbb{R}^d))$, we consider
\begin{eqnarray}\label{skeleton}
\left\{
  \begin{array}{ll}
   d\rho(t)=\Delta\rho dt-\nabla\cdot (\rho V\ast\rho)dt-\nabla\cdot (\sqrt{\rho}g)dt , &  \\
    \rho(\cdot,0)=\rho_0, &
  \end{array}
\right.
\end{eqnarray}
which is called the skeleton equation of (\ref{s-1}).

According to the results of \cite{FG23} and \cite{WWZ22}, under Assumptions (A1) and (A2) on $V$, the standard $L^p(\mathbb{T}^d)-$theory ($p\geq 2$) are not available for (\ref{skeleton}). To compensate it, we derive
 an entropy dissipation estimate under Assumption (A1), where the initial data is restricted to have finite entropy:
\begin{equation}\label{t-16}
{\rm{Ent}}(\mathbb{T}^d):=\Big\{\rho\in L^1(\mathbb{T}^d): \rho\geq0\ a.e.\ {\rm{and}}\ \int_{\mathbb{T}^d}\rho\log\rho dx<\infty\Big\}.
\end{equation}

In the sequel, we will find that the equation (\ref{skeleton}) can be equipped with renormalized kinetic solution and weak solution at the same time. The former is employed to prove uniqueness, and the later is used to show existence. Under this circumstance, the proof of the equivalence between these two concepts
is necessary. To gain a deeper understanding of the regularity on the kernel $V$ required for different concepts, we will list them separately.

\subsection{Definition of renormalized kinetic solution}
Define the kinetic function $\bar{\chi}:\mathbb{R}^2\rightarrow \mathbb{R}$ by
\begin{align*}
  \bar{\chi}(s,\xi)=I_{0<\xi<s}.
\end{align*}
Suppose that $\rho$ is a solution of (\ref{skeleton}),
the kinetic function $\chi$ of $\rho$ is defined by $\chi(x,t,\xi)=\bar{\chi}(\rho(x,t),\xi)=I_{0<\xi<\rho(x,t)}$ for every $(x,\xi,t)\in \mathbb{T}^d\times \mathbb{R}\times [0,T]$. The following identities hold in distributional sense,
\begin{align*}
  \nabla  \chi(x,\xi,t)=\delta_0(\xi-\rho(x,t))\nabla \rho(x,t),\quad \partial_{\xi}\chi(x,\xi,t)=\delta_0(\xi)-\delta_0(\xi-\rho(x,t)).
\end{align*}
As a result, formally, we deduce from (\ref{skeleton}) that
\begin{align}\label{eqq-22}
\partial_t\chi= \Delta_x\chi-(\partial_{\xi}\sqrt{\xi})g(x,t)\cdot\nabla \chi
+(\nabla g(x,t))\sqrt{\xi}\partial_{\xi}\chi
 -\delta_0(\xi-\rho)\nabla\cdot (\rho V\ast \rho)+\partial_{\xi}q,
\end{align}
where $\chi(\cdot,0)=\bar{\chi}(\rho_0)$ and $q$ is the parabolic defect measure on $\mathbb{T}^d\times \mathbb{R}\times [0,T]$ given by
\begin{align}\label{k-10}
q=4\delta_0(\xi-\rho)\xi|\nabla\sqrt{\rho}|^2.
\end{align}
Moreover, (\ref{eqq-22}) can be rewritten in the following form
\begin{align}\label{eqq-23}
\partial_t\chi=\Delta_x\chi-\partial_{\xi}(g(x,t)\cdot\sqrt{\xi}\nabla \chi)
+\nabla\cdot (g(x,t)\sqrt{\xi}\partial_{\xi}\chi)-\delta_0(\xi-\rho)\nabla\cdot (\rho V\ast \rho)+\partial_{\xi}q,
\end{align}
on $\mathbb{T}^d\times \mathbb{R}\times (0,T)$.


Now, we state the definition of renormalized kinetic solution rigorously based on  (\ref{eqq-23}).
\begin{definition}\label{dfn-3}
 Let $\rho_0\in {\rm{Ent}} (\mathbb{T}^d)$. A nonnegative function $\rho\in L^{\infty}([0,T];L^1(\mathbb{T}^d))$ is
a renormalized kinetic solution to (\ref{skeleton}) with initial value $\rho_0$ if $\rho$ satisfies
\begin{enumerate}
  \item $\nabla\sqrt{\rho}\in L^2([0,T];L^2(\mathbb{T}^d;\mathbb{R}^d))$, $V\ast\rho\in L^2([0,T];L^2(\mathbb{T}^d))$.
  \item There exists a subset $\mathcal{P}\subseteq (0,T]$ of Lebesgue measure zero such that for every $t\in [0, T]\setminus \mathcal{P}$, for any $\psi\in C_c^{\infty}(\mathbb{T}^d\times (0,\infty))$ with $\psi(x,0)=0$, the kinetic function $\chi$ of $\rho$ and the parabolic defect measure $q$ satisfy that
\begin{align}\notag
\int_{\mathbb{T}^d}\int_{\mathbb{R}}\chi(x,\xi,t)\psi(x,\xi)dxd\xi
=&\int_0^t\int_{\mathbb{R}}\int_{\mathbb{T}^d}\chi\Delta_x\psi dxd\xi ds-\int_0^t\int_{\mathbb{R}}\int_{\mathbb{T}^d}q\partial_{\xi}\psi dxd\xi ds\\ \notag
+&\int_0^t\int_{\mathbb{T}^d} 2\rho \nabla \sqrt{\rho}\cdot g(x,s)(\partial_{\xi}\psi)(x,\rho)dxds
+\int_0^t\int_{\mathbb{T}^d}\sqrt{\rho}g(x,s)\cdot(\nabla \psi)(x,\rho)dxds\\
\label{eqq-20}
-& \int_0^t\int_{\mathbb{T}^d}\psi(x,\rho)\nabla\cdot(\rho V\ast \rho) dxds
+\int_{\mathbb{R}}\int_{\mathbb{T}^d}\bar{\chi}(\rho_0(x),\xi)\psi(x,\xi)dxd\xi.
\end{align}
\end{enumerate}
\end{definition}
We point out that the regularity (1) in Definition \ref{dfn-3} is not strong enough to guarantee the global integrability of the parabolic defect measure $q$, so
the test function $\psi$ in Definition \ref{dfn-3} is required to have compact support with respect to $x$ and $\xi$.

\begin{remark}\label{remark-1}
When $V$ satisfies Assumption (A1), the kernel term $\int_0^t\int_{\mathbb{T}^d}\psi(x,\rho)\nabla\cdot (\rho V\ast \rho)dxds$ is well-defined. In fact, by using (ii) in Lemma \ref{lem-r-1}, we have
\begin{align}\label{t-17-1}
	\int_{0}^{T}\int_{\mathbb{T}^{d}}|\nabla\cdot\left(\rho V(t)\ast \rho\right)|dxdt\le\int_{0}^{T}\|\nabla \rho\cdot V(t)\ast \rho\|_{L^1(\mathbb{T}^d)}dt+\int_{0}^{T}\| \rho (V(t)\ast \nabla \rho)\|_{L^1(\mathbb{T}^d)}dt.
	\end{align}
Since the skeleton equation (\ref{skeleton}) is conservative and by the nonnegativity of the solution, we have for any $t\in [0,T]$,
\begin{align}\label{r-11}
 \|\rho(t)\|_{L^1(\mathbb{T}^d)}=\|\rho_0\|_{L^1(\mathbb{T}^d)}.
\end{align}
With the aid of (i) in Lemma \ref{lem-r-1}, H\"{o}lder inequality, (\ref{r-10}) and  (\ref{r-11}), we have
\begin{align}
\int_{\mathbb{T}^d}|\nabla\rho\cdot V\ast\rho| dx
&\leq 2\|\nabla\sqrt{\rho}\|_{L^2(\mathbb{T}^d)}\|\sqrt{\rho}\cdot V\ast\rho\|_{L^2(\mathbb{T}^d)}\notag\\
&\leq  c(d)\|\rho_0\|^{3/2-\frac{d}{2p}}_{L^{1}(\mathbb{T}^{d})}\|\nabla\sqrt{\rho}\|^{1+\frac{d}{p}}_{L^{2}(\mathbb{T}^{d})}
\|V\|_{L^p(\mathbb{T}^d)}.\label{qq-1}
\end{align}


Since $p>d$, applying Young inequality to (\ref{qq-1}) with $(\frac{2p}{ p+d},\frac{2p}{p-d})$, it gives
\begin{align}\notag
\int_{0}^{T}\|\nabla \rho\cdot V(t)\ast \rho\|_{L^1(\mathbb{T}^d)}dt
\leq &
\int^T_0\|\nabla\sqrt{\rho}\|_{L^2(\mathbb{T}^d)}^2dt
+C(\|\rho_0\|_{L^1(\mathbb{T}^d)},d)\int^T_0\|V\|_{L^p(\mathbb{T}^d)}^{\frac{2p}{p-d}}dt\\
\label{ent-1}
\leq&\int^T_0\|\nabla\sqrt{\rho}\|_{L^2(\mathbb{T}^d)}^2dt
+C(T,p,d,\|\rho_0\|_{L^1(\mathbb{T}^d)})\int^T_0\|V\|_{L^p(\mathbb{T}^d)}^{p^{*}}dt.
\end{align}
Let $r$ be the conjugate number of $p$,
by H\"{o}lder and convolution Young inequalities, (i) in Lemma \ref{lem-r-1}, (\ref{r-12}) and (\ref{r-11}), it gives
	\begin{align}\notag
	\int_{0}^{T}\|\rho(V(t)\ast \nabla \rho)\|_{L^1(\mathbb{T}^d)}dt\lesssim&\|\rho_0\|^{3/2-\frac{d}{2p}}_{L^1(\mathbb{T}^d)}
\int_{0}^{T}\|\nabla\sqrt{\rho}\|^{1+\frac{d}{p}}_{L^2(\mathbb{T}^d)}\|V(t)\|_{L^p(\mathbb{T}^d)}dt\\
\label{k-26}
\lesssim& \|\rho_0\|^{3/2-\frac{d}{2p}}_{L^1(\mathbb{T}^d)}\Big(\int_{0}^{T}\|\nabla\sqrt{\rho}\|^{2}_{L^2(\mathbb{T}^d)} dt\Big)^{\frac{d+p}{2p}}\|V\|_{L^{p^*}([0,T];L^p(\mathbb{T}^d))}.	
\end{align}
It follows from (\ref{ent-1}) and (\ref{k-26}) that (\ref{t-17-1}) holds.
We emphasize that $p>d$ in Assumption (A1) is necessary to apply the Young inequality to (\ref{qq-1}).

\end{remark}
\subsection{Uniqueness of renormalized kinetic solution}
Before giving the result of $L^1$-uniqueness, we need the following technical lemma proved by \cite[Lemma 7]{FG23}.
\begin{lemma}\label{lem-4}
Let $(X,S,\mu)$ be a measure space. For $K\in\mathbb{N}$, let $\{f_k:X\rightarrow\mathbb{R}\}_{k\in\{1,2,\dots,K\}}\subseteq L^1(X)$, and $\{B_{n,k}\subseteq X\}_{n\in\mathbb{N}}\subseteq  S$ be disjoint subsets for every $k\in\{1,2,...,K\}$. Then,
\begin{equation*}
\liminf_{n\rightarrow\infty}\Big(n\sum_{k=1}^K\int_{B_{n,k}}|f_k|d\mu\Big)=0.
\end{equation*}
\end{lemma}

The uniqueness of (\ref{skeleton}) will be proved by the doubling variables method, which has been applied to stochastic conservation laws successfully by \cite{DV10,DWZZ20}.
\begin{proposition}\label{L1uniq}
Assume that $V$ satisfies Assumptions (A1) and (A2).
 Let $g\in L^2([0,T];L^2(\mathbb{T}^d;\mathbb{R}^d))$ and $\rho_0^1, \rho_0^2\in {\rm{Ent}} (\mathbb{T}^d)$. Assume that $\rho^1, \rho^2\in L^{\infty}([0,T];L^1(\mathbb{T}^d))$  are renormalized kinetic solutions of (\ref{skeleton}) in the sense of Definition \ref{dfn-3} with control $g$ and initial data $\rho_0^1, \rho_0^2$. If $\rho_0^1=\rho_0^2$ for almost every $x\in\mathbb{T}^d$, then
\begin{equation}\label{qq-31-1}
\rho^1(t,x)=\rho^2(t,x),\ \ \text{for almost every }(t,x)\in[0,T]\times\mathbb{T}^d.
\end{equation}

\end{proposition}

\begin{proof}
  Let $\chi^1,\chi^2$ be kinetic functions of $\rho^1$ and $\rho^2$ with the corresponding parabolic defect measures $q^1$ and $q^2$.
Denote by $\mathcal{P}^1,\mathcal{P}^2\subseteq(0,T]$ the zero sets appearing in Definition \ref{dfn-3}, and define $\mathcal{P}:=\mathcal{P}^1\cup\mathcal{P}^2$.
Let $\kappa^s\in C^{\infty}_c(\mathbb{T}^d;[0,\infty))$ and $\kappa^v\in C_c^{\infty}(\mathbb{R};[0,\infty))$ be two standard mollifiers. For any $ \varepsilon,\delta\in(0,1)$, let $\kappa^{\varepsilon,\delta}:(\mathbb{T}^d)^2\times\mathbb{R}^2\rightarrow[0,\infty)$ be defined by
\begin{equation*}
\kappa^{\varepsilon,\delta}(x,y,\xi,\eta)
=\Big(\varepsilon^{-d}\kappa^s(\frac{x-y}{\varepsilon})\Big)\Big(\delta^{-1}\kappa^{v}(\frac{\xi-\eta}{\delta})\Big).
\end{equation*}
For any $M>0$, let $\zeta^M:\mathbb{R}\rightarrow[0,1]$ be a  continuous piecewise linear function satisfying
\begin{eqnarray}
\zeta^M(\xi)=\left\{
  \begin{array}{ll}
   0
  & \ {\rm{if}}\ |\xi|\leq\frac{1}{M},\\
   M\Big(|\xi|-\frac{1}{M}\Big) & \ {\rm{if}}\ \frac{1}{M}\leq|\xi|\leq\frac{2}{M},\\
   1 & \ {\rm{if}} \ \frac{2}{M}\leq|\xi|\leq M,\\
   M+1-|\xi| & \ {\rm{if}}\ M\leq|\xi|\leq M+1,\\
   0 & \ {\rm{if}}\ |\xi|\geq M+1.
  \end{array}
\right.
\end{eqnarray}
For every $i\in\{1,2\}$ and $\varepsilon,\delta\in(0,1)$, let $\chi^{i,\varepsilon,\delta}:\mathbb{T}^d\times\mathbb{R}\times[0,T]\rightarrow\mathbb{R}$ be defined by
\begin{align*}
\chi_t^{i,\varepsilon,\delta}(y,\eta)=\int_{\mathbb{R}}\int_{\mathbb{T}^d}\chi_t^i(x,\xi)\kappa^{\varepsilon,\delta}(x,y,\xi,\eta)dxd\xi.
\end{align*}
Let $t\in(0,T]/\mathcal{P}$ and $M>0$, by the properties of the kinetic function, we have
\begin{align}\notag
&\int_{\mathbb{R}}\int_{\mathbb{T}^d}|\chi_t^1-\chi_t^2|^2\zeta^M(\eta)dyd\eta\\
\notag
=&\lim_{\varepsilon,\delta\rightarrow0}\int_{\mathbb{R}}\int_{\mathbb{T}^d}\Big|\chi_t^{1,\varepsilon,\delta}-\chi_t^{2,\varepsilon,\delta}\Big|^2\zeta^M(\eta)dyd\eta\\
\label{qq-28}
=&\lim_{\varepsilon,\delta\rightarrow0}\int_{\mathbb{R}}\int_{\mathbb{T}^d}\Big(\chi_t^{1,\varepsilon,\delta}+\chi_t^{2,\varepsilon,\delta}-2\chi_t^{1,\varepsilon,\delta}\chi_t^{2,\varepsilon,\delta}\Big)\zeta^M(\eta)dyd\eta.
\end{align}
Let $\varepsilon,\delta\in(0,1)$, we deduce from Definition \ref{dfn-3} that for every $i\in\{1,2\}$, as distributions on $\mathbb{T}^d\times \mathbb{R}\times (0,T)$,
\begin{align}\notag
  \partial_t \chi^{i,\varepsilon,\delta}_t(y,\eta)=&\int_{\mathbb{R}}\int_{\mathbb{T}^d}\chi_t^i\Delta_x\kappa^{\varepsilon,\delta}(x,y,\xi,\eta)dxd\xi
  -\int_{\mathbb{R}}\int_{\mathbb{T}^d}q^i_t\partial_{\xi}\kappa^{\varepsilon,\delta}(x,y,\xi,\eta)dxd\xi\\
  \notag
  +&2\int_{\mathbb{T}^d}\rho^ig(x,t)\cdot(\nabla_x \sqrt{\rho^i})(\partial_{\xi}\kappa^{\varepsilon,\delta})(x,y,\rho^i,\eta)dx\\
 \label{k-11}
  +&\int_{\mathbb{T}^d}\sqrt{\rho^i}g(x,t)\cdot(\nabla_x \kappa^{\varepsilon,\delta})(x,y,\rho^i,\eta)dx
 -\int_{\mathbb{T}^d}\nabla_x\cdot(\rho^i(V\ast \rho^i)) \kappa^{\varepsilon,\delta}(x,y,\rho^i,\eta)dx.
\end{align}


For $i=1,2$, let $\bar{\kappa}^{\varepsilon,\delta}_{i}:(\mathbb{T}^d)^2\times \mathbb{R}\times [0,T]\rightarrow [0,\infty)$ be defined by
\begin{align*}
  \bar{\kappa}^{\varepsilon,\delta}_{i,t}(x,y,\eta):=\kappa^{\varepsilon,\delta}(x,y,\rho^i(x,t),\eta).
\end{align*}

 Compared with the full expressions in \cite[(28)]{FG23}, it is sufficient to handle the additional terms associated with the nonlocal interaction kernel.
  Concretely, let $\mathbb{I}^{\varepsilon,\delta,M}_{t,ker}$ be the term generated by the kernel term in the expression of \cite[(28)]{FG23}, we have

	\begin{align*}
	\mathbb{I}^{\varepsilon,\delta,M}_{t,ker}
	=&\int_{0}^{t}\int_{\mathbb{R}}\int_{\left(\mathbb{T}^{d}\right)^{2}}\bar{\kappa}_{r,1}^{\varepsilon,\delta}\nabla\cdot(\rho^1V(r)\ast\rho^1)\left(2\chi_{r,2}^{\varepsilon,\delta}-1\right)\zeta_{M}(\eta)\notag\\
	&+\int_{0}^{t}\int_{\mathbb{R}}\int_{\left(\mathbb{T}^{d}\right)^{2}}\bar{\kappa}_{r,2}^{\varepsilon,\delta}\nabla\cdot(\rho^2V(r)\ast\rho^2)\left(2\chi_{r,1}^{\varepsilon,\delta}-1\right)\zeta_{M}(\eta).
	\end{align*}
	By the same procedure as in the proof of \cite[Theorem 4.2]{WWZ22}, an analysis of the passing to the limits shows that
	\begin{align}\notag &\lim_{M\to\infty}\left(\lim_{\beta\to0}\left(\lim_{\delta\to0}\left(\lim_{\varepsilon\to0}\mathbb{I}^{\varepsilon,\delta,M}_{t,ker}\right)\right)\right)\\
	\notag 
=&\int_{0}^{t}\int_{\mathbb{T}^{d}}\left(\mathbf{1}_{\left\{\rho^{1}=\rho^{2}\right\}}+2\mathbf{1}_{\left\{\rho^{1}<\rho^{2}\right\}}-1\right)\left(\nabla\cdot(\rho^1V(r)\ast\rho^1)-\nabla\cdot(\rho^2V(r)\ast\rho^2)\right)\\
	\notag &+\int_{0}^{t}\int_{\mathbb{T}^{d}}\left(\mathbf{1}_{\left\{\rho^{1}=\rho^{2}\right\}}+2\mathbf{1}_{\left\{\rho^{1}<\rho^{2}\right\}}-1+\mathbf{1}_{\left\{\rho^{1}=\rho^{2}\right\}}+2\mathbf{1}_{\left\{\rho^{2}<\rho^{1}\right\}}-1\right)
	\nabla\cdot(\rho^2V(r)\ast\rho^2)\\
	\label{eq-4.23}
	=:& \tilde{J}_1+\tilde{J}_2.
	\end{align}
	Thanks to $\mathbf{1}_{\left\{\rho^{2}<\rho^{1}\right\}} =1-\mathbf{1}_{\left\{\rho^{1}=\rho^{2}\right\}}-\mathbf{1}_{\left\{\rho^{1}<\rho^{2}\right\}}$, this yields $\tilde{J}_2=0$.

	Regarding $\tilde{J}_1$, with the help of the identity $\sgn(\rho^2-\rho^1)=\mathbf{1}_{\left\{\rho^{1}=\rho^{2}\right\}}+2\mathbf{1}_{\left\{\rho^{1}<\rho^{2}\right\}}-1$, we can divide $\tilde{J}_1$ into two parts
	\begin{align*}
	\tilde{J}_1=&\int_{0}^{t}\int_{\mathbb{T}^d}\sgn(\rho^2-\rho^1)\nabla\cdot\left((\rho^1-\rho^2)V(r)\ast\rho^1\right)\\
	&+\int_{0}^{t}\int_{\mathbb{T}^d}\sgn(\rho^2-\rho^1)\nabla\cdot\left(\rho^2V(r)\ast(\rho^1-\rho^2)\right)\\
	=:&\tilde{J}_{11}+\tilde{J}_{12},
	\end{align*}
	Let $\sgn^{\delta}:=(\sgn\ast\kappa_{1}^{\delta})$ for every $\delta\in(0,1)$. The uniform boundedness of $(\delta\kappa_{1}^{\delta})$ in $\delta\in(0,\beta/4)$ implies that there exists $c \in(0,\infty)$ independent of $\delta$ but depending on the convolution kernel such that for all $\delta\in(0,\beta/4)$,
	\begin{align}
	\tilde{J}_{11} =&\lim_{\delta\to0}\int_{0}^{t}\int_{\mathbb{T}^d}\sgn^{\delta}(\rho^2-\rho^1)\nabla\cdot\left((\rho^1-\rho^2)V(r)\ast\rho^1\right)\notag\\
	=&-\lim_{\delta\to0}\int_{0}^{t}\int_{\mathbb{T}^d}(\sgn^{\delta})'(\rho^2-\rho^1)(\rho^1-\rho^2)\nabla(\rho^2-\rho^1)\cdot V(r)\ast\rho^1\notag\\
	\leq&c\lim_{\delta\to0}\int_{0}^{t}\int_{\mathbb{T}^d}\mathbf{1}_{\left\{0<|\rho^{1}-\rho^{2}|<\delta\right\}}\nabla(\rho^2-\rho^1)\cdot V(r)\ast\rho^1.
	\end{align}
	Furthermore, since $\nabla(\rho^2-\rho^1)\cdot V\ast\rho^1\in L^1([0,T]\times \mathbb{T}^d)$, the dominated convergence theorem shows that $\tilde{J}_{11}=0$.
	
	Regarding the term $\tilde{J}_{12}$,
	using the chain rule, thanks to (\ref{r-10}) and Remark \ref{remark-1}, for every $t\in[0,T]$,
	\begin{align}\label{eq-4.28}
	\tilde{J}_{12}\le&C(\|\rho^2_0\|_{L^1(\mathbb{T}^d)},p,d)\int_{0}^{t}\|V(r)\|_{L^{p}(\mathbb{T}^d)}\|\nabla\sqrt{\rho^2}\|^{1+\frac{d}{p}}_{L^2(\mathbb{T}^d)}\|\rho^1-\rho^2\|_{L^{1}(\mathbb{T}^d)}\notag\\
	&+C(\|\rho^2_0\|_{L^1(\mathbb{T}^d)},q,d)\int_{0}^{t}\|\nabla\cdot V(r)\|_{L^{q}(\mathbb{T}^d)}\|\nabla\sqrt{\rho^2}\|^{\frac{d}{q}}_{L^{2}(\mathbb{T}^{d})}\|\rho^1-\rho^2\|_{L^{1}(\mathbb{T}^d)}.
	\end{align}
	Combining \eqref{eq-4.23}-\eqref{eq-4.28}, we conclude that  for every $t\in[0,T]$,
	 \begin{align*}
  \lim_{M\rightarrow\infty}\lim_{\delta\rightarrow 0}( \lim_{\varepsilon\rightarrow 0}\mathbb{I}^{\delta,M}_{t,ker})
  \leq&c(\|\rho^2_0\|_{L^{1}(\mathbb{T}^{d})},d)\int_0^t\|\nabla\sqrt{\rho^2}\|^{1+\frac{d}{p}}_{L^{2}(\mathbb{T}^{d})}
\|V\|_{L^p(\mathbb{T}^d)}\|\rho^1-\rho^2\|_{L^1(\mathbb{T}^d)}ds\\
+&C(\|\rho^2_0\|_{L^1(\mathbb{T}^d)},d)\int_0^t\|\nabla\sqrt{\rho^2}\|_{L^2(\mathbb{T}^d)}^{d/q}\|\nabla\cdot V\|_{L^{q}(\mathbb{T}^d)}
\|\rho^1-\rho^2\|_{L^1(\mathbb{T}^d)}ds,
  \end{align*}


Based on the $L^1(\mathbb{T}^d)-$contraction principle established in \cite[Theorem 8]{FG23},
we conclude that
\begin{align}\notag
\|\rho^1(t)-\rho^2(t)\|_{L^1(\mathbb{T}^d)}&\leq\|\rho^1_0-\rho^2_0\|_{L^1(\mathbb{T}^d)}
+C(\|\rho^2_0\|_{L^1(\mathbb{T}^d)},d)\int_0^t\Big(\|\nabla\sqrt{\rho^2}\|_{L^2(\mathbb{T}^d)}^{1+d/p}\|V\|_{L^p(\mathbb{T}^d)}\\
\label{r-13}
&+\|\nabla\sqrt{\rho^2}\|_{L^2(\mathbb{T}^d)}^{d/q}\|\nabla\cdot V\|_{L^q(\mathbb{T}^d)}\Big)\|\rho^1-\rho^2\|_{L^1(\mathbb{T}^d)}ds.
\end{align}
Since $p>d$ and $q>d/2$, we have $1+\frac{d}{p}<2$ and $d/q<2$. Applying Gronwall inequality to (\ref{r-13}) and by H\"{o}lder inequality, we get the desired result (\ref{qq-31-1}).
\end{proof}

By a similar method as in the proof of \cite[Proposition 9]{FG23}, we get the following result.
\begin{proposition}\label{prp-5}
For any $\rho_0\in {\rm{Ent}} (\mathbb{T}^d)$, let $\rho\in L^{\infty}([0, T];L^1(\mathbb{T}^d))$ be a renormalized kinetic solution of (\ref{skeleton}) in the sense of Definition \ref{dfn-3} with initial data $\rho_0$. Assume that Assumptions (A1) and (A2) are in force, then $\rho$ has a representative in $ C([0, T];L^1(\mathbb{T}^d))$.
\end{proposition}

\subsection{Equivalence of renormalized kinetic solutions and weak solutions}
Firstly, we give the definition of weak solution to (\ref{skeleton}). Recall that ${\rm{Ent}} (\mathbb{T}^d)$ is defined by (\ref{t-16}).
%

%
%

\begin{definition}\label{dfn-1}
Let $g\in L^2(\mathbb{T}^d\times[0,T];\mathbb{R}^d)$, and $\rho_0\in {\rm{Ent}} (\mathbb{T}^d)$. A nonnegative function $\rho\in L^{\infty}([0,T];L^1(\mathbb{T}^d))$ is called a weak solution of (\ref{skeleton}) with initial data $\rho_0$ if $\rho $ satisfies the following properties.
\begin{enumerate}
  \item
\begin{align}\label{rrr-17}
  \nabla\sqrt{\rho}\in L^2([0,T];L^2(\mathbb{T}^d;\mathbb{R}^d)),\ \ \rho V\ast\rho\in L^1([0,T];L^1(\mathbb{T}^d)).
\end{align}
  \item There exists a subset $\mathcal{P}\subseteq (0,T]$ of Lebesgue measure zero such that, for every $t\in [0,T]\setminus \mathcal{P}$ and
  for every $\psi\in C^{\infty}(\mathbb{T}^d)$,
\begin{align}\notag
  \int_{\mathbb{T}^d}\rho(x,t)\psi(x)dx=&-\int^t_0\int_{\mathbb{T}^d}2\sqrt{\rho}\nabla\sqrt{\rho}\cdot \nabla\psi dxds+\int^t_0\int_{\mathbb{T}^d}\sqrt{\rho}g\cdot \nabla \psi dxds\\
\label{eqq-19}
& +\int^t_0\int_{\mathbb{T}^d}\rho V\ast \rho\nabla\psi(x) dxds+\int_{\mathbb{T}^d}\rho_0(x)\psi(x)dx.
\end{align}
\end{enumerate}
\end{definition}
\begin{theorem}\label{thm-1}
Let $\rho_0 \in \operatorname{Ent}\left(\mathbb{T}^{d}\right)$  and  $g\in L^2([0,T];L^2(\mathbb{T}^d;\mathbb{R}^d))$. If
 $ \rho $ is a renormalized kinetic solution of (\ref{skeleton}) and an extra regularity holds that
  \begin{align}\label{cond-1}
 	\rho V\ast\rho\in L^1([0,T];L^1(\mathbb{T}^d;\mathbb{R}^d)),
 \end{align}
then $ \rho $ is a weak solution of (\ref{skeleton}). On the other hand, if $\rho$ is a weak solution of (\ref{skeleton}) and
\begin{align}\label{cond-2}
	\rho|V\ast\rho|^2\in L^1([0,T];L^1(\mathbb{T}^d)),\quad V\ast\rho\in L^2([0,T];L^2(\mathbb{T}^d;\mathbb{R}^d)),
\end{align}
then $\rho $ is a renormalized kinetic solution of (\ref{skeleton}).
\end{theorem}
\begin{proof}
(Renormalized kinetic solutions are weak solutions.) Let $\rho$  be a renormalized kinetic solution with control $g$ and initial $ \rho_0$. For every $\psi\in \mathrm{C}^{\infty}\left(\mathbb{T}^{d}\right)$, taking $\psi$ as a test function in (\ref{eqq-20}), it follows from the definition of the kinetic function $\chi$ that $\rho$ solves (\ref{skeleton}) in the weak sense of (\ref{eqq-19}) as well.

(Weak solutions are renormalized kinetic solutions.)  Let $\rho $ be a weak solution of (\ref{skeleton}). Testing (\ref{eqq-19}) with 1 and by the nonnegative of $\rho$, it gives that for almost every $t \in[0, T]$,
$$
\left\|\rho(t)\right\|_{L^1\left(\mathbb{T}^d\right)}=\left\|\rho_0\right\|_{L^1\left(\mathbb{T}^d\right)} .
$$
This, together with (\ref{cond-2}), implies that regularity properties (1) in Definition \ref{dfn-3} hold.

Let $q=4\delta_{0}(\xi-\rho)\xi|\nabla \sqrt{\rho}|^{2}$ denote the parabolic defect measure, it remains to show that $\rho$ satisfies \eqref{eqq-20}.
Let $\chi: \mathbb{T}^d\times \mathbb{R}\times [0,T]\rightarrow \mathbb{R}$ be the kinetic function of $ \rho $. For every $\delta \in (0,1)$, let $\kappa^{\delta}: \mathbb{T}^d\rightarrow \mathbb{R}$ be a standard symmetric convolution kernel of scale $\delta\in (0,1)$ and let $ \rho ^{\delta}: \mathbb{T}^d\times [0,T]\rightarrow \mathbb{R}$ be defined by $ \rho ^{\delta}=( \rho \ast\kappa^{\delta})$.  It follows from \eqref{eqq-19} and the definition of the convolution kernel that, as distributions on $\mathbb{T}^d\times (0,T)$, for every $\delta\in (0,1)$,
\begin{align*}
  \partial_t  \rho ^{\delta}
  =&\int_{\mathbb{T}^d}2\sqrt{ \rho (y,t)}\nabla\sqrt{ \rho (y,t)}\cdot \nabla_x\kappa^{\delta}(y-x)dy\\
  &
  -\int_{\mathbb{T}^d} \sqrt{ \rho (y,t)}g (y,t)\cdot\nabla_x\kappa^{\delta}(y-x)dy
  -\int_{\mathbb{T}^d}  \rho (y,t)V\ast \rho (y,t)\cdot \nabla_x\kappa^{\delta}(y-x)dy.
\end{align*}
Let $\mathcal{P}\subset[0,T]$ be the null set in Definition \ref{dfn-1}. Let $\psi\in C^{\infty}(\mathbb{T}^d\times \mathbb{R})$ and let $\Psi: \mathbb{T}^d\times \mathbb{R}\rightarrow \mathbb{R}$ be defined by $\Psi(x,\zeta)=\int^{\zeta}_0 \psi(x,\zeta')d\zeta'$. Then, for every $t\in [0,T]\setminus \mathcal{P}$,
\begin{align*}
  \int_{\mathbb{T}^d}\Psi(x, \rho ^{\delta}(x,r))dx|^t_{r=0}
  =& \int^t_0\int_{\mathbb{T}^d}\partial_t \rho ^{\delta}(x,r)\psi(x, \rho ^{\delta}(x,r))dxdr
  =:I^{\delta}_1+I^{\delta}_2+I^{\delta}_3,
\end{align*}
where
\begin{align*}
  I^{\delta}_1 &= 2\int^t_0\int_{(\mathbb{T}^d)^2}\nabla\sqrt{ \rho (y,r)}\sqrt{ \rho (y,r)}\cdot \nabla_x\kappa^{\delta}(y-x)\psi(x, \rho ^{\delta}(x,r))dydxdr,\\
  I^{\delta}_2 & =- \int^t_0\int_{(\mathbb{T}^d)^2}\sqrt{\rho(y,r)}g (y,r)\cdot\nabla_x\kappa^{\delta}(y-x)
  \psi(x, \rho ^{\delta}(x,r))dydxdr ,\\
  I^{\delta}_3 & = -\int^t_0\int_{(\mathbb{T}^d)^2} \rho (y,r)V\ast \rho (y,r)\cdot\nabla_x\kappa^{\delta}(y-x) \psi(x, \rho ^{\delta}(x,r))dydxdr.
\end{align*}
Proceeding the same process as in \cite[Theorem 14]{FG23}, we obtain that
\begin{align}\notag
\lim_{\delta\rightarrow 0}(I^{\delta}_1+I^{\delta}_2)
=& \int_0^t\int_{\mathbb{T}^d}\int_{\mathbb{R}}\chi(x,\xi,r)\Delta_x \psi(x,\xi)dxd\xi
dr
-  \int_0^t\int_{\mathbb{T}^d}\int_{\mathbb{R}}q(x,\xi,r)(\partial_{\xi} \psi)(x,\xi)dxd\xi dr\\
\label{I1 2}
&+2\int_0^t\int_{\mathbb{T}^d} \rho g \nabla \sqrt{ \rho (x,r)}\cdot \partial_{\zeta}\psi(x, \rho )dxdr
+\int_0^t\int_{\mathbb{T}^d}\sqrt{ \rho }g (x,r)\cdot\nabla_x\psi(x, \rho )dxdr.
\end{align}
To prove weak solutions are renormalized kinetic solutions, it needs to show that
\begin{align}\label{eqq-28}
  \lim_{\delta\rightarrow 0}I^{\delta}_3
   =& \int_0^t\int_{\mathbb{T}^d} \rho V\ast \rho \cdot\nabla_x \psi(x, \rho (x,r))dxdr
+  2\int_0^t\int_{\mathbb{T}^d}\sqrt{ \rho }\nabla\sqrt{ \rho }\cdot  \rho V\ast \rho \partial_{\xi}\psi(x, \rho )dxdr.
\end{align}

Since $\psi\in C^{\infty}_c(\mathbb{T}^d\times \mathbb{R})$, there exists $M\in (0,\infty)$ such that
\begin{align*}
{\rm{Supp}} (\psi)\subset \mathbb{T}^d\times [-M,M].
\end{align*}
For every $k\in \mathbb{N}_+$, let $A_k\subset \mathbb{T}^d\times [0,T]$ be defined by
\begin{align*}
  A_k=\{(y,t)\in\mathbb{T}^d\times [0,T]: \sqrt{ \rho (y,t)}\geq \sqrt{M}+k \},
\end{align*}
and let $A_0=(\mathbb{T}^d\times [0,T])\setminus A_1$. 
Let $I^{\delta}_3$ be decomposed by $I^{\delta}_3=I^{\delta}_{31}+I^{\delta}_{32}$, where
\begin{align*}
I^{\delta}_{31}
=& -\int^t_0\int_{(\mathbb{T}^d)^2} \rho (y,r)V\ast \rho (y,r)\cdot\nabla_x\kappa^{\delta}(y-x) \psi(x, \rho ^{\delta}(x,r))I_{A_0}(y,r)dydxdr,\\
I^{\delta}_{32}
=&- \int^t_0\int_{(\mathbb{T}^d)^2} \rho (y,r)V\ast \rho (y,r)\cdot\nabla_x\kappa^{\delta}(y-x) \psi(x, \rho ^{\delta}(x,r))I_{A_1}(y,r)dydxdr.
\end{align*}
Let us start with the term $I^{\delta}_{32}$. We will show that
\begin{align}\label{eqq-38}
  \limsup_{\delta\rightarrow 0} |I^{\delta}_{32}|=0.
\end{align}
Employing H\"{o}lder's inequality, it follows that
\begin{align}\label{E OF K W q4}
  |I^{\delta}_{32}|\leq & \delta^{-1}\Big(\int^t_0\int_{(\mathbb{T}^d)^2} \rho (y,r)|V\ast \rho |^2(y,r)|\delta\nabla_x\kappa^{\delta}(y-x)|
  |\psi(x, \rho ^{\delta}(x,r))|I_{A_1}(y,r)dydxdr\Big)^{\frac{1}{2}}\notag\\
 & \cdot \Big(\int^t_0\int_{(\mathbb{T}^d)^2} \rho (y,r)|\delta\nabla_x\kappa^{\delta}(y-x)|
  |\psi(x, \rho ^{\delta}(x,r))|I_{A_1}(y,r)dydxdr\Big)^{\frac{1}{2}}.
\end{align}
For every $k\in\mathbb{N}$, let $I_{k,k+1} : \mathbb{T}^d\times[0,T] \rightarrow \{0, 1\}$ denote the indicator function of $A_k\setminus A_{k+1}$. By the definition of $A_k$, it holds that
\begin{align}
  \sqrt{M}+k\leq \sqrt{ \rho (y,s)}< \sqrt{M}+k+1,\quad  {\text{on}}\  A_k\setminus A_{k+1}.
\end{align}
Then, we get
\begin{align}\label{E OF K W q1}
  &\int^t_0\int_{(\mathbb{T}^d)^2} \rho (y,r)|\delta\nabla_x\kappa^{\delta}(y-x)|
  |\psi(x, \rho ^{\delta}(x,r))|I_{A_1}(y,r)dydxdr\nonumber\\
 = & \int^t_0\int_{(\mathbb{T}^d)^2} \rho (y,r)|\delta\nabla_x\kappa^{\delta}(x)|
  |\psi(y+x, \rho ^{\delta}(y+x,r))|I_{A_1}(y,r)dydxdr\nonumber\\
  = &\sum^{\infty}_{k=1}\int^t_0\int_{(\mathbb{T}^d)^2} \rho (y,r)|\delta\nabla_x\kappa^{\delta}(x)|
  |\psi(y+x, \rho ^{\delta}(y+x,r))|I_{k,k+1}(y,r)dydxdr\nonumber\\
  \leq & \sum^{\infty}_{k=1}\int^t_0\int_{(\mathbb{T}^d)^2}(\sqrt{M}+k+1)^2|\delta\nabla_x\kappa^{\delta}(x)|
  |\psi(y+x, \rho ^{\delta}(y+x,r))|I_{k,k+1}(y,r)dydxdr.
\end{align}
For every $x\in \mathbb{T}^d$ and $k\in \mathbb{N}$, let $B^{\delta}_{k,x}\subseteq \mathbb{T}^d\times [0,T]$ be defined by
\begin{align*}
B^{\delta}_{k,x}=\{(y,r)\in \mathbb{T}^d\times [0,T]:|\sqrt{ \rho (y,r)}-(\sqrt{ \rho })^{\delta}(y+x,r)|I_{k,k+1}(y,r)\geq k\},
\end{align*}
where for each $(y,r)\in \mathbb{T}^d\times [0,T]$,
\begin{align*}
  (\sqrt{ \rho })^{\delta}(y,r):=\int_{\mathbb{T}^d}\sqrt{ \rho (y',r)}\kappa^{\delta}
  (y-y')dy'.
\end{align*}
 Since the square root function $\sqrt{\cdot}$ is concave, it then follows from H\"older's inequality
 that
$$
\left(\sqrt{ \rho }\right)^{\delta}(y+x, r) \leq \sqrt{ \rho ^{\delta}(y+x, r)}.
$$
Therefore, for every $(y, r) \in\left(\left(A_{k} \backslash A_{k+1}\right) \cap \operatorname{Supp}\left[\psi\left(\cdot+x, \rho^{\delta}(\cdot+x, \cdot)\right)\right]\right)$,
$$
\sqrt{ \rho (y, r)}-\left(\sqrt{ \rho }\right)^{\delta}(y+x, r) \geq \sqrt{ \rho (y, r)}-\sqrt{ \rho ^{\delta}(y+x, r)} \geq k,
$$
it implies that
\begin{align}\label{E OF K W q2}
(A_k\setminus A_{k+1})\cap {\rm{Supp}}(\psi(\cdot+x, \rho ^{\delta}(\cdot+x,r)))\subseteq B^{\delta}_{k,x}.
\end{align}
Returning to \eqref{E OF K W q1}, it follows from $\psi\in C^{\infty}_c(\mathbb{T}^d\times \mathbb{R})$ and \eqref{E OF K W q2} that there exists $C\in(0,\infty)$ such
that
\begin{align}\label{E OF K W q3}
  &\int^t_0\int_{(\mathbb{T}^d)^2} \rho (y,r)|\delta\nabla_x\kappa^{\delta}(y-x)|
  |\psi(x, \rho ^{\delta}(x,r))|I_{A_1}(y,r)dydxdr\nonumber\\
  \leq & C\int_{\mathbb{T}^d}\sum^{\infty}_{k=1}(\sqrt{M}+k+1)^2|B^{\delta}_{k,x}|
  |\delta\nabla_x\kappa^{\delta}(x)|
  dx.
\end{align}
It remains to estimate the measure of the sets $\left\{B_{k, x}^{\delta}\right\}_{k \in \mathbb{N}, x \in \mathbb{T}^{d}}$.
Let $x \in \mathbb{T}^{d}$ and $k \in \mathbb{N}$. It follows from Chebyshev's inequality, H\"older's inequality, Jensen's inequality and the fundamental theorem of calculus that, omitting the integration variables,
$$
\begin{aligned}
\left|B_{k, x}^{\delta}\right| & \leq \frac{1}{k^{2}} \int_{0}^{T} \int_{\mathbb{T}^{d}}\left|\sqrt{ \rho (y, t)}-\left(\sqrt{ \rho }\right)^{\delta}(y+x, t)\right|^{2} \mathbf{1}_{k, k+1}(y, t) \\
& \leq \frac{1}{k^{2}} \int_{0}^{T} \int_{\left(\mathbb{T}^{d}\right)^{2}}\left|\sqrt{ \rho (y, t)}-\sqrt{ \rho (y', t)}\right|^{2} \kappa^{\delta}\left(y+x-y^{\prime}\right) \mathbf{1}_{k, k+1}(y, t) \\
& \leq \frac{1}{k^{2}} \int_{0}^{T} \int_{\left(\mathbb{T}^{d}\right)^{2}}\left(\int_{0}^{1}\left|
\nabla\sqrt{ \rho }\left(s y+(1-s) y^{\prime}, t\right)\right|^{2} \mathrm{~d} s\right)\left|y-y^{\prime}\right|^{2} \kappa^{\delta}\left(y+x-y^{\prime}\right) \mathbf{1}_{k, k+1}(y, t).
\end{aligned}
$$
Since $\left|y-y^{\prime}\right| \leq|x|+\delta$ on the support of the convolution kernel, we get
$$
\left|B_{k, x}^{\delta}\right| \leq \frac{\left(\delta^{2}+|x|^{2}\right)}{k^{2}} \int_{0}^{T} \int_{\left(\mathbb{T}^{d}\right)^{2}}\left(\int_{0}^{1}\left|\nabla
\sqrt{ \rho }\left(s y+(1-s) y^{\prime}, t\right)\right|^{2} \mathrm{~d} s\right) \kappa^{\delta}\left(y+x-y^{\prime}\right) \mathbf{1}_{k, k+1}(y, t) .
$$
Therefore, it gives that
\begin{align*}
  \sum^{\infty}_{k=1}(\sqrt{M}+k+1)^2|B^{\delta}_{k,x}|
  &\leq  \sum^{\infty}_{k=1}\Big(
  \frac{(\delta^2+|x|^2)(\sqrt{M}+k+1)^2}{k^2}\\
  &\times \int^T_0\int_{(\mathbb{T}^d)^2}
\Big(\int^1_0|\nabla\sqrt{ \rho }(sy+(1-s)y',t)|^2ds\Big)\kappa^{\delta}(y+x-y')I_{k,k+1}(y,t)
dy'dydt\Big)\\
\leq&c(M)(\delta^2+|x|^2)\int^T_0\int_{(\mathbb{T}^d)^2}
\Big(\int^1_0|\nabla\sqrt{ \rho }(sy+(1-s)y',t)|^2ds\Big)\kappa^{\delta}(y+x-y')I_{A_1}(y,t)
dy'dydt\Big)\\
\leq&c(M)(\delta^2+|x|^2)\|\nabla\sqrt{ \rho }\|_{L^2([0,T];L^2(\mathbb{T}^d)} .
\end{align*}
Returning to \eqref{E OF K W q3}, the above inequality implies that
\begin{align*}
  &\int^t_0\int_{(\mathbb{T}^d)^2} \rho (y,r)|\delta\nabla_x\kappa^{\delta}(y-x)|
  |\psi(x, \rho ^{\delta}(x,r))|I_{A_1}(y,r)dydxdr\\
  \leq &\int_{\mathbb{T}^d}c(M)(\delta^2+|x|^2)\|\nabla\sqrt{ \rho }\|_{L^2([0,T];L^2(\mathbb{T}^d)}
  |\delta\nabla_x\kappa^{\delta}(x)|dx\\
  \leq & c(M)\delta^2\|\nabla\sqrt{ \rho }\|_{L^2([0,T];L^2(\mathbb{T}^d)}.
\end{align*}
As a result, we deduce from \eqref{E OF K W q4} that
\begin{align}\label{eqq-32}
  |I^{\delta}_{32}|\leq & c(M)\|\nabla\sqrt{ \rho }\|_{L^2([0,T];L^2(\mathbb{T}^d)}^{1/2} \Big(\int^t_0\int_{(\mathbb{T}^d)^2} \rho |V\ast \rho |^2(y,r)\nonumber\\
  &\times|\delta\nabla_x\kappa^{\delta}(y-x)|
  |\psi(x, \rho ^{\delta}(x,r))|I_{A_1}(y,r)dydxdr\Big)^{\frac{1}{2}}.
\end{align}
For $\delta=1$, let $\kappa=\kappa^{1}$, $c_{0}=\int_{\mathbb{T}^{d}}|\nabla \kappa(x)| \mathrm{d} x$, the functions $\left\{c_{0}^{-1}\left|\delta \nabla \kappa^{\delta}\right|\right\}_{\delta \in(0,1)}$ are a Dirac sequence on $\mathbb{T}^{d}$. Thanks to (\ref{cond-2}), we have
\begin{align*}
  &\int_{0}^{T}\int_{\mathbb{T}^d} \rho (y,r)|V\ast \rho (y,r)|^2dydr<\infty,
\end{align*}
it then follows from the definition of $A_{1}$ and the dominated convergence theorem that
\begin{align}\label{eqq-37}
  \limsup_{\delta\rightarrow 0}\int^t_0\int_{(\mathbb{T}^d)^2} \rho |V\ast \rho |^2(y,r)|\delta\nabla_x\kappa^{\delta}(y-x)|
  |\psi(x, \rho ^{\delta}(x,r))|I_{A_1}(y,r)dydxdr=0.
\end{align}
Combining (\ref{eqq-32}) and (\ref{eqq-37}),  we have
$
  \lim\sup_{\delta\rightarrow 0} |I^{\delta}_{32}|=0.
$

It remains to make estimates of $I^{\delta}_{31}$. The chain rule implies that
\begin{align*}
 \nabla \psi(x, \rho ^{\delta}(x,r))=\nabla_x \psi(x, \rho ^{\delta}(x,r))
 +\partial_{\zeta}\psi(x, \rho ^{\delta}(x,r))\nabla \rho ^{\delta},
\end{align*}
therefore, after integration by parts in the $x-$variable, we obtain that $I^{\delta}_{31}=: J^{\delta}_1+J^{\delta}_2$, where
\begin{align*}
J^{\delta}_1
= & \int^t_0\int_{(\mathbb{T}^{d})^2} \rho (y,r)V\ast \rho (y,r)\kappa^{\delta}(y-x)\cdot \nabla_x \psi(x, \rho ^{\delta}(x,r))I_{A_0}(y,r)dydxdr,\\
J^{\delta}_2= &\int^t_0\int_{(\mathbb{T}^{d})^2} \rho (y,r)V\ast \rho (y,r)\kappa^{\delta}(y-x) \partial_{\zeta}\psi(x, \rho ^{\delta}(x,r))\cdot\nabla \rho ^{\delta}(x,r)I_{A_0}(y,r)dydxdr.
\end{align*}
By the definition of $I_{A_0}$, we know that there exists $C\in(0,\infty)$ such that
$| \rho (y,r)V\ast \rho (y,r)I_{A_0}(y,r)|\leq C(M)|V\ast  \rho (y,r)|$. It combines the convolution Young's inequality proves that $| \rho V\ast \rho I_{A_0}|$ is integrable. Therefore, by the definition of $\psi$ and the dominated convergence theorem,
\begin{align}\label{eqq-26}
\lim_{\delta\rightarrow 0}J^{\delta}_1
= \int^t_0\int_{\mathbb{T}^{d}} \rho (x,r)V\ast \rho (x,r)\cdot \nabla_x\psi(x, \rho (x,r))dxdr.
\end{align}
In the following, we will prove that
\begin{align}\label{eqq-27}
\lim_{\delta\rightarrow 0}J^{\delta}_2
=2\int^t_0\int_{\mathbb{T}^{d}} \rho (x,r)V\ast \rho (x,r)\sqrt{ \rho (x,r)}\partial_{\zeta}\psi(x, \rho (x,r))\cdot
\nabla \sqrt{ \rho (x,r)}dxdr.
\end{align}
Noting that
\begin{align*}
  \nabla  \rho ^{\delta}(x,r)=& \int_{\mathbb{T}^{d}} \rho (z,r)\nabla_x\kappa^{\delta}(x-z)dz
  = 2\int_{\mathbb{T}^{d}}\sqrt{ \rho(z,r) }\nabla_z(\sqrt{ \rho (z,r)})
  \kappa^{\delta}(x-z)dz,
\end{align*}
it gives that
\begin{align}\label{E OF K W q5}
J^{\delta}_2
= & \int^t_0\int_{(\mathbb{T}^{d})^2} \rho (y,r)V\ast \rho (y,r)\kappa^{\delta}(y-x) (\partial_{\zeta}\psi)(x, \rho ^{\delta}(x,r))\nonumber\\
& \times\Big(\int_{\mathbb{T}^{d}}2\sqrt{ \rho(z,r) }\nabla_z(\sqrt{ \rho (z,r)})
  \kappa^{\delta}(z-x)dz\Big)I_{A_0}(y,r)dydxdr =: K^{\delta}_1+K^{\delta}_2,
\end{align}
where
\begin{align*}
  K^{\delta}_1  = & \int^t_0\int_{(\mathbb{T}^{d})^2} \rho (y,r)V\ast \rho (y,r)\kappa^{\delta}(y-x) (\partial_{\zeta}\psi)(x, \rho ^{\delta}(x,r))\\
&\times\Big(\int_{\mathbb{T}^{d}}2\sqrt{ \rho(z,r) }\nabla_z(\sqrt{ \rho (z,r)})
  \kappa^{\delta}(z-x)I_{A_0}(z)dz\Big)I_{A_0}(y,r)dydxdr,
  \end{align*}
  and
  \begin{align*}
  K^{\delta}_2  = & \int^t_0\int_{(\mathbb{T}^{d})^2} \rho (y,r)V\ast \rho (y,r)\kappa^{\delta}(y-x) (\partial_{\zeta}\psi)(x, \rho ^{\delta}(x,r))\\
& \times\Big(\int_{\mathbb{T}^{d}}2\sqrt{ \rho(z,r) }\nabla_z(\sqrt{ \rho (z,r)})
  \kappa^{\delta}(z-x)I_{A_1}(z)dz\Big)I_{A_0}(y,r)dydxdr.
\end{align*}
Regarding the term $K^{\delta}_2$,  since $\psi\in C^{\infty}_c(\mathbb{T}^d\times \mathbb{R})$, with the help of H\"older's inequality and the definition of $I_{A_0}$, we get
\begin{align*}
K^{\delta}_2\leq & 2\int^t_0\int_{(\mathbb{T}^{d})^2}\Big( \rho (y,r)V\ast \rho (y,r)\kappa^{\delta}(y-x) (\partial_{\zeta}\psi)(x, \rho ^{\delta}(x,r))( \rho ^{\delta}(x,r))^{\frac{1}{2}}I_{A_0}(y,r)\\
& \times\Big(\int_{\mathbb{T}^{d}}|\nabla_z\sqrt{ \rho (z,r)}|^2
  \kappa^{\delta}(z-x)I_{A_1}(z,r)dz\Big)^{\frac{1}{2}}\Big)dydxdr\\
  \leq & C(M)\int^t_0\int_{(\mathbb{T}^{d})^2}\Big(|V\ast \rho (y,r)|\kappa^{\delta}(y-x) (\partial_{\zeta}\psi)(x, \rho ^{\delta}(x,r))I_{A_0}(y,r)\\
& \times\Big(\int_{\mathbb{T}^{d}}|\nabla_z\sqrt{ \rho (z,r)}|^2
  \kappa^{\delta}(z-x)I_{A_1}(z,r)dz\Big)^{\frac{1}{2}}\Big)dydxdr.
\end{align*}
Since  $V\ast  \rho \in L^2([0,T]\times \mathbb{T}^{d})$, therefore,  $K^{\delta}_2$ can be bounded as follows,
\begin{align*}
K^{\delta}_2\leq &
 C(M)\Big(\int^t_0\int_{(\mathbb{T}^{d})^2}|V\ast \rho (y,r)|^2\kappa^{\delta}(y-x) (\partial_{\zeta}\psi)(x, \rho ^{\delta}(x,r))I_{A_0}(y,r)dydxdr\Big)^{\frac{1}{2}}\\
& \times\Big(\int^t_0\int_{(\mathbb{T}^{d})^3}|\nabla_z\sqrt{ \rho (z,r)}|^2
(\partial_{\zeta}\psi)(x, \rho ^{\delta}(x,r))
  \kappa^{\delta}(z-x)\kappa^{\delta}(y-x)I_{A_0}(y,r) I_{A_1}(z,r) dydx dz dr\Big)^{\frac{1}{2}}\\
  \leq & C(M)\|V\ast \rho \|^2_{L^2([0,T]\times\mathbb{T}^d)}\Big(\int^t_0\int_{(\mathbb{T}^{d})^2}|\nabla_z\sqrt{ \rho (z,r)}|^2
(\partial_{\zeta}\psi)(x, \rho ^{\delta}(x,r))
  \kappa^{\delta}(z-x)I_{A_1}(z,r) dx dz dr\Big)^{\frac{1}{2}}.
\end{align*}
Moreover, since $|\nabla\sqrt{ \rho (z,r)}|^2I_{A_1}(z,r)$ is integrable, we obtain
\begin{align*}
  \int_{\mathbb{T}^{d}}|\nabla\sqrt{ \rho (z,\cdot)}|^2
  \kappa^{\delta}(z-\cdot)I_{A_1}(z,\cdot) dz \rightarrow |\nabla\sqrt{ \rho (\cdot,\cdot)}|^2
  I_{A_1}(\cdot,\cdot)\ \  {\rm{strongly\ in\ }} L^1([0,T];L^1(\mathbb{T}^d)),
\end{align*}
as $\delta\rightarrow 0$. By the definition of $A_1$ and the fact that $\psi\in C^{\infty}_c(\mathbb{T}^d\times \mathbb{R})$, we obtain
\begin{align*}\notag
 & \lim_{\delta\rightarrow 0}\Big(\int^t_0\int_{(\mathbb{T}^{d})^2}|\nabla_z\sqrt{ \rho (z,r)}|^2
(\partial_{\zeta}\psi)(x, \rho ^{\delta}(x,r))
  \kappa^{\delta}(z-x)I_{A_1}(z,r) dx dz dr dr\Big)\\
  \label{eqq-30}
  =& \int^t_0\int_{\mathbb{T}^{d}}|\nabla_z\sqrt{ \rho (z,r)}|^2
(\partial_{\zeta}\psi)(z, \rho (z,r))
  I_{A_1}(z,r)  dz dr=0.
\end{align*}
It implies that
\begin{equation}\label{wr-5}
  \lim_{\delta\rightarrow 0}K^{\delta}_2=0.
\end{equation}

Concerning the term $K^{\delta}_1$,
by the definition of $A_0$, we obtain that $2\sqrt{ \rho (z,r)}\nabla\sqrt{ \rho (z,r)}
  I_{A_0}(z,r)$ is integrable. It then follows from the fact that $\psi\in C^{\infty}_c(\mathbb{T}^d\times \mathbb{R})$ and the  dominated convergence theorem that
\begin{align}\notag
  &\lim_{\delta\rightarrow 0}\partial_{\zeta}\psi(\cdot,\rho^{\delta}(\cdot,\cdot))
  \Big(\int_{\mathbb{T}^{d}}2\sqrt{ \rho (z,\cdot)}\nabla\sqrt{ \rho (z,\cdot)}
  \kappa^{\delta}(z-\cdot)I_{A_0}(z,\cdot)dz\Big)\\
  \label{eqq-24}
  = & \partial_{\zeta}\psi(\cdot, \rho (\cdot,\cdot))
 2\sqrt{ \rho (\cdot,\cdot)}\nabla\sqrt{ \rho (\cdot,\cdot)}\ \  {\rm{strongly\ in \ }} L^2([0,T];L^2(\mathbb{T}^d)).
\end{align}
Moreover,  by the definition of $A_0$ and the fact that $V\ast  \rho \in L^2([0,T];L^2(\mathbb{T}^d))$,
\begin{align}\notag
&\lim_{\delta\rightarrow 0}\int_{\mathbb{T}^{d}} \rho (y,\cdot)V\ast \rho (y,\cdot)I_{A_0}(y,\cdot)\kappa^{\delta}(y-\cdot) dy\\
\label{eqq-25}
&= \rho (\cdot,\cdot)V\ast  \rho (\cdot,\cdot)I_{A_0}(\cdot,\cdot)\quad {\rm{strongly\ in \ }} L^2([0,T];L^2(\mathbb{T}^d)).
\end{align}
Combining (\ref{eqq-24}) and (\ref{eqq-25}), by H\"{o}lder inequality and dominated convergence theorem, it gives that
\begin{align*}
  \lim_{\delta\rightarrow 0}K^{\delta}_1
  =2\int^t_0\int_{\mathbb{T}^{d}} \rho (x,r)V\ast \rho (x,r)\sqrt{ \rho (x,r)}(\partial_{\zeta}\psi)(x, \rho (x,r))\cdot
\nabla \sqrt{ \rho (x,r)}dxdr,
\end{align*}
this, together with \eqref{E OF K W q5} and \eqref{wr-5}, yields (\ref{eqq-27}).

Based on  (\ref{eqq-26}) and  (\ref{eqq-27}), we get
\begin{align}\label{I3}
\lim_{\delta\rightarrow 0}I^{\delta}_{31}
= &\int^t_0\int_{\mathbb{T}^{d}} \rho (x,r)V\ast \rho (x,r)\cdot (\nabla \psi)(x, \rho (x,r))dxdr\nonumber\\
&+ 2\int^t_0\int_{\mathbb{T}^{d}} \rho (x,r)V\ast \rho (x,r)\sqrt{ \rho (x,r)}(\partial_{\zeta}\psi)(x, \rho (x,r))\cdot
\nabla \sqrt{ \rho (x,r)}dxdr.
\end{align}
Therefore, (\ref{eqq-28}) follows from  (\ref{I3}) and (\ref{eqq-38}).

In combination with \eqref{I1 2} and \eqref{eqq-28}, it follows that
\begin{align*}
&\left.\int_{\mathbb{T}^{d}} \Psi(x,  \rho (x, r)) \mathrm{d} x\right|_{r=0} ^{t}=\left.\lim _{\delta \rightarrow 0} \int_{\mathbb{T}^{d}} \Psi\left(x, \rho ^{\delta}(x, r)\right) \mathrm{d} x\right|_{r=0} ^{t} \\
=&\int_0^t\int_{\mathbb{R}}\int_{\mathbb{T}^d}\chi\Delta_x\psi dxd\xi dr-\int_0^t\int_{\mathbb{R}}\int_{\mathbb{T}^d}q\partial_{\xi}\psi dxd\xi dr-\int_{0}^{t} \int_{\mathbb{T}^{d}}\nabla\cdot( \rho  V\ast \rho ) \psi(x,  \rho (x, r)) \mathrm{d} x \mathrm{~d} r \\
& +\int_{0}^{t} \int_{\mathbb{T}^{d}}\sqrt{\rho}(x,r) g (x, r)  \cdot\left(\nabla_{x} \psi\right)(x, \rho (x, r)) \mathrm{d} x \mathrm{~d}r \\
&+\int_{0}^{t} \int_{\mathbb{T}^{d}} \sqrt{\rho(x,r)}g (x, r) \left(\partial_{\zeta} \psi\right)(x,  \rho (x,r)) \cdot \nabla \rho (x, r) \mathrm{d} x \mathrm{~d} r.
\end{align*}
Thus, $\rho$ is the kinetic solution of (\ref{skeleton}).
\end{proof}

\begin{remark}
Under Assumption (A1), by regularities of the solution $\rho$, we  readily check that both (\ref{cond-1}) and (\ref{cond-2}) hold.	
\end{remark}

\subsection{Existence of weak solutions to the skeleton equation}


We need the following result whose proof can be found in \cite[Lemma 17]{FG23}.
\begin{lemma}\label{lem-3}
 Assume that the sequence $\{\rho_n\}_{n\in \mathbb{N}}\subseteq L^{\infty}([0,T];L^1(\mathbb{T}^d))$ satisfy that for some $c\in (0,\infty)$ and $s\geq\frac{d}{2}+1$ independent of $n\in \mathbb{N}$,
\begin{align*}
 \|\rho_n\|_{L^{\infty}([0,T];L^1(\mathbb{T}^d))}+\|\sqrt{\rho_n}\|_{L^2([0,T];H^1(\mathbb{T}^d))}
+\|\partial_t \rho_n\|_{L^1([0,T];H^{-s}(\mathbb{T}^d))}\leq c.
\end{align*}
Then,
\begin{align*}
 \{\rho_n\}_{n\in \mathbb{N}}\ {\rm{is\ relatively\ precompact\ in\ }} L^{1}([0,T];L^1(\mathbb{T}^d)),
\end{align*}
 and
 \begin{align*}
 \{\sqrt{\rho_n}\}_{n\in \mathbb{N}}\ {\rm{is\ relatively\ precompact\ in\ }} L^{2}([0,T];L^2(\mathbb{T}^d)).
\end{align*}
\end{lemma}


We need \cite[Proposition 20]{FG23} as well, which reads as follows.
\begin{lemma}\label{lem-2}
For any $\eta\in(0,1)$, there exist nondecreasing functions $\sigma^{\frac{1}{2},\eta}: [0,\infty)\rightarrow [0,\infty)$ satisfying that $\{\sigma^{\frac{1}{2},\eta}(\cdot)\}_{\eta>0}\subseteq C^{\infty}([0,\infty))$, $\sigma^{\frac{1}{2},\eta}(0)=0$ and the following properties.
\begin{description}
  \item[(i)]
  There exists $c_1\in(0,\infty)$ such that for every $\eta\in(0,1)$ and $\zeta\in[0,\infty)$,
\begin{equation}\label{t-14}
0\leq \sigma^{\frac{1}{2},\eta}(\zeta)\leq c_1\sqrt{\zeta},\ and\ 0\leq(\sigma^{\frac{1}{2},\eta})'(\zeta)\leq \frac{c_1 }{\sqrt{\zeta}}.
\end{equation}
  \item[(ii)]
  For every $\eta\in(0,1)$, there exists $c\in(0,\infty)$ depending on $\eta$ such that
\begin{equation}\label{t-15}
\|\sigma^{\frac{1}{2},\eta}\|_{L^{\infty}([0,\infty))}+\|(\sigma^{\frac{1}{2},\eta})'\|_{L^{\infty}([0,\infty))}\leq c.
\end{equation}
\item[(iii)]
 For every compact set $A\subseteq[0,\infty)$,
\begin{equation}\label{s-38}
\lim_{\eta\rightarrow0}\|\sigma^{\frac{1}{2},\eta}(\zeta)-\sqrt{\zeta}\|_{L^{\infty}(A)}=0.
\end{equation}

\end{description}
\end{lemma}
Assume that the kernel $V$ satisfies Assumption (A1) and let $\kappa_{\gamma}$ be a standard convolution kernel. Set
$V^{\gamma}=((V\wedge(1/\gamma))\vee(-1/\gamma))\ast\kappa_{\gamma}$. Then, $V^{\gamma}$ satisfies that
\begin{align}\label{kk-65}
\|V^{\gamma}\|_{L^{p^*}([0,T];L^{p}(\mathbb{T}^d))}\leq&\|V\|_{L^{p^*}([0,T];L^{p}(\mathbb{T}^d))}, \text{ and }\|V^{\gamma}\|_{L^{\infty}([0,T];L^{\infty}(\mathbb{T}^d))}\le\frac{1}{\gamma}.
\end{align}
Furthermore,
\begin{align}\label{r-18}
\|V-V^{\gamma}\|_{ L^{p^{*}}([0,T];L^p(\mathbb{T}^d))}\rightarrow 0,\ \ \text{as }\gamma\rightarrow0.
\end{align}
Let $\sigma^{\eta}:=(\sigma^{\frac{1}{2},\eta})^2$,
 we introduce a regularized approximating equation of (\ref{skeleton}). For every $\eta\in(0,1)$ and $\gamma\in(0,1)$, consider
 \begin{align}\label{qq-25}
 \left\{
   \begin{array}{ll}
     \partial_t\rho^{\eta,\gamma}=\Delta\rho^{\eta,\gamma}-\nabla\cdot (\sigma^{\frac{1}{2},\eta}(\rho^{\eta,\gamma})g)-\nabla\cdot (\sigma^{\eta}(\rho^{\eta,\gamma}) V^{\gamma}\ast\rho^{\eta,\gamma}), & \\
    \rho^{\eta,\gamma}(0)=\rho_0. &
   \end{array}
 \right.
 \end{align}

\begin{definition}\label{dfn-4}
Let $\eta, \gamma\in(0,1)$, $g\in L^2(\mathbb{T}^d\times[0,T];\mathbb{R}^d)$ and nonnegative $\rho_0\in L^{\infty}(\mathbb{T}^d)$. A nonnegative function $\rho\in L^{\infty}([0,T];L^1(\mathbb{T}^d))$ is called a weak solution of (\ref{qq-25}) with initial data $\rho_0$ if $\rho $
 is in $L^2([0,T];H^1(\mathbb{T}^d))$, and for every $\psi\in C^{\infty}(\mathbb{T}^d)$ and almost every $t\in[0,T]$,
	\begin{align}\notag
	\int_{\mathbb{T}^d}\rho(x,t)\psi(x)dx=&\int_{\mathbb{T}^d}\rho_0\psi dx
-\int_0^t\int_{\mathbb{T}^d}\nabla\rho\cdot\nabla\psi dxds
+\int_0^t\int_{\mathbb{T}^d}\sigma^{\eta}(\rho) V^{\gamma}\ast\rho\cdot\nabla\psi dxds\\
\notag
	&+\int_0^t\int_{\mathbb{T}^d}\sigma^{\frac{1}{2},\eta}(\rho)g\cdot\nabla\psi dxds.
	\end{align}
\end{definition}

\begin{proposition}\label{existence-2}
Let $V$ satisfy Assumption (A1). Let $\eta,\gamma\in(0,1)$, let $\rho_0\in L^{\infty}(\mathbb{T}^d)$ be nonnegative and $g\in L^2(\mathbb{T}^d\times[0,T];\mathbb{R}^d)$. Then (\ref{qq-25}) admits a nonnegative weak solution $\rho^{\eta,\gamma}\in L^2([0,T];H^1(\mathbb{T}^d))$ with initial data $\rho_0$ in the sense of Definition \ref{dfn-4}. Furthermore, almost surely for almost every $t\in [0,T]$,
\begin{align*}
  \|\rho^{\eta,\gamma}(t)\|_{L^1(\mathbb{T}^d)}=\|\rho_0\|_{L^1(\mathbb{T}^d)}.
\end{align*}
\end{proposition}
\begin{proof}
For any $\eta,\gamma\in (0,1)$, let
\begin{equation*}
S:  L^2([0,T];L^2(\mathbb{T}^d))\rightarrow L^2([0,T];H^1(\mathbb{T}^d))\subset L^2([0,T];L^2(\mathbb{T}^d)),
\end{equation*}
be a map defined by the unique weak solution of the equation
\begin{equation}\label{qq-9}
\left\{
  \begin{array}{ll}
   \partial_tS(v)=\Delta S(v)-\nabla\cdot (\sigma^{\frac{1}{2},\eta}(v)g)-\nabla\cdot (\sigma^{\eta}(v)V^{\gamma}\ast S(v)) ,& \\
   S(v)(\cdot,0)=\rho_0(\cdot), &
  \end{array}
\right.
\end{equation}
%
for every $v\in  L^2([0,T];L^2(\mathbb{T}^d))$. Due to (\ref{kk-65}), by integration by parts, the Young and convolution Young inequalities, we deduce that
\begin{align}\label{qq-6}
\Big|\int^T_0\langle S(v), -\nabla\cdot  (\sigma^{\eta}(v)V^{\gamma}\ast S(v))\rangle ds\Big|\leq \frac{1}{4}\int^T_0\|\nabla S(v)\|^2_{L^2(\mathbb{T}^d)}ds+ c(\eta,
  \gamma)\int^T_0\|S(v)\|^2_{L^2(\mathbb{T}^d)} ds.
  \end{align}
Taking $L^2(\mathbb{T}^d)$-inner product of (\ref{qq-9}), by using integration by parts formula, (\ref{qq-6}) and Gronwall inequality, we get
\begin{align}\label{qq-16}
\|S(v)\|_{L^{\infty}([0,T];L^2(\mathbb{T}^d))}^2+\|\nabla S(v)\|_{L^2([0,T];L^2(\mathbb{T}^d))}^2\leq c(\eta,\gamma,\rho_0)(1+\|g\|_{L^2(\mathbb{T}^d\times[0,T])}^2).
\end{align}
Moreover, from (\ref{qq-9}) and (\ref{qq-16}), it is readily to deduce that
\begin{align}\label{eqq-1}
\|\partial_tS(v)\|^2_{L^2([0,T];H^{-1}(\mathbb{T}^d))}
\leq c(\eta,\gamma,\rho_0)(1+\|g\|_{L^2(\mathbb{T}^d\times[0,T])}^2).
\end{align}
The Aubin-Lions-Simons lemma (see \cite[Corollary 5]{Simon}) implies that the image of the mapping $S$ lies in a compact subset of $ L^{2}([0,T];L^2(\mathbb{T}^d))$. The Schauder fixed point theorem yields that $S$ has a fixed point, which proves the existence of a weak solution to (\ref{qq-25}).

Let $\rho^{\eta,\gamma}\in L^2([0,T];H^1(\mathbb{T}^d))$ be a weak solution of (\ref{qq-25}). We follow the same procedure for the estimates of the map $S$, taking $L^2(\mathbb{T}^d)$-inner product for (\ref{qq-25}) to see that
\begin{align*}
\|\rho^{\eta,\gamma}\|_{L^{\infty}([0,T];L^2(\mathbb{T}^d))}^2+\|\nabla \rho^{\eta,\gamma}\|_{L^2([0,T];L^2(\mathbb{T}^d;\mathbb{R}^d))}^2\leq c(\eta,\gamma,\rho_0)\Big(1+\|g\|^2_{L^2(\mathbb{T}^d\times[0,T])}+\int^T_0\| \rho^{\eta,\gamma}\|^2_{L^2(\mathbb{T}^d)}ds\Big).
\end{align*}
Applying Gronwall inequality to the above equation, it gives
\begin{align}\label{qq-r-2}
\|\rho^{\eta,\gamma}\|_{L^{\infty}([0,T];L^2(\mathbb{T}^d))}^2+\|\nabla \rho^{\eta,\gamma}\|_{L^2([0,T];L^2(\mathbb{T}^d;\mathbb{R}^d))}^2\leq C(\eta,\gamma,\rho_0,\|g\|^2_{L^2(\mathbb{T}^d\times[0,T])},T).
\end{align}
In the following, we will show that $\rho^{\eta,\gamma}$ is nonnegative. Let $(\rho^{\eta,\gamma})^{-}=\min (\rho^{\eta,\gamma},0)$. Applying the chain rule to a regularization of the function $a(\xi)=|\xi|$ and by similar method as \cite[Theorem 5.7]{WWZ22}, it implies that $(\rho^{\eta,\gamma})^{-}=0$ for almost every $(t,x)\in[0,T]\times\mathbb{T}^d$. Thus, $\rho^{\eta,\gamma}$ is nonnegative.

Testing (\ref{qq-25}) with a constant function $1$, thanks to the fact that $\rho^{\eta,\gamma}$ is nonnegative, it gives that  for almost every $t\in [0,T]$,
\begin{align}\label{L1}
  \|\rho^{\eta,\gamma}(t)\|_{L^1(\mathbb{T}^d)}=\|\rho_0\|_{L^1(\mathbb{T}^d)}.
\end{align}
This completes the proof.
\end{proof}
In the sequel, we will show the well-posedness of the following equation
\begin{align}\label{rrr-3}
 \left\{
   \begin{array}{ll}
     \partial_t\rho^{\eta}=\Delta\rho^{\eta}-\nabla\cdot (\sigma^{\frac{1}{2},\eta}(\rho^{\eta})g)-\nabla\cdot (\sigma^{\eta}(\rho^{\eta}) V\ast\rho^{\eta}), & \\
    \rho^{\eta}(0)=\rho_0. &
   \end{array}
 \right.
 \end{align}
\begin{proposition}\label{existence-rhoeta}
Let $V$ satisfy Assumption (A1). Let $\rho_0\in L^{\infty}(\mathbb{T}^d)$ be nonnegative and $g\in L^2(\mathbb{T}^d\times[0,T];\mathbb{R}^d)$. Then for every $\eta\in (0,1)$, there exists a nonnegative weak solution $\rho^{\eta}\in L^2([0,T];H^1(\mathbb{T}^d))$ of (\ref{rrr-3}). Furthermore, for every $t\in[0,T]$,
\begin{align}\label{eq-12}
&\|\rho^{\eta}(\cdot,t)\|_{L^1(\mathbb{T}^d)}=\|\rho_0\|_{L^1(\mathbb{T}^d)},
\end{align}
and
\begin{align}\notag
&\|\rho^{\eta}\|_{L^{\infty}([0,T];L^2(\mathbb{T}^d))}^2+\|\nabla \rho^{\eta}\|_{L^2([0,T];L^2(\mathbb{T}^d;\mathbb{R}^d))}^2\\
\label{rrr-1}
\leq&  \|\rho_0\|_{L^2(\mathbb{T}^d)}^2+c(\eta)
\|g\|^2_{L^2(\mathbb{T}^d\times[0,T])}+c(\eta,\|V\|_{L^{p^*}([0,T];L^{p}(\mathbb{T}^d))}).
\end{align}
\end{proposition}
\begin{proof}
Note that the $L^2(\mathbb{T}^d)$-norm estimate (\ref{qq-r-2}) depends on $\gamma$. However,
with the aid of (\ref{L1}), (\ref{qq-r-2}) can be improved to be uniform with respect to $\gamma$. Indeed, by H\"{o}lder inequality, convolution Young inequality, (\ref{t-15}), (\ref{kk-65}), (\ref{L1}) and Young inequality, the kernel term can be estimated as
\begin{align}\notag
  \Big|\int^T_0\int_{\mathbb{T}^d}\sigma^{\eta}(\rho^{\eta,\gamma})V^{\gamma}\ast\rho^{\eta,\gamma}\cdot \nabla \rho^{\eta,\gamma}dx ds\Big|\leq& c(\eta)\|\rho_0\|_{L^1(\mathbb{T}^d)}\Big(\int^T_0\|\nabla \rho^{\eta,\gamma}\|^2_{L^2(\mathbb{T}^d)}ds\Big)^{\frac{1}{2}}\Big(\int^T_0\|V^{\gamma}\|^2_{L^2(\mathbb{T}^d)}ds\Big)^{\frac{1}{2}}\\
  \label{s-41-1}
  \leq&  \frac{1}{4}\int^T_0\|\nabla \rho^{\eta,\gamma}\|^{2}_{L^2(\mathbb{T}^d)}ds
  +c(\eta,\|V\|_{L^{p^*}([0,T];L^{p}(\mathbb{T}^d))}).
\end{align}
Consequently, this implies that there exists a constant $c$ independent of $\gamma$ such that
\begin{align}\notag
&\|\rho^{\eta,\gamma}\|_{L^{\infty}([0,T];L^2(\mathbb{T}^d))}^2+\|\nabla \rho^{\eta,\gamma}\|_{L^2([0,T];L^2(\mathbb{T}^d;\mathbb{R}^d))}^2\\
\label{rrr-4}
\leq&\|\rho_0\|_{L^2(\mathbb{T}^d)}^2+c(\eta)
\|g\|^2_{L^2(\mathbb{T}^d\times[0,T])}+c(\eta,\|V\|_{L^{p^*}([0,T];L^{p}(\mathbb{T}^d))}).
\end{align}

For every $\psi\in H^s(\mathbb{T}^d)$ with $s>\frac{d}{2}+1$, by using Sobolev embedding theorem, it follows that there exists $c>0$ such that $\|\psi\|_{L^{\infty}(\mathbb{T}^d)}+\|\nabla \psi\|_{L^{\infty}(\mathbb{T}^d;\mathbb{R}^d)}\leq c\|\psi\|_{H^s(\mathbb{T}^d)}$. A direct computation shows that
\begin{align*}\notag
  \|\partial_t \rho^{\eta,\gamma}\|_{L^2([0,T];H^{-s}(\mathbb{T}^d))}
\leq&  \|\nabla\rho^{\eta,\gamma}\|_{L^2([0,T];L^1(\mathbb{T}^d))}
+\|\sigma^{\eta}(\rho^{\eta,\gamma})V^{\gamma}\ast\rho^{\eta,\gamma}\|_{L^2([0,T];L^1(\mathbb{T}^d))}\\
&+\|g\sigma^{\frac{1}{2},\eta}(\rho^{\eta,\gamma})\|_{L^2([0,T];L^1(\mathbb{T}^d))}.
\end{align*}
Regarding the kernel term, with the aid of (\ref{t-15}), (\ref{kk-65}) and (\ref{L1}), it gives
\begin{align*}
\|\sigma^{\eta}(\rho^{\eta,\gamma})V^{\gamma}\ast\rho^{\eta,\gamma}\|^2_{L^2([0,T];L^1(\mathbb{T}^d))}\leq c(\eta,T)\|\rho_0\|^2_{L^1(\mathbb{T}^d)} \|V\|^2_{L^{p^*}([0,T];L^p(\mathbb{T}^d))}.
\end{align*}
For the control term, it follows from (\ref{t-15}) that
\begin{align*}
\|g\sigma^{\frac{1}{2},\eta}(\rho^{\eta,\gamma})\|^2_{L^2([0,T];L^1(\mathbb{T}^d))}
\leq \int^T_0\|\sigma^{\frac{1}{2},\eta}\|^2_{L^{\infty}(\mathbb{R})}\|g\|^2_{L^1(\mathbb{T}^d)}dt
\leq c(\eta)\|g\|^2_{L^2(\mathbb{T}^d\times[0,T];\mathbb{R}^d)}.
\end{align*}
Based on all the above estimates, there exists $c>0$ independent of $\gamma$ such that
\begin{align}\notag
  \|\partial_t \rho^{\eta,\gamma}\|_{L^2([0,T];H^{-s}(\mathbb{T}^d))}
\leq& \Big[\|\rho_0\|_{L^2(\mathbb{T}^d)}^2+c(\eta)
\|g\|^2_{L^2(\mathbb{T}^d\times[0,T])}+c(\eta,\|V\|_{L^{p^*}([0,T];L^{p}(\mathbb{T}^d))})\Big]^{\frac{1}{2}}\\
\label{eq-16-1}
&+ c(\eta,T)\|\rho_0\|_{L^1(\mathbb{T}^d)} \|V\|_{L^{p^*}([0,T];L^p(\mathbb{T}^d))}+c(\eta)\|g\|_{L^2(\mathbb{T}^d\times[0,T];\mathbb{R}^d)}.
\end{align}
Based on (\ref{rrr-4}) and (\ref{eq-16-1}), it follows from the Aubin-Lions-Simon lemma that there exists $\rho^{\eta}\in L^2([0,T];H^1(\mathbb{T}^d))$ such that, after passing to a subsequence $\gamma^k\rightarrow 0$, as $k\rightarrow \infty$,
\begin{align}\label{qq-19-1}
&\rho^{\eta,\gamma^k}\rightarrow \rho^{\eta}\ {\rm{almost\ everywhere\ and\ strongly\ in\ }} L^2([0,T];L^2(\mathbb{T}^d)),\\
\label{qq-18-1}
& \rho^{\eta,\gamma^k}\rightarrow \rho^{\eta}\ {\rm{weakly\ in \ }} L^2([0,T];H^1(\mathbb{T}^d)).
 \end{align}
In the following, we will show that $\rho^{\eta}$ is a weak solution of (\ref{rrr-3}), which will be achieved by a procedure of passing to the limit $k\rightarrow \infty$. Thanks to \cite[Proposition 20]{FG23}, it suffices to handle the kernel term. For any $\eta\in (0,1)$, we claim that
 \begin{align}\label{kk-19}
\lim_{k\rightarrow \infty}\int^T_0\int_{\mathbb{T}^d}(\sigma^{\eta}(\rho^{\eta,\gamma^k})V^{\gamma^k}\ast \rho^{\eta,\gamma^k}-\sigma^{\eta}(\rho^{\eta})V\ast\rho^{\eta} )\cdot \nabla \psi dxdt= 0.
 \end{align}
 Note that
\begin{align*}
\Big|\int^T_0\int_{\mathbb{T}^d}(\sigma^{\eta}(\rho^{\eta,\gamma^k})V^{\gamma^k}\ast \rho^{\eta,\gamma^k}-\sigma^{\eta}(\rho^{\eta})V\ast\rho^{\eta} )\cdot \nabla \psi dxdt\Big|\leq I^k_1+I^k_2+I^k_3,
\end{align*}
where
\begin{align*}
I^k_1=& \Big|\int^T_0\int_{\mathbb{T}^d}\sigma^{\eta}(\rho^{\eta,\gamma^k})(V^{\gamma^k}-V)\ast \rho^{\eta,\gamma^k}\cdot \nabla \psi dxdt\Big|,\\
I^k_2=& \Big|\int^T_0\int_{\mathbb{T}^d}(\sigma^{\eta}(\rho^{\eta,\gamma^k})-\sigma^{\eta}(\rho^{\eta}))V\ast \rho^{\eta,\gamma^k}\cdot \nabla \psi dxdt\Big|,\\
I^k_3=&\Big |\int^T_0\int_{\mathbb{T}^d}\sigma^{\eta}(\rho^{\eta})V\ast (\rho^{\eta,\gamma^k}-\rho^{\eta})\cdot \nabla \psi dxdt\Big|.
 \end{align*}
With the aid of (\ref{t-15}), (\ref{r-18}) and (\ref{L1}), by convolution Young inequality, we get as $k\rightarrow \infty$,
\begin{align*}
  I^k_1\leq& C(\eta,\|\rho_0\|_{L^1(\mathbb{T}^d})\int^T_0\|V^{\gamma^k}-V\|_{L^1(\mathbb{T}^d)}dt\rightarrow 0.
\end{align*}
Since
\begin{align*}
 \Big|\int^T_0\int_{\mathbb{T}^d}(\sigma^{\eta}(\rho^{\eta,\gamma^k})-\sigma^{\eta}(\rho^{\eta}))V\ast \rho^{\eta,\gamma^k}\cdot \nabla \psi dxdt\Big|\leq C(\eta,\|\rho_0\|_{L^1(\mathbb{T}^d})\int^T_0\|V\|_{L^1(\mathbb{T}^d)}dt<\infty,
\end{align*}
by (\ref{qq-19-1}) and the dominated convergence theorem, it holds that $I^k_2\rightarrow 0$ as $k\rightarrow \infty$.
Regarding $I^k_3$, it follows that
\begin{align*}
  I^k_3\leq& C(\eta,\psi)\|V\|_{L^2([0,T];L^1(\mathbb{T}^d))}\|\rho^{\eta,\gamma^k}-\rho^{\eta}\|_{L^2([0,T];L^2(\mathbb{T}^d))}\rightarrow0,
\end{align*}
as $k\rightarrow\infty$. Then (\ref{kk-19}) is shown. Consequently, $\rho^{\eta}$ is a weak solution of (\ref{rrr-3}).
The preservation of the $L^1(\mathbb{T}^d)$-norm follows from (\ref{L1}), (\ref{qq-19-1}), and (\ref{rrr-1}) follows from (\ref{rrr-4}), (\ref{qq-18-1}) and the weak lower-semicontinuity of the Sobolev norm. This completes the proof.

\end{proof}

In the following, we aim to pass the limit $\eta\rightarrow 0$. Closely following \cite[Proposition 19]{FG23} and \cite[Proposition 5.4]{WWZ22}, we can make an entropy dissipation estimate for the weak solution of (\ref{rrr-3}). Define
\begin{equation*}
\Psi(\zeta)=\int_0^{\zeta}\log(\zeta')d\zeta'=\zeta\log \zeta-\zeta.
\end{equation*}
\begin{proposition}\label{prp-1}
Let $V$ satisfy Assumption (A1). Let $\rho_0\in L^{\infty}(\mathbb{T}^d)$ be nonnegative and $g\in L^2(\mathbb{T}^d\times[0,T];\mathbb{R}^d)$. Let $\rho^{\eta}$ be the weak solution of (\ref{rrr-3}) with initial data $\rho_0$, then  there exists a constant $C>0$ independent of $\eta$  such that for almost every $t\in[0,T]$,
\begin{align}\label{r-4}
\int_{\mathbb{T}^d}\Psi(\rho^{\eta}(x,t))dx
+\int_0^t\|\nabla \sqrt{\rho^{\eta}}\|^2_{L^2(\mathbb{T}^d)}ds\leq  \int_{\mathbb{T}^d}\Psi(\rho_0)dx+ c(\rho_0,V)(\|g\|_{L^2(\mathbb{T}^d\times[0,T];\mathbb{R}^d)}^2+1).
\end{align}
\end{proposition}


%
%

Based on the entropy dissipation estimate, we get the following result.
\begin{corollary}
Let $V$ satisfy Assumption (A1). Let $\rho_0\in L^{\infty}(\mathbb{T}^d)$ be nonnegative and $g\in L^2(\mathbb{T}^d\times[0,T];\mathbb{R}^d)$. Let $\rho^{\eta}$ be the weak solution of (\ref{rrr-3}) with initial data $\rho_0$. Then for any $r\in [1,1+\frac{2}{d}]$,
\begin{align}\label{eqq-2}
 \int^T_0\|\rho^{\eta}\|^{r}_{L^{r}(\mathbb{T}^d)}dt
 \leq & c(d,\|\rho_0\|_{L^{1}(\mathbb{T}^{d})} ) \Big[\int_{\mathbb{T}^d}\Psi(\rho_0)dx+ c(d,\rho_0,V)(\|g\|_{L^2(\mathbb{T}^d\times[0,T];\mathbb{R}^d)}^2+1)\Big].
\end{align}
\end{corollary}

\begin{proof}
For simplicity, denote by $\rho=:\rho^{\eta}$.
With the aid of (\ref{r-12}),  we have
\begin{align*}
 \int^T_0\|\rho\|^{r}_{L^{r}(\mathbb{T}^d)}dt=\int^T_0\|\sqrt{\rho}\|^{2r}_{L^{2r}(\mathbb{T}^d)}dt\leq &c(d)\int^T_0\|\nabla\sqrt{\rho}\|^{rd(1-1/r)}_{L^{2}(\mathbb{T}^{d})}
\|\sqrt{\rho}\|^{2r(1-\frac{d}{2}(1-1/r))}_{L^{2}(\mathbb{T}^{d})}dt\\
 \leq& c(d)\|\rho_0\|^{r(1-\frac{d}{2}(1-1/r))}_{L^{1}(\mathbb{T}^{d})}
 \int^T_0\|\nabla\sqrt{\rho}\|^{rd(1-1/r)}_{L^{2}(\mathbb{T}^{d})}
dt.
\end{align*}
When $r\in [1,1+\frac{2}{d}]$, it holds that $rd(1-1/r)\leq 2$. As a consequence of (\ref{r-4}), we get (\ref{eqq-2}).

\end{proof}

Recall that the finite entropy class ${\rm{Ent}}(\mathbb{T}^d)$ is defined by (\ref{t-16}). The following result states the existence of weak solutions to the skeleton equation (\ref{skeleton}).
\begin{proposition}\label{prp-2}
Assume that $ V$ satisfies Assumption (A1). Let
 $\rho_0\in \rm{Ent}(\mathbb{T}^d)$ and $g\in L^2([0,T]\times\mathbb{T}^d;\mathbb{R}^d)$. Then there exists a nonnegative weak solution $\rho\in L^{\infty}([0,T];L^1(\mathbb{T}^d))$ to (\ref{skeleton})
in the sense of Definition \ref{dfn-1}. Furthermore, for almost every $t\in [0,T]$,
\begin{equation}\label{eqq-7}
\|\rho(t)\|_{L^1(\mathbb{T}^d)}=\|\rho_0\|_{L^1(\mathbb{T}^d)},
\end{equation}
and
\begin{align}\label{eqq-8}
\int_{\mathbb{T}^d}\Psi(\rho(x,t))dx
+\int_0^t\|\nabla \sqrt{\rho}\|^2_{L^2(\mathbb{T}^d)}ds
\leq & \int_{\mathbb{T}^d}\Psi(\rho_0)dx+ c(d,\rho_0,V)(\|g\|_{L^2(\mathbb{T}^d\times[0,T];\mathbb{R}^d)}^2+1).
\end{align}
\end{proposition}
\begin{proof}
We firstly prove the existence of weak solutions to (\ref{skeleton}) for nonnegative initial data $\rho_0\in L^{\infty}(\mathbb{T}^d)$. For any $\eta\in(0,1)$, let $\rho^{\eta}$ be a weak solution to (\ref{rrr-3}) constructed in Proposition \ref{existence-rhoeta} with the initial data $\rho_0$ and the control $g$.
It follows from (\ref{r-4}) and (\ref{eq-12}) that there exists $c>0$ independent of $\eta$ such that
\begin{align}\label{qq-r-6}
  \|\sqrt{\rho^{\eta}}\|^2_{L^2([0,T];H^1(\mathbb{T}^d))}\leq&  \int_{\mathbb{T}^d}\Psi(\rho_0)dx+ C(d,\rho_0,V)(\|g\|_{L^2(\mathbb{T}^d\times[0,T];\mathbb{R}^d)}^2+1).
\end{align}

For any $s>\frac{d}{2}+1$, with the aid of (\ref{t-14}), we get
\begin{align*}\notag
  \|\partial_t \rho^{\eta}\|_{L^1([0,T];H^{-s}(\mathbb{T}^d))}
\leq& 2 \|\sqrt{\rho^{\eta}}\nabla \sqrt{\rho^{\eta}}\|_{L^1([0,T];L^1(\mathbb{T}^d))}
+c^2_1\|\rho^{\eta}V\ast\rho^{\eta}\|_{L^1([0,T];L^1(\mathbb{T}^d))}\\
&+c_1\|g\sqrt{\rho^{\eta}}\|_{L^1([0,T];L^1(\mathbb{T}^d))},
\end{align*}
where $c_1$ is the constant appeared in Lemma \ref{lem-2}. Thanks to H\"{o}lder inequality, combining with (\ref{eq-12}) and (\ref{r-4}), it follows that
\begin{align*}
  \|\sqrt{\rho^{\eta}}\nabla \sqrt{\rho^{\eta}}\|_{L^1([0,T];L^1(\mathbb{T}^d))}\leq \|\rho_0\|_{L^1(\mathbb{T}^d)}T+ \int_{\mathbb{T}^d}\Psi(\rho_0)dx+ c(d,\rho_0,V)(\|g\|_{L^2(\mathbb{T}^d\times[0,T];\mathbb{R}^d)}^2+1).
\end{align*}
Regarding the kernel term, with the aid of (\ref{r-10}), (\ref{eq-12}) and (\ref{r-4}), it gives
\begin{align*}
\|\rho^{\eta}V\ast\rho^{\eta}\|_{L^1([0,T];L^1(\mathbb{T}^d))}
\leq& \int^T_0\|\sqrt{\rho^{\eta}}\|_{L^2(\mathbb{T}^d)}\|\sqrt{\rho^{\eta}}V\ast\rho^{\eta}\|_{L^2(\mathbb{T}^d)}dt\\
\leq & c(d,\|\rho_0\|_{L^1(\mathbb{T}^d)})\int^T_0\|\nabla\sqrt{\rho^{\eta}}\|^{\frac{d}{p}}_{L^{2}(\mathbb{T}^{d})}
\|V\|_{L^p(\mathbb{T}^d)}dt\\
\leq& C(d,\rho_0,V)\Big(\int_{\mathbb{T}^d}\Psi(\rho_0)dx+ c(d,\rho_0,V)(\|g\|_{L^2(\mathbb{T}^d\times[0,T];\mathbb{R}^d)}^2+1)\Big).
\end{align*}
Here we have used $p^{*}>\frac{2p}{p-d}>\frac{2p}{2p-d}$.
Based on the above, there exists $c>0$ independent of $\eta$ such that
\begin{align}\label{kk-16}
  \|\partial_t \rho^{\eta}\|_{L^1([0,T];H^{-s}(\mathbb{T}^d))}
\leq& C(d,\rho_0,V)\Big(\int_{\mathbb{T}^d}\Psi(\rho_0)dx+ c(d,\rho_0,V)(\|g\|_{L^2(\mathbb{T}^d\times[0,T];\mathbb{R}^d)}^2+1)\Big).
\end{align}
Since estimates (\ref{eq-12})-(\ref{kk-16}) are uniform on $\eta$, we are able to apply Lemma \ref{lem-3} to $\{\rho^{\eta}\}_{\eta\in (0,1)}$. Then, there exists $\rho\in L^1([0,T];L^1(\mathbb{T}^d))$, with $\nabla\sqrt{\rho}\in L^2([0,T];L^2(\mathbb{T}^d;\mathbb{R}^d))$ such that,
passing to a subsequence $\{\rho^{\eta^k}\}_{k\geq 1}$ satisfying
\begin{align}\label{qq-19}
&\rho^{\eta^k}\rightarrow \rho\ {\rm{almost\ everywhere\ and\ strongly\ in\ }} L^1([0,T];L^1(\mathbb{T}^d)),\notag\\
&\sqrt{\rho^{\eta^k}}\rightharpoonup \sqrt{\rho}\ {\rm{weakly\ in\ }} L^2([0,T];H^1(\mathbb{T}^d)).
 \end{align}
Owing to (\ref{eq-12}) and (\ref{qq-19}), we have
 \begin{equation}\label{kk-14}
\|\rho\|_{L^{\infty}([0,T];L^1(\mathbb{T}^d))}=\|\rho_0\|_{L^1(\mathbb{T}^d)}.
\end{equation}
With the aid of (\ref{qq-r-6}) and (\ref{qq-19}), by the weak lower-semicontinuity of the Sobolev norm, it gives
\begin{align}\label{kk-15}
  \|\sqrt{\rho}\|^2_{L^2([0,T];H^1(\mathbb{T}^d))}\leq  \int_{\mathbb{T}^d}\Psi(\rho_0)dx+ c(d,\rho_0,V)(\|g\|_{L^2(\mathbb{T}^d\times[0,T];\mathbb{R}^d)}^2+1). .
\end{align}

Taking into account (\ref{kk-14})-(\ref{kk-15}), similar to the proof of (\ref{eqq-2}) and (\ref{t-17-1}), for $r\in [1,1+\frac{2}{d}]$, it yields
\begin{align}\label{s-53}
\rho\in L^r([0,T];L^r(\mathbb{T}^d)),\quad {\rm{and}}\quad \nabla\cdot (\rho V\ast \rho)\in L^1([0,T];L^1(\mathbb{T}^d)).
\end{align}
In addition, we claim that as $ k\rightarrow \infty$,
 \begin{align}\label{rr-6}
 \sigma^{\eta^k}(\rho^{\eta^k})\rightarrow\rho \ {\rm{strongly\ in\ }} L^1([0,T];L^1(\mathbb{T}^d)).
 \end{align}
Indeed, by (\ref{t-14}), for any $M>0$, there exists a constant $c_1$ independent of $k$ such that
  \begin{align}\notag
\int^T_0 \| \sigma^{\eta^k}(\rho^{\eta^k})-\rho\|_{L^1}dt
 &\leq  \int^T_0\| \sigma^{\eta^k}(\rho^{\eta^k})-\rho^{\eta^k}\|_{L^1}dt
 + \int^T_0\|\rho^{\eta^k}-\rho\|_{L^1}dt\\ \notag
&\leq  \int_{\{(x,t)\in \mathbb{T}^d\times [0,T]: 0<\rho^{\eta^k}(x,t)\leq M\}}|\sigma^{\eta^k}(\rho^{\eta^k})-\rho^{\eta^k}|dxdt\\
\notag
& + \int_{\{(x,t)\in\mathbb{T}^d\times [0,T]: \rho^{\eta^k}(x,t)> M\}}|\sigma^{\eta^k}(\rho^{\eta^k})-\rho^{\eta^k}|dxdt+ \int^T_0\|\rho^{\eta^k}-\rho\|_{L^1(\mathbb{T}^d)}dt\\ \notag
&\leq  \int_{\{(x,t)\in \mathbb{T}^d\times [0,T]: 0<\rho^{\eta^k}(x,t)\leq M\}}|\sigma^{\eta^k}(\rho^{\eta^k})-\rho^{\eta^k}|dxdt\\
\label{s-51}
 &+c_1\int_{\{(x,t)\in\mathbb{T}^d\times [0,T]: \rho^{\eta^k}(x,t)> M\}}|\rho^{\eta^k}|dxdt+ \int^T_0\|\rho^{\eta^k}-\rho\|_{L^1(\mathbb{T}^d)}dt.
 \end{align}
By (\ref{eqq-2}), for some $r\in (1,1+\frac{2}{d}]$, we have
\begin{align*}
\int_{\{(x,t)\in\mathbb{T}^d\times [0,T]: \rho^{\eta^k}(x,t)> M\}}|\rho^{\eta^k}|dxdt
\leq& \frac{1}{M^{r-1}}\int^T_0\|\rho^{\eta^k}\|^r_{L^r(\mathbb{T}^d)}dt\\
\leq& \frac{1}{M^{r-1}}c(d,\|\rho_0\|_{L^{1}(\mathbb{T}^{d})} ) \Big[\int_{\mathbb{T}^d}\Psi(\rho_0)dx+ c(c_1)\|g\|_{L^2(\mathbb{T}^d\times[0,T];\mathbb{R}^d)}^2\\
& + C(c_1,d,\|\rho_0\|_{L^1(\mathbb{T}^d)}, \|V\|_{L^{p^*}([0,T];L^p(\mathbb{T}^d))})\Big]\\
\rightarrow&  0,\quad {\rm{as\ }} M\rightarrow \infty.
\end{align*}
Hence, for any $\iota>0$, there exists $M_0$ independent of $k$ such that
\begin{align*}
  2\int_{\{(x,t)\in\mathbb{T}^d\times [0,T]: \rho^{\eta^k}(x,t)> M\}}|\rho^{\eta^k}|dxdt<\frac{\iota}{3}.
\end{align*}
Taking $M=M_0$ in (\ref{s-51}),  we obtain
\begin{align*}
\int^T_0 \| \sigma^{\eta^k}(\rho^{\eta^k})-\rho\|_{L^1}dt
\leq &\frac{\iota}{3} +\int_{\{(x,t)\in \mathbb{T}^d\times [0,T]: 0<\rho^{\eta^k}(x,t)\leq M_0\}}|\sigma^{\eta^k}(\rho^{\eta^k})-\rho^{\eta^k}|dxdt
+ \int^T_0\|\rho^{\eta^k}-\rho\|_{L^1(\mathbb{T}^d)}dt.
 \end{align*}
For such a constant $M_0$, by using (\ref{s-38}), we get
\begin{align*}
  \int_{\{(x,t)\in \mathbb{T}^d\times [0,T]: 0<\rho^{\eta^k}(x,t)\leq M_0\}}|\sigma^{\eta^k}(\rho^{\eta^k})-\rho^{\eta^k}|dxdt\rightarrow 0, \quad {\rm{as}}\  k\rightarrow \infty,
\end{align*}
This, together with (\ref{qq-19}), yields that (\ref{rr-6}) holds.
In the following, we aim to show that the limit $\rho$ is a weak solution of (\ref{skeleton}). This can be obtained by passing to the limits of the weak formulation of (\ref{rrr-3}). We highlight the proof of passage to the limits for the kernel term. In view of (\ref{s-53}), by integration by parts formula, for every $\psi\in C^{\infty}(\mathbb{T}^d)$, it suffices to prove
 \begin{align}\label{eqqq-1}
\lim_{k\rightarrow \infty}\int^T_0\int_{\mathbb{T}^d}(\sigma^{\eta^k}(\rho^{\eta^k})V\ast \rho^{\eta^k}-\rho V\ast\rho )\cdot \nabla \psi dxdt= 0.
 \end{align}
Clearly, we have
\begin{align}\notag
&\int^T_0\int_{\mathbb{T}^d}(\sigma^{\eta^k}(\rho^{\eta^k})V\ast \rho^{\eta^k}-\rho V\ast\rho )\cdot \nabla \psi dxdt\\
\notag
=&\int^T_0\int_{\mathbb{T}^d}(\sigma^{\eta^k}(\rho^{\eta^k})-\rho)V\ast \rho^{\eta^k}\cdot \nabla \psi dxdt
+\int^T_0\int_{\mathbb{T}^d}\rho V\ast (\rho^{\eta^k}-\rho )\cdot \nabla \psi dxdt\\
\label{kk-17}
=:& J^k_1+J^k_2.
\end{align}
Thanks to the fact that $V^{\gamma}$ satisfies (\ref{kk-65})-(\ref{r-18}), by H\"older's inequality and convolution Young's inequality, we get
\begin{align*}\notag
J^k_1\leq&\int_0^T\|(\sigma^{\eta^k}(\rho^{\eta^k})-\rho)V\ast\rho^{\eta^k}\|_{L^1(\mathbb{T}^d)}
\|\nabla\psi\|_{L^{\infty}(\mathbb{T}^d)}dt\\
\notag
\leq &  \int_0^T\|\sigma^{\eta^k}(\rho^{\eta^k})|V-V^{\gamma}|\ast\rho^{\eta^k}\|_{L^1(\mathbb{T}^d)}\|\nabla\psi\|_{L^{\infty}(\mathbb{T}^d)}dt
\\ \notag
& + \int_0^T\|\rho|V-V^{\gamma}|\ast\rho^{\eta^k}\|_{L^1(\mathbb{T}^d)}\|\nabla\psi\|_{L^{\infty}(\mathbb{T}^d)}dt
\\ \notag
& + \int_0^T\|(\sigma^{\eta^k}(\rho^{\eta^k})-\rho)V^{\gamma}\ast\rho^{\eta^k}\|_{L^1(\mathbb{T}^d)}\|\nabla\psi\|_{L^{\infty}(\mathbb{T}^d)}dt\\
\notag
=: & J^k_{11}+J^k_{12}+J^k_{13},
\end{align*}
Since $p>d$, by (\ref{r-10}), (\ref{eq-12}) and (\ref{qq-r-6}), it follows that
\begin{align}\label{r-15}
J^k_{11} \leq & c_1\int_0^T\|\sqrt{\rho^{\eta^k}}\sqrt{\rho^{\eta^k}}|V-V^{\gamma}|\ast\rho^{\eta^k}\|_{L^1(\mathbb{T}^d)}\|\nabla\psi\|_{L^{\infty}(\mathbb{T}^d)}dt\\
\notag
\leq & c_1\int_0^T\|\sqrt{\rho^{\eta^k}}\|_{L^2(\mathbb{T}^d)}\|\sqrt{\rho^{\eta^k}}|V-V^{\gamma}|\ast\rho^{\eta^k}\|_{L^2(\mathbb{T}^d)}dt\\
\notag
\leq & c_1\|\rho_0\|^{2-\frac{d}{2p}}_{L^1(\mathbb{T}^d)}c(d)\int_0^T\|\nabla\sqrt{\rho^{\eta^k}}\|^{\frac{d}{p}}_{L^{2}(\mathbb{T}^{d})}\|V-V^{\gamma}\|_{L^p(\mathbb{T}^d)}
dt\\ \notag
\leq& C(c_1,d,\rho_0,V,T) \Big[\int_{\mathbb{T}^d}\Psi(\rho_0)dx+ \|g\|_{L^2(\mathbb{T}^d\times[0,T];\mathbb{R}^d)}^2+1\Big] \int^T_0\|V-V^{\gamma}\|^{p^{*}}_{L^p(\mathbb{T}^d)}dt.
\end{align}
Similar to (\ref{r-15}), by (\ref{kk-14}) and (\ref{kk-15}), we have
\begin{align}\label{r-16}
J^k_{12}\leq& C(c_1,d,\rho_0,V,T) \Big[\int_{\mathbb{T}^d}\Psi(\rho_0)dx+ \|g\|_{L^2(\mathbb{T}^d\times[0,T];\mathbb{R}^d)}^2+1\Big] \int^T_0\|V-V^{\gamma}\|^{p^{*}}_{L^p(\mathbb{T}^d)}dt.
\end{align}
The term $J^k_{13}$ can be estimated as
\begin{align}\notag
J^k_{13}
\leq &\int_0^T\|\sigma^{\eta^k}(\rho^{\eta^k})-\rho\|_{L^1(\mathbb{T}^d)}
\|V^{\gamma}\|_{L^{\infty}(\mathbb{T}^d)}\|\rho^{\eta^k}\|_{L^1(\mathbb{T}^d)}
\|\nabla\psi\|_{L^{\infty}(\mathbb{T}^d)}dt\\
\label{r-17}
\leq& C(\|\rho_0\|_{L^1(\mathbb{T}^d)}, \|V\|_{L^{p^*}([0,T];L^p(\mathbb{T}^d))})\frac{1}{\gamma^{d+1}} \int_0^T\|\sigma^{\eta^k}(\rho^{\eta^k})-\rho\|_{L^1(\mathbb{T}^d)}dt.
\end{align}
Combining (\ref{r-15})-(\ref{r-17}), we obtain
\begin{align}\notag
J^k_1\leq&C(c_1,d,\rho_0,V,T) \Big[\int_{\mathbb{T}^d}\Psi(\rho_0)dx+ \|g\|_{L^2(\mathbb{T}^d\times[0,T];\mathbb{R}^d)}^2+1\Big] \int^T_0\|V-V^{\gamma}\|^{p^{*}}_{L^p(\mathbb{T}^d)}dt\\
\label{s-52}
+&C(\|\rho_0\|_{L^1(\mathbb{T}^d)}, \|V\|_{L^{p^*}([0,T];L^p(\mathbb{T}^d))})\frac{1}{\gamma^{d+1}} \int_0^T\|\sigma^{\eta^k}(\rho^{\eta^k})-\rho\|_{L^1(\mathbb{T}^d)}dt.
\end{align}
Taking into account (\ref{r-18}), (\ref{rr-6}) and (\ref{s-52}), it follows that $J^k_1\rightarrow 0$ as $k\rightarrow \infty$.

The term $J^k_2$ can be treated in a similar way to $J^k_1$. With the aid of (\ref{r-10}), (\ref{eq-12}), (\ref{kk-14}) and (\ref{kk-15}), for $p>d$, we have
\begin{align*}
J^k_2\leq& C(c_1,d,\rho_0,V,T) \Big[\int_{\mathbb{T}^d}\Psi(\rho_0)dx+ \|g\|_{L^2(\mathbb{T}^d\times[0,T];\mathbb{R}^d)}^2+1\Big] \int^T_0\|V-V^{\gamma}\|^{p^{*}}_{L^p(\mathbb{T}^d)}dt\\
+&C(\|\rho_0\|_{L^1(\mathbb{T}^d)}, \|V\|_{L^{p^*}([0,T];L^p(\mathbb{T}^d))})\frac{1}{\gamma^{d+1}} \int_0^T\|\sigma^{\eta^k}(\rho^{\eta^k})-\rho\|_{L^1(\mathbb{T}^d)}dt.
\end{align*}
A repetition of the above method, we have $J^k_2\rightarrow 0$ as $k\rightarrow \infty$.
As a result, (\ref{eqqq-1}) is shown. Thus, thanks to (\ref{kk-14}) and (\ref{kk-15}), we conclude that $\rho$ is a nonnegative weak solution to (\ref{skeleton}).

For the initial data $\rho_0\in {\rm{Ent}} (\mathbb{T}^d)$, let $\{{\rho^n_0:=\rho_0\wedge n}\}_{n\in \mathbb{N}}\subseteq L^{\infty}(\mathbb{T}^d)$ be a sequence with uniformly bounded entropy such that, as $n\rightarrow \infty$, $\rho^n_0\rightarrow\rho_0$ strongly in $L^1(\mathbb{T}^d)$. For every $n\in \mathbb{N}$, let $\rho^n\in L^{\infty}([0,T];L^1(\mathbb{T}^d))$ be a weak solution to  (\ref{skeleton}) satisfying  (\ref{eqq-7}) and (\ref{eqq-8}). Since (\ref{eqq-7}), (\ref{eqq-8}) and (\ref{kk-16}) hold uniformly on $n\geq 1$, we can apply Lemma \ref{lem-3}  to $\{\rho^n\}_{n\geq 1}$. Hence, there exists $\rho\in L^1([0,T];L^1(\mathbb{T}^d))$, $\nabla\sqrt{\rho}\in L^2([0,T];L^2(\mathbb{T}^d;\mathbb{R}^d ))$ such that,
passing to a subsequence still denoted by $n$ satisfying that $\rho^n$ converges to $\rho$ strongly in $L^1([0,T];L^1(\mathbb{T}^d))$.
A repetition of the arguments in the proof of (\ref{eqqq-1}), we can show that $\rho$ is a weak solution of (\ref{skeleton})  with initial data $\rho_0$ and $\rho$ satisfies (\ref{eqq-7}) and (\ref{eqq-8}).
\end{proof}

Based on the above, we will provide a result on the stability of the skeleton equation (\ref{skeleton}), which lays the foundation for the study of large deviations.
\begin{proposition}\label{prp-3}
Assume that $ V$ satisfies Assumptions (A1) and (A2). Let $\rho_0\in \rm{Ent}(\mathbb{T}^d)$.
 For any $N>0$, assume that $\{g_n\}_{n\in\mathbb{N}_+}, g\subseteq L^2([0,T]\times\mathbb{T}^d;\mathbb{R}^d)$ satisfying
\begin{align*}
\sup_{n\geq 1}\int^T_0 \|g_n(s)\|^2_{L^2(\mathbb{T}^d)}ds\leq N,
\quad g_n\rightharpoonup g\ {\rm{weakly\ in\ }} L^2([0,T]\times\mathbb{T}^d;\mathbb{R}^d).
\end{align*}
For every $n\in \mathbb{N}$, let $\rho_n$ be the weak solution of the skeleton equation (\ref{skeleton}) with control $g_n$ and initial data $\rho_0$. Let $\rho$ be the weak solution of the skeleton equation to (\ref{skeleton}) with control $g$ and initial data $\rho_0$. Then
\begin{align}\label{qq-r-7}
\rho_n\rightarrow\rho\ {\rm{strongly\ in\ }} L^1([0,T];L^1(\mathbb{T}^d)).
\end{align}
as $n\rightarrow \infty$.
\end{proposition}
\begin{proof}

Due to Proposition \ref{L1uniq} and Theorem \ref{thm-1}, $\{\rho_n\}_{n\in \mathbb{N}}$ and $\rho$ are the unique renormalized kinetic solutions to (\ref{skeleton}) with $g_n$ and $g$, respectively. With the help of (\ref{eqq-7}), (\ref{eqq-8}) and (\ref{kk-16}), by Lemma \ref{lem-3}, it follows that $\{\rho_n\}_{n\in \mathbb{N}}$ is relatively pre-compact on $L^{1}([0,T];L^1(\mathbb{T}^d))$ and $\{\sqrt{\rho_n}\}_{n\in \mathbb{N}}$ is relatively pre-compact on $L^{2}([0,T];L^2(\mathbb{T}^d))$. Thus, there exists $\rho^*\in L^{1}([0,T];L^1(\mathbb{T}^d))$, such that, passing to a subsequence (not relabeled) such that,
\begin{align*}
&\rho_{n}\rightarrow \rho^{*}\ {\rm{almost\ everywhere\ and\ strongly\ in\ }} L^1([0,T];L^1(\mathbb{T}^d)),\\
&\sqrt{\rho_{n}}\rightarrow \sqrt{\rho^*}\ {\rm{strongly\ in}}\ L^2([0,T];L^2(\mathbb{T}^d)).
 \end{align*}
To achieve (\ref{qq-r-7}), it suffices to show that $\rho^{*}$ is a weak solution of the skeleton equation (\ref{skeleton}) with control $g$ and initial data $\rho_0$. Compared with the proof of Proposition \ref{prp-2}, the passage to the limits for the kernel term can be obtained in the same procedure. Moreover, the passage to the limits for the control term can be obtained by using H\"older's inequality, the $L^1([0,T];L^1(\mathbb{T}^d))$ convergence of $\rho_n$ and the weak convergence of $g_n$. This completes the proof.
\end{proof}

\section{Large deviations and the weak convergence approach}\label{sec-4}
\subsection{Preliminaries on large deviations }
We start with a brief account of notions of large deviations.
Let $\mathcal{E}$ be a Polish space. Let $\{X^\varepsilon\}_{\varepsilon>0}$ be a family of $\mathcal{E}$-valued random variables defined on a given probability space $(\Omega, \mathcal{F}, \mathbb{P})$.
\begin{definition}
 A function $I: \mathcal{E}\rightarrow [0,\infty]$ is called a rate function if $I$ is lower semicontinuous.
\end{definition}
\begin{definition}(Large deviation principle)
The sequence $\{X^\varepsilon\}$ is said to satisfy the large deviation principle with rate function $I$ if for each Borel subset $A$ of $\mathcal{E}$
      \begin{align*}
      -\inf_{x\in A^o}I(x)\leq \lim \inf_{\varepsilon\rightarrow 0}\varepsilon \log \mathbb{P}(X^\varepsilon\in A)\leq \lim \sup_{\varepsilon\rightarrow 0}\varepsilon \log \mathbb{P}(X^\varepsilon\in A)\leq -\inf_{x\in \bar{A}}I(x),
      \end{align*}
      where $A^o$ and $\bar{A}$ denote the interior and closure of $A$ in $\mathcal{E}$, respectively.
\end{definition}

Let $U,\mathcal{U}$ be two Hilbert spaces with Hilbert-Schmidt embedding $U\subset \mathcal{U}$. Suppose $W(t)$ is a $U$-cylindrical Wiener process defined on a filtered probability space $(\Omega, \mathcal{F},\{\mathcal{F}_t\}_{t\in [0,T]}, \mathbb{P} )$, then the paths of $W$ take values in $C([0,T];\mathcal{U})$. Now we define
\begin{align*}
\mathcal{A}&:=\{\phi: \phi\ is\ a\ U\text{-}valued\ \{\mathcal{F}_t\}\text{-}predictable\ process\ such\ that \ \int^T_0 |\phi(s)|^2_Uds<\infty\ \mathbb{P}\text{-}a.s.\};\\
S_N&:=\{ h\in L^2([0,T];U): \int^T_0 |h(s)|^2_Uds\leq N\};\\
\mathcal{A}_N&:=\{\phi\in \mathcal{A}: \phi(\omega)\in S_N,\ \mathbb{P}\text{-}a.s.\}.
\end{align*}
Here and in the sequel of this paper, we will always refer to the weak topology on the set $S_N$.

Suppose for each $\varepsilon>0, \mathcal{G}^{\varepsilon}: C([0,T];\mathcal{U})\rightarrow \mathcal{E}$ is a measurable map and let $X^{\varepsilon}:=\mathcal{G}^{\varepsilon}(W)$. Now, we list below sufficient conditions for the large deviation principle of the sequence $X^{\varepsilon}$ as $\varepsilon\rightarrow 0$.
\begin{description}
  \item[\textbf{Condition A} ] There exists a measurable map $\mathcal{G}^0: C([0,T];\mathcal{U})\rightarrow \mathcal{E}$ such that the following conditions hold.
\end{description}
\begin{description}
  \item[(a)] For every $N<\infty$, let $\{g^{\varepsilon}: \varepsilon>0\}$ $\subseteq \mathcal{A}_N$. If $g^{\varepsilon}$ converges to $g$ as $S_N$-valued random elements in distribution, then $\mathcal{G}^{\varepsilon}(W(\cdot)+\frac{1}{\sqrt{\varepsilon}}\int^{\cdot}_{0}g^\varepsilon(s)ds)$ converges in distribution to $\mathcal{G}^0(\int^{\cdot}_{0}g(s)ds)$.
  \item[(b)] For every $N<\infty$, the set $K_N=\{\mathcal{G}^0(\int^{\cdot}_{0}g(s)ds): g\in S_N\}$ is a compact subset of $\mathcal{E}$.
\end{description}

The following result is due to Budhiraja et al. in \cite{BD}.
\begin{theorem}\label{thm-7}
If $\{\mathcal{G}^{\varepsilon}\}$ satisfies {condition A}, then $X^{\varepsilon}$ satisfies the large deviation principle on $\mathcal{E}$ with the
following good rate function $I$ defined by
\begin{eqnarray}\label{equ-27-1}
I(f)=\inf_{\{g\in L^2([0,T];U): f= \mathcal{G}^0(\int^{\cdot}_{0}g(s)ds)\}}\Big\{\frac{1}{2}\int^T_0|g(s)|^2_{U}ds\Big\},\ \ \forall f\in\mathcal{E}.
\end{eqnarray}
By convention, $I(f)=\infty$, if  $\{g\in L^2([0,T];U): f= \mathcal{G}^0(\int^{\cdot}_{0}g(s)ds)\}=\emptyset.$
\end{theorem}

\subsection{Statement of the main result}


The purpose of this section is to prove large deviations for the renormalized kinetic solution of (\ref{s-1}) along a joint limit for the noise intensity $\varepsilon\rightarrow 0$ and the ultra-violet cut-off $K(\varepsilon)\rightarrow \infty$.

 Theorem \ref{thm-14} implies that there exists a measurable mapping
 \begin{align*}
   \mathcal{G}^{\varepsilon,K}:  C([0,T];\mathbb{R}^{\infty})\rightarrow L^{1}([0,T];L^1(\mathbb{T}^d))
 \end{align*}
  such that for every $\rho_0\in {\rm{Ent}} (\mathbb{T}^d)$,
  \begin{align}\label{s-6}
    \rho^{\varepsilon,K}=\mathcal{G}^{\varepsilon,K}\Big((B^k,W^k)_{|k|\leq K}\Big),\quad \mathbb{P}-a.s.,
  \end{align}
where $\rho^{\varepsilon,K}$ is the unique stochastic kinetic solution of (\ref{s-2}) with initial data $\rho_0$.

For any $g\in L^2([0,T];L^2(\mathbb{T}^d;\mathbb{R}^d))$, consider the skeleton equation
\begin{eqnarray}\label{s-54}
\left\{
  \begin{array}{ll}
   \partial_t\rho=\Delta\rho -\nabla\cdot (\rho V\ast\rho)-\nabla\cdot (\sqrt{\rho}g) , &  \\
    \rho(\cdot,0)=\rho_0\in {\rm{Ent}} (\mathbb{T}^d), &
  \end{array}
\right.
\end{eqnarray}
which has been studied in Sections \ref{sec-3} and \ref{sec-4}.
Combining the results of Proposition \ref{L1uniq} and Proposition \ref{prp-2}, (\ref{s-54}) admits a unique weak solution $\rho^g$. It implies that there exists a measurable mapping $\mathcal{G}^0: C([0,T];\mathbb{R}^{\infty})\rightarrow L^1([0,T];L^1(\mathbb{T}^d))$ such that
\begin{align}\label{k-30}
\mathcal{G}^0\Big(\Big(\int^{\cdot}_0 g_k(s)ds,\int^{\cdot}_0 g^{\sharp}_k(s)ds\Big)_{k\in \mathbb{N}}\Big):=\rho^g(\cdot),
\end{align}
where $g_k(s)=( g(\cdot,s), \cos(k\cdot ))$ and $g^{\sharp}_k(s)=(g(\cdot,s), \sin(k\cdot ))$.


Our main result of this paper reads as follows.
\begin{theorem}\label{thm-2}
Assume that $V$ satisfies Assumptions (A1) and (A2). Let $\rho_0\in  {\rm{Ent}} (\mathbb{T}^d)$, $\varepsilon\in (0,1)$, $K\in \mathbb{N}$ satisfying $\varepsilon K^{d+2}(\varepsilon)\rightarrow 0, K(\varepsilon)\rightarrow \infty$. Then the solutions $\{\rho^{\varepsilon,K(\varepsilon)}(\rho_0)\}_{\varepsilon\in (0,1)}$ of (\ref{s-2}) satisfy a large deviation principle on $L^1([0,T];L^1(\mathbb{T}^d))$ with rate function $I: L^1([0,T];L^1(\mathbb{T}^d))\rightarrow [0,\infty]$ defined by
\begin{align}\label{kk-1}
  I(\rho)=\frac{1}{2}\inf\Big\{\|g\|^2_{L^2([0,T];L^2(\mathbb{T}^d;\mathbb{R}^d))}: \partial_t \rho=\Delta \rho-\nabla\cdot (\rho V\ast \rho)-\nabla\cdot (\sqrt{\rho}g)\ {\rm{and}}\ \rho(\cdot, 0)=\rho_0\Big\}.
\end{align}

\end{theorem}
Based on Theorem \ref{thm-7}, it suffices to prove {\textbf{Condition A}} holds. Clearly, {\textbf{(b)}} in {\textbf{Condition A}} is a direct consequence of Proposition \ref{prp-3}. Therefore, it is sufficient to verify {\textbf{(a)}} in {\textbf{Condition A}}.

\section{The stochastic controlled equation}\label{sec-5}
As stated in the above section, we need to verify {\textbf{(a)}} in {\textbf{Condition A}} for the proof of large deviation principle.  For any $N<\infty$, let $\{g^{\varepsilon}: 0<\varepsilon<1\}$ $\subset \mathcal{A}_N$.
Then, we have
\begin{eqnarray}\label{k-37}
  \sup_{\varepsilon\in (0,1)}\|g^{\varepsilon}\|^2_{L^{\infty}(\Omega;L^2(\mathbb{T}^d\times [0,T];\mathbb{R}^d) )}\leq N.
\end{eqnarray}
We focus on the following conservative stochastic PDE with stochastic control term, which is called stochastic controlled equation
  \begin{align}\notag
 \partial_t \bar{\rho}^{\varepsilon,K,g^{\varepsilon}}=&\Delta \bar{\rho}^{\varepsilon,K,g^{\varepsilon}}-\nabla\cdot  (\bar{\rho}^{\varepsilon,K,g^{\varepsilon}} V\ast \bar{\rho}^{\varepsilon,K,g^{\varepsilon}})-\sqrt{\varepsilon}\nabla\cdot  (\sqrt{\bar{\rho}^{\varepsilon,K,g^{\varepsilon}}} \xi^K)\\
 \label{s-3}
 & +\frac{\varepsilon N_K}{8}\nabla\cdot ({(\bar{\rho}^{\varepsilon,K,g^{\varepsilon}})}^{-1}\nabla \bar{\rho}^{\varepsilon,K,g^{\varepsilon}})-\nabla\cdot (\sqrt{\bar{\rho}^{\varepsilon,K,g^{\varepsilon}}} P_K g^{\varepsilon}),
 \end{align}
with $ \bar{\rho}^{\varepsilon,K,g^{\varepsilon}}(0)=\rho_0$. Here, $K=K(\varepsilon)\rightarrow \infty$, as $\varepsilon\rightarrow 0$ and $P_Kg^{\varepsilon}$ denotes the Fourier projection of the random control $g^{\varepsilon}$ onto the span $\{\sin (k\cdot x), \cos (k\cdot x)\}_{\{|k|\leq K\}}$.
\textbf{For simplicity, we denote by $\bar{\rho}^{\varepsilon}:=\bar{\rho}^{\varepsilon,K(\varepsilon),g^{\varepsilon}}$.}


\subsection{Well-posedness of the stochastic controlled equation}
Firstly, we introduce the following definition based on the solutions to (\ref{s-2}) and (\ref{skeleton}).
\begin{definition}\label{dfn-n-1}
Let $\rho_0\in  {\rm{Ent}} (\mathbb{T}^d)$, $\varepsilon\in (0,1)$ and $K\in \mathbb{N}$. A nonnegative function $\bar{\rho}^{\varepsilon}\in L^{1}(\Omega\times[0,T];L^1(\mathbb{T}^d))$ is called a stochastic renormalized kinetic solution of (\ref{s-3}) with initial data $\rho_0$ if $\rho$ is almost surely continuous $L^1(\mathbb{T}^d)-$valued $\mathcal{F}_t-$predicatable function and satisfies the following properties.
\begin{enumerate}
  \item Preservation of mass: almost surely, for almost every $t\in [0,T]$,
  \begin{align}\label{s-28}
  \|\bar{\rho}^{\varepsilon}(\cdot,t)\|_{L^1(\mathbb{T}^d)}=\|\rho_0\|_{L^1(\mathbb{T}^d)}.
  \end{align}
  \item Regularity:
  $\nabla\sqrt{\bar{\rho}^{\varepsilon}}\in L^2(\Omega\times[0,T];L^2(\mathbb{T}^d;\mathbb{R}^d))$, $V\ast\bar{\rho}^{\varepsilon}\in L^2(\Omega\times[0,T];L^2(\mathbb{T}^d))$.

 Furthermore, there exists a nonnegative kinetic measure $\bar{q}^{\varepsilon}$ satisfying the following properties.
  \item Regularity: for any nonnegative $\psi\in C^{\infty}_c(\mathbb{T}^d\times (0,\infty))$, we have $\bar{q}^{\varepsilon}(\psi)\geq \bar{o}^{\varepsilon}(\psi)$, where
  $\bar{o}^{\varepsilon}: \Omega\rightarrow \mathcal{M}^+(\mathbb{T}^d\times (0,\infty)\times [0,T])$ is defined by
\begin{align*}
  \bar{o}^{\varepsilon}(\psi):=4\int^T_0\int_{\mathbb{T}^d}\int_{\mathbb{R}}\psi(x,\xi)
\delta_0(\xi-\bar{\rho}^{\varepsilon}) \xi |\nabla\sqrt{\bar{\rho}^{\varepsilon}}|^2d\xi dxdt.
\end{align*}
\item Vanishing at infinity: we have
     \begin{align}\label{qq-r-8}
       \liminf_{M\rightarrow \infty}\mathbb{E}\Big[\bar{q}^{\varepsilon}(\mathbb{T}^d\times [0,T]\times [M,M+1])\Big]=0.
     \end{align}
\item The equation: the pair $(\bar{\chi}^{\varepsilon}(x,\xi,t)=I_{0<\xi<\bar{\rho}^{\varepsilon}(x,t)},\bar{q}^{\varepsilon})$ satisfies, for every $\psi\in C^{\infty}_c(\mathbb{T}^d\times (0,\infty))$, for almost every $t\in [0,T]$, $\mathbb{P}-$a.s.,
     \begin{align}\notag
      &\int_{\mathbb{R}}\int_{\mathbb{T}^d}\bar{\chi}^{\varepsilon}(x,\xi,t)\psi(x,\xi)dxd\xi
      =\int_{\mathbb{R}}\int_{\mathbb{T}^d}\bar{\chi}(\rho_0(x),\xi)\psi(x,\xi)dxd\xi
      +\int^t_0\int_{\mathbb{R}}\int_{\mathbb{T}^d}\bar{\chi}^{\varepsilon}\Delta_x\psi dxd\xi ds\\ \notag
      -&\int^t_0\int_{\mathbb{T}^d}\psi(x,\bar{\rho}^{\varepsilon})\nabla\cdot(\bar{\rho}^{\varepsilon} V\ast\bar{\rho}^{\varepsilon})dxds
      +2\int^t_0\int_{\mathbb{T}^d}\bar{\rho}^{\varepsilon}\nabla \sqrt{\bar{\rho}^{\varepsilon}}\cdot P_Kg^{\varepsilon} (\partial_{\xi}\psi)(x,\bar{\rho}^{\varepsilon})dxdr\\ \notag
      +& \int^t_0\int_{\mathbb{T}^d}\sqrt{\bar{\rho}^{\varepsilon}}P_Kg^{\varepsilon}(x,r)\cdot (\nabla  \psi)(x,\bar{\rho}^{\varepsilon})dxdr
      -\frac{\varepsilon N_K}{8}\int^t_0\int_{\mathbb{T}^d}(\bar{\rho}^{\varepsilon})^{-1}\nabla \bar{\rho}^{\varepsilon}
\cdot \nabla \psi(x,\bar{\rho}^{\varepsilon})dxds\\
\label{s-5}
-&\int^t_0\int_{\mathbb{R}}\int_{\mathbb{T}^d}\partial_{\xi}\psi(x,\xi)d\bar{q}^{\varepsilon}+\frac{\varepsilon M_K}{2}\int^t_0\int_{\mathbb{T}^d} \bar{\rho}^{\varepsilon} \partial_{\xi}\psi(x,\bar{\rho}^{\varepsilon})dxds
-\sqrt{\varepsilon}\int^t_0\int_{\mathbb{T}^d}\psi(x,\bar{\rho}^{\varepsilon})\nabla\cdot(\sqrt{\bar{\rho}^{\varepsilon}}d\xi^{K}).
\end{align}
\end{enumerate}
\end{definition}

Applying similar method as in the proof of \cite[Proposition 4.6]{FG24} and using (\ref{ent-1}), we get the kinetic measure also decays at zero and satisfies the other version of vanishing at infinity. That is,
\begin{align}\label{k-28}
		&\lim_{M\rightarrow \infty}M\mathbb{E}\Big[\bar{q}^{\varepsilon}(\mathbb{T}^d\times [0,T]\times [1/M,2/M])\Big]=0,\\
\label{k-29}
&\lim_{M\rightarrow \infty}M^{-1}\mathbb{E}\Big[\bar{q}^{\varepsilon}(\mathbb{T}^d\times [0,T]\times [M,2M])\Big]=0.
	\end{align}


The well-posedness of the stochastic controlled equation (\ref{s-3}) reads as follows.
\begin{theorem}\label{thm-n-1}
Let $\rho_0\in {\rm{Ent}} (\mathbb{T}^d)$. Assume that $V$  satisfies Assumptions (A1) and (A2), then (\ref{s-3}) admits a unique renormalized kinetic solution in the sense of Definition \ref{dfn-n-1} which is a probabilistically strong solution.
\end{theorem}
\begin{proof}
The proof of Theorem \ref{thm-n-1} is divided into several parts. For the uniqueness of solutions to (\ref{s-3}), it can be easily derived by
combining techniques used in Proposition \ref{L1uniq} and in \cite[Theorem 4.7]{FG24}, thus we omit it. To show the existence of renormalized kinetic solutions to (\ref{s-3}), we firstly introduce a regularized equation. Recall from Lemma \ref{lem-2} that there exists a sequence of smooth functions $\{\sigma^{\frac{1}{2},n}\}_{n\geq 1}$ approximating the square root function. We consider the following approximation equation of (\ref{s-3}),
 \begin{align}\notag
 d\rho^{n,\varepsilon}&=\Delta \rho^{n,\varepsilon} dt-\nabla\cdot (\sigma^{ n}(\rho^{n,\varepsilon}) V\ast \rho^{n,\varepsilon})dt-\sqrt{\varepsilon}\nabla\cdot (\sigma^{\frac{1}{2}, n}(\rho^{n,\varepsilon})d\xi^K)\\
 \label{s-8}
     &\quad\quad \quad\quad -\nabla\cdot (\sigma^{\frac{1}{2}, n}(\rho^{n,\varepsilon})P_K g^{\varepsilon})dt+\frac{N_K\varepsilon}{2}\nabla\cdot (|(\sigma^{\frac{1}{2}, n})'(\rho^{n,\varepsilon})|^2\nabla \rho^{n,\varepsilon})dt,
 \end{align}
 with $\rho^{n,\varepsilon}(0)=\rho_0\in L^{\infty}(\mathbb{T}^d)$. Employing
the same argument as used in \cite[Proposition 5.20, Proposition 5.26]{FG24} and \cite[Theorem 5.7, Proposition 5.14]{WWZ22}, we get the existence of probabilistically weak solutions to (\ref{s-8}) and tightness of $(\rho^{n,\varepsilon})_{n\geq1}$ on $L^1([0,T];L^1(\mathbb{T}^d))$. Moreover, the weak solution $\rho^{n,\varepsilon} $ of (\ref{s-8}) is a stochastic renormalized kinetic solution. The final step is to pass to the limits of (\ref{s-3}), which can be achieved in a similar way to \cite[Theorem 6.2]{WWZ22}.

\end{proof}

Based on Theorem \ref{thm-n-1}, Let $\bar{\rho}^{\varepsilon}$ be the unique solution of (\ref{s-3}) with initial data $\rho_0$ and control $g^{\varepsilon}$, and let the map $\mathcal{G}^{\varepsilon,K}$ be defined by (\ref{s-6}). By Girsanov theorem, it follows that
 \begin{align}
  \bar{\rho}^{\varepsilon}=\mathcal{G}^{\varepsilon,K}\left( (B^k,W^k)_{|k|\leq K}+\frac{1}{ \sqrt{\varepsilon}}\Big(\int^{\cdot}_0g^{\varepsilon}_k(s)ds,\int^{\cdot}_0(g^{\varepsilon})^{\sharp}_k(s)ds\Big)_{|k|\leq K}\right).
\end{align}

\subsection{Tightness of the laws of solutions to stochastic controlled equation}\label{sec-12}
We adopt the method from \cite{FG24} and \cite{WWZ22}. As pointed out in \cite[Section 5.3]{FG24}, a stable $W^{\beta,1}([0,T];H^{-s}(\mathbb{T}^{d}))$-estimate is not available for the approximate solutions $\rho^{n,\varepsilon}$. Instead, for every  $\delta\in(0,1)$, we have a stable $W^{\beta,1}([0,T];H^{-s}(\mathbb{T}^{d}))$-estimate
for
 \begin{align*}
 h_{\delta}(\rho^{n,\varepsilon}):=\psi_{\delta}(\rho^{n,\varepsilon}) \rho^{n,\varepsilon}.
 \end{align*}
Here,
$\psi_{\delta}\in C^{\infty}([0,\infty))$ is a smooth nondecreasing function satisfying that $0\le\psi_{\delta}\le1$, $\psi_{\delta}(\xi)=1$ if $\xi\ge\delta$, $\psi_{\delta}(\xi)=0$ if $\xi\le\delta/2$. Moreover, $|\psi_{\delta}'(\xi)|\le c/\delta$ for some $c \in(0, \infty)$ independent of $\delta$.


We also need a new metric on  $L^1([0,T];L^1(\mathbb{T}^d))$ defined by $h_{\delta}$.
For every $\delta\in(0,1)$, let $D: L^{1}\left([0,T];L^{1}\left(\mathbb{T}^{d}\right)\right)\times L^{1}\left([0,T];L^{1}\left(\mathbb{T}^{d}\right)\right)\to[0,\infty)$ be defined by
	\begin{equation} \label{k-14} D(f,g)=\sum_{k=1}^{\infty}2^{-k}\left(\frac{\|h_{1/k}(f)-h_{1/k}(g)\|_{L^{1}\left([0,T];L^{1}\left(\mathbb{T}^{d}\right)\right)}}{1+\|h_{1/k}(f)-h_{1/k}(g)\|_{L^{1}\left([0,T];L^{1}\left(\mathbb{T}^{d}\right)\right)}}\right)\notag.
	\end{equation}
It has been shown by \cite[Lemma 5.24]{FG24} that the function $D$ is a metric on $L^{1}\left([0,T];L^{1}\left(\mathbb{T}^{d}\right)\right)$ and the corresponding metric topology is proved to be equal to the usual strong norm topology on $L^1([0,T];L^1(\mathbb{T}^d))$.

According to the proof of Theorem \ref{thm-n-1}, for the kinetic solution $\bar{\rho}^{\varepsilon}$ of (\ref{s-3}), there exists a sequence of
 kinetic solutions $\{\rho^{n,\varepsilon}\}_{n\geq 1}$ of (\ref{s-8}) defined
on the probability space $(\Omega,  \mathcal{F}, \mathbb{P})$ such that $\mathbb{P}-$a.s., as $n\rightarrow \infty$,
 $\rho^{n,\varepsilon}\rightarrow\bar{\rho}^{\varepsilon}$ strongly in $L^1([0,T];L^1(\mathbb{T}^d))$ and $\nabla \sqrt{\rho^{n,\varepsilon}}\rightharpoonup \nabla\sqrt{\bar{\rho}^{\varepsilon}}$ weakly in $L^2([0,T];L^2(\mathbb{T}^d))$.
Then the weak lower semicontinuity of the Sobolev norm implies the following result.
\begin{proposition}\label{prp-n-1}
  Let $\rho_0\in {\rm{Ent}}(\mathbb{T}^d)$, $\varepsilon\in (0,1)$, $K\in \mathbb{N}$ satisfying $\varepsilon K^{d+2}(\varepsilon)\leq 1$. Assume that $V$ satisfies Assumption (A1), then the solution $\bar{\rho}^{\varepsilon}$ of (\ref{s-3}) satisfies that
   almost surely for every $t\in [0,T]$, $\|\bar{\rho}^{\varepsilon}(\cdot,t)\|_{L^1(\mathbb{T}^d)}=\|\rho_0\|_{L^1(\mathbb{T}^d)}$ and there exists $C>0$ independent of $\varepsilon$ such that
 \begin{align}\notag
	&\mathbb{E}\Big[\sup_{t\in [0,T]}\int_{\mathbb{T}^{d}} \Psi(\bar{\rho}^{\varepsilon}(x, t))dx\Big]+\mathbb{E}\int_{0}^{T} \int_{\mathbb{T}^{d}} \left|\nabla \sqrt{\bar{\rho}^{\varepsilon}}\right|^{2}dxdt\\
\label{k-35}
\leq &\int_{\mathbb{T}^{d}} \Psi(\rho_0)dx+ C(T,c_1,\|V\|_{L^{p^*}([0,T];L^p(\mathbb{T}^d))}, \|\rho_0\|_{L^1(\mathbb{T}^d)},N).
\end{align}
\end{proposition}
Further,  by Proposition \ref{prp-n-1}, we can show the tightness of solutions to stochastic controlled equation.
\begin{proposition}\label{prp-n-4}
Let $\rho_0\in {\rm{Ent}}(\mathbb{T}^d)$, $\varepsilon\in (0,1), K\in \mathbb{N}$ satisfying $\varepsilon K(\varepsilon)^{d+2}\leq 1$. Assume that $V$ satisfies Assumption (A1),
then the laws of the solutions $\{\bar{\rho}^{\varepsilon}\}_{\varepsilon\in (0,1)}$ of (\ref{s-3}) are tight on $L^1([0,T];L^1(\mathbb{T}^d))$.
\end{proposition}
\begin{proof}
Arguing similarly as in \cite[Proposition 5.14]{FG24} and \cite[Proposition 5.13]{WWZ22}, by (\ref{k-37}), for every $\beta\in(0,1/2)$ and $s>\frac{d}{2}+1$, there exists a constant $C$ independent of $n$ and $\varepsilon$ such that
	\begin{align}\label{wr-1}
	\mathbb{E}\left[\| h_{\delta}(\bar{\rho}^{\varepsilon})\|_{W^{\beta,1}([0,T];H^{-s}(\mathbb{T}^{d}))}\right]\le C(\delta,T,\|\rho_0\|_{L^1(\mathbb{T}^{d})},\|V\|_{L^{p^{*}}([0,T];L^{p}(\mathbb{T}^d))})(1+N),
	\end{align}
\begin{align}\label{wr-2}
  \mathbb{E}\Big[\|h_{\delta}(\bar{\rho}^{\varepsilon})\|_{L^1([0,T];W^{1,1}(\mathbb{T}^d))}\Big]
  \leq & C(T,c_1,\|V\|_{L^{p^*}([0,T];L^p(\mathbb{T}^d))},\rho_0,N),
\end{align}
and
\begin{align}\label{wr-3}
  \mathbb{E}\Big[\|h_{\delta}(\bar{\rho}^{\varepsilon})\|_{L^{\infty}([0,T];L^1(\mathbb{T}^d))}\Big]
  \leq \mathbb{E}\|\bar{\rho}^{\varepsilon}\|_{L^{\infty}([0,T];L^1(\mathbb{T}^d))}=\|\rho_0\|_{L^1(\mathbb{T}^d)}.
\end{align}
With (\ref{wr-1})-(\ref{wr-3}) in hand, a diagonal argument (see \cite[Proposition 5.26]{FG24}) implies that the laws of $\{\bar{\rho}^{\varepsilon}\}_{\varepsilon\in (0,1)}$ are tight on $L^{1}\left([0,T]; L^{1}\left(\mathbb{T}^{d}\right)\right)$ in the strong topology.
\end{proof}

\section{Proof of large deviation principles}\label{sec-6}
As discussed at the end of Section \ref{sec-5}, to obtain the LDP,
 it remains to establish {\textbf{(a)}} in {\textbf{Condition A}}.
Let $\bar{\rho}^{\varepsilon,K,g^{\varepsilon}}$ be the unique solution of stochastic controlled equation (\ref{s-3}) with initial data $\rho_0$ and control $\{g^{\varepsilon}: \varepsilon\in (0,1)\}\subseteq\mathcal{A}_N$ in the sense of Definition \ref{dfn-n-1}.
Denote  by  $\bar{q}^{\varepsilon,g^{\varepsilon}}$ the corresponding kinetic measure of $\bar{\rho}^{\varepsilon,K,g^{\varepsilon}}$. Using the same argument as in \cite[Theorem 6.2, step 2]{WWZ22}, we can show the result below.
\begin{lemma}
Let $\rho_0\in L^{\infty}(\mathbb{T}^d)$ be nonnegative, $\varepsilon\in (0,1)$ and $K\in \mathbb{N}$ satisfying $\varepsilon K^{d+2}(\varepsilon)\leq 1$. Let $\bar{\rho}^{\varepsilon,K(\varepsilon),g^{\varepsilon}}$ be the kinetic solution of (\ref{s-3}) with initial data $\rho_0$.
  Assume that $V$ satisfies Assumption (A1), then for every $M>0$ the kinetic measure $\bar{q}^{\varepsilon,g^{\varepsilon}}$ with respect to $\bar{\rho}^{\varepsilon,K(\varepsilon),g^{\varepsilon}}$ satisfies that
 \begin{align}\label{k-32}
  \sup_{0<\varepsilon<1} \mathbb{E}|\bar{q}^{\varepsilon,g^{\varepsilon}}([0,T]\times \mathbb{T}^d\times [0,M])|^2
  \leq C\Big(M, \|\rho_0\|_{L^1(\mathbb{T}^{d})}, \int_{\mathbb{T}^{d}} \Psi(\rho_0)dx, c_1, V, N \Big).
  \end{align}
\end{lemma}

Now we are ready to verify {\textbf{(a)}} in {\textbf{Condition A}}.
\begin{theorem}\label{thm-3}
Assume that $V$ satisfies Assumptions (A1) and (A2). Let $T>0$, $ \rho_0\in\rm{Ent}(\mathbb{T}^d)$,   $\varepsilon\in (0,1), K\in \mathbb{N}$ satisfies $\varepsilon K(\varepsilon)^{d+2}\rightarrow 0$ and $K(\varepsilon)\rightarrow \infty$ as $\varepsilon\rightarrow 0$. For every $N<\infty$, let $\{g^{\varepsilon}: 0<\varepsilon<1\}$ $\subset \mathcal{A}_N$. If $g^{\varepsilon}\rightarrow g$ in distribution weakly in $L^2([0,T]\times \mathbb{T}^d;\mathbb{R}^d)$, then
the solution $\bar{\rho}^{\varepsilon,K(\varepsilon),g^{\varepsilon}}$ of (\ref{s-3}) with initial data $\rho_0$ satisfy, as $\varepsilon\rightarrow 0$,
\begin{align*}
  \bar{\rho}^{\varepsilon,K(\varepsilon),g^{\varepsilon}}\rightarrow \rho^g\quad {\rm{in\ distribution\ on\ }} L^1([0,T];L^1(\mathbb{T}^d)),
\end{align*}
where $\rho^g$ is the unique solution of the skeleton equation (\ref{s-54}) with initial data $\rho_0$.
\end{theorem}
\begin{proof}
From (\ref{k-30}), we have  $\rho^g(\cdot)=\mathcal{G}^0\Big(\Big(\int^{\cdot}_0 g_k(s)ds,\int^{\cdot}_0 g^{\sharp}_k(s)ds\Big)_{k\in \mathbb{N}}\Big)$,
where $g_k(s)=( g(x,s), \cos(k\cdot x))$ and $g^{\sharp}_k(s)=(g(x,s), \sin(k\cdot x))$.
For simplicity, with a little of abuse of notations, denote by $\bar{\rho}^{\varepsilon} := \bar{\rho}^{\varepsilon,K(\varepsilon),g^{\varepsilon}}$, $\bar{q}^{\varepsilon}:=\bar{q}^{\varepsilon,g^{\varepsilon}}$ and $\rho:=\rho^g$.
From (\ref{s-5}),
for every $\psi\in C^{\infty}_c(\mathbb{T}^d\times (0,\infty))$, for almost every $t\in [0,T]$, $\mathbb{P}-$a.s., the kinetic function $\bar{\chi}^{\varepsilon}$ of $\bar{\rho}^{\varepsilon}$ and $\bar{q}^{\varepsilon}$ satisfy
     \begin{align}\notag
     &\int_{\mathbb{R}}\int_{\mathbb{T}^d}\bar{\chi}^{\varepsilon}(x,\xi,t)\psi(x,\xi)dxd\xi
      =\int_{\mathbb{R}}\int_{\mathbb{T}^d}\bar{\chi}^{\varepsilon}(x,\xi,0)\psi(x,\xi)dxd\xi
      +\int^t_0\int_{\mathbb{R}}\int_{\mathbb{T}^d}\bar{\chi}^{\varepsilon}\Delta_x\psi dxd\xi ds\\ \notag
      -&\int^t_0\int_{\mathbb{T}^d}\psi(x,\bar{\rho}^{\varepsilon})\nabla\cdot(\bar{\rho}^{\varepsilon} V\ast\bar{\rho}^{\varepsilon})dxds
      +2\int^t_0\int_{\mathbb{T}^d}\bar{\rho}^{\varepsilon}\nabla \sqrt{\bar{\rho}^{\varepsilon}}P_{K(\varepsilon)}g^{\varepsilon}(x,s)\partial_{\xi}\psi(x,\bar{\rho}^{\varepsilon})dxds\\ \notag
      +&\int^t_0\int_{\mathbb{T}^d}\sqrt{\bar{\rho}^{\varepsilon}}P_{K(\varepsilon)}g^{\varepsilon}(x,s)\cdot\nabla_{x}\psi(x,\bar{\rho}^{\varepsilon})dxds
      -\frac{\varepsilon N_{K(\varepsilon)}}{8}\int^t_0\int_{\mathbb{T}^d}(\bar{\rho}^{\varepsilon})^{-1}\nabla \bar{\rho}^{\varepsilon}
\cdot \nabla \psi(x,\bar{\rho}^{\varepsilon})dxds\\
\label{r-2}
-&\int^t_0\int_{\mathbb{R}}\int_{\mathbb{T}^d}\partial_{\xi}\psi(x,\xi)d\bar{q}^{\varepsilon}
+\frac{\varepsilon M_{K(\varepsilon)}}{2}\int^t_0\int_{\mathbb{T}^d} \bar{\rho}^{\varepsilon} \partial_{\xi}\psi(x,\bar{\rho}^{\varepsilon})dxds-\sqrt{\varepsilon}\int^t_0\int_{\mathbb{T}^d}\psi(x,\bar{\rho}^{\varepsilon})\nabla\cdot(\sqrt{\bar{\rho}^{\varepsilon}}d\xi^{K(\varepsilon)}).
\end{align}


For every $\varepsilon\in (0,1)$, let $\bar{p}^{\varepsilon}$ be defined by
 \begin{align*}
   d\bar{p}^{\varepsilon}:=\delta_0(\xi-\bar{\rho}^{\varepsilon})|\nabla \sqrt{\bar{\rho}^{\varepsilon}}||P_{K(\varepsilon)} g^{\varepsilon}|dxd\xi dt.
 \end{align*}
Then, for every $\varepsilon\in (0,1)$, $\bar{p}^{\varepsilon}$ is a nonnegative, almost surely finite measure on $\mathbb{T}^d\times \mathbb{R}\times [0,T]$ by the entropy estimate (\ref{k-35}). In fact, by (\ref{k-35}), we deduce that
\begin{align}\label{kk-33}
  \mathbb{E}|\bar{p}^{\varepsilon}(\mathbb{T}^d\times [0,T]\times \mathbb{R})|^2=& \mathbb{E}\Big(\int^T_0\|\nabla \sqrt{\bar{\rho}^{\varepsilon}}\|^2_{L^2(\mathbb{T}^d)}dt\Big)\Big(\int^T_0\|P_{K(\varepsilon)} g^{\varepsilon}\|^2_{L^2(\mathbb{T}^d)}dt\Big)\\ \notag
   \leq& N\Big[\int_{\mathbb{T}^{d}} \Psi(\rho_0)dx+ C(T,c_1,\|V\|_{L^{p^*}([0,T];L^p(\mathbb{T}^d))}, \|\rho_0\|_{L^1(\mathbb{T}^d)},N)\Big].
 \end{align}

 For any $\varepsilon\in (0,1)$, Let
 \begin{align*}
   Y^{\varepsilon}:=(\bar{\rho}^{\varepsilon}, \nabla \sqrt{\bar{\rho}^{\varepsilon}}, g^{\varepsilon}, (W^k,B^k)_{k\in \mathbb{N}}).
 \end{align*}
 $Y^{\varepsilon}$ takes values in the space
   \begin{align*}
     \bar{Y}=L^1([0,T];L^1(\mathbb{T}^d))\times \Big(L^2([0,T];L^2(\mathbb{T}^d;\mathbb{R}^d) )\Big)^2
    \times C([0,T];\mathbb{R})^{\mathbb{N}},
   \end{align*}
 where $\bar{Y}$ is equipped with the same topology as
in Theorem \ref{thm-n-1}.

By Proposition \ref{prp-n-4} and (\ref{k-35}),
 the laws of $ \{Y^{\varepsilon}\}_{\varepsilon\in (0,1)}$ are tight on $\bar{Y}$. Since $\bar{Y}$ is separable, by Prokhorov's theorem, after passing to a subsequence (still denoted by $Y^{\varepsilon}$), there exists a probability measure $\mu$ on $\bar{Y}$ such that the law of $Y^{\varepsilon}$ converges to $\mu$ in distribution, as  $\varepsilon\rightarrow 0$. Moreover, with the help of the Skorokhod representative theorem, there exists a stochastic probability space $(\tilde{\Omega},\tilde{\mathcal{F}},\tilde{\mathbb{P}},\tilde{\mathbb{E}})$ and $\bar{Y}-$valued random variables $\{\tilde{Y}^{\varepsilon}\}_{\varepsilon\in (0,1)}$ and $\tilde{Y}$ on $\tilde{\Omega}$ such that, for every $\varepsilon\in (0,1)$,
\begin{align}\label{kk-2}
  &\tilde{Y}^{\varepsilon}= Y^{\varepsilon}\ {\rm{in\ law\
  on\ }}\bar{X}\\
  \label{kk-3}
  &\tilde{Y} =\mu \ {\rm{in\ law\
  on\ }}\bar{X},
\end{align}
and such that $\tilde{\mathbb{P}}-$almost surely as $\varepsilon\rightarrow 0$,
\begin{align*}
  \tilde{Y}^{\varepsilon}\rightarrow \tilde{Y} \ {\rm{in}}\ \bar{Y}.
\end{align*}
By similar method as Theorem \ref{thm-n-1} and in \cite[Theorem 6.6]{WWZ22}, it follows that, for every $\varepsilon\in (0,1)$,
$\tilde{\mathbb{P}}-$almost surely there exists $\tilde{\rho}^{\varepsilon}$, and Brownian motions
$(\tilde{W}^{k,\varepsilon},\tilde{B}^{k,\varepsilon} )_{k\in \mathbb{N}}$ such that
\begin{align*}
  \tilde{Y}^{\varepsilon}=\Big(\tilde{\rho}^{\varepsilon}
  , \nabla \sqrt{\tilde{\rho}^{\varepsilon}},\tilde{g}^{\varepsilon},
  (\tilde{W}^{k,\varepsilon},\tilde{B}^{k,\varepsilon} )_{k\in \mathbb{N}}\Big).
\end{align*}
Moreover, define
\begin{align*}
   d\tilde{p}^{\varepsilon}:=\delta_0(\xi-\tilde{\rho}^{\varepsilon})|\nabla \sqrt{\tilde{\rho}^{\varepsilon}}||P_{K(\varepsilon)} \tilde{g}^{\varepsilon}|dxd\xi dt.
 \end{align*}
 Define a mapping $F_{\psi}$ as follows,
  \begin{align*}\notag
  &F_{\psi}(\bar{\rho}^{\varepsilon}, g^{\varepsilon}, (W^{k},B^{k})_{k\in \mathbb{N}})\\
  \notag
  := &
      \int_{\mathbb{R}}\int_{\mathbb{T}^d}\bar{\chi}^{\varepsilon}(x,\xi,0)\psi(x,\xi)dxd\xi
      -\int_{\mathbb{R}}\int_{\mathbb{T}^d}\bar{\chi}^{\varepsilon}(x,\xi,t)\psi(x,\xi)dxd\xi
      +\int^t_0\int_{\mathbb{R}}\int_{\mathbb{T}^d}\bar{\chi}^{\varepsilon}\Delta_x\psi dxd\xi ds\\ \notag
      -&\int^t_0\int_{\mathbb{T}^d}\psi(x,\bar{\rho}^{\varepsilon})\nabla\cdot(\bar{\rho}^{\varepsilon} V\ast\bar{\rho}^{\varepsilon})dxds
      +2\int^t_0\int_{\mathbb{T}^d}\bar{\rho}^{\varepsilon}\nabla \sqrt{\bar{\rho}^{\varepsilon}}P_{K(\varepsilon)}g^{\varepsilon}(x,s)\partial_{\xi}\psi(x,\bar{\rho}^{\varepsilon})dxds\\ \notag
      +&\int^t_0\int_{\mathbb{T}^d}\sqrt{\bar{\rho}^{\varepsilon}}P_{K(\varepsilon)}g^{\varepsilon}(x,s)\cdot\nabla_{x}\psi(x,\bar{\rho}^{\varepsilon})dxds
      -\frac{\varepsilon N_{K(\varepsilon)}}{8}\int^t_0\int_{\mathbb{T}^d}(\bar{\rho}^{\varepsilon})^{-1}\nabla \bar{\rho}^{\varepsilon}
\cdot \nabla \psi(x,\bar{\rho}^{\varepsilon})dxds\\
&
+\frac{\varepsilon M_{K(\varepsilon)}}{2}\int^t_0\int_{\mathbb{T}^d} \bar{\rho}^{\varepsilon} \partial_{\xi}\psi(x,\bar{\rho}^{\varepsilon})dxds-\sqrt{\varepsilon}\int^t_0\int_{\mathbb{T}^d}\psi(x,\bar{\rho}^{\varepsilon})\nabla\cdot(\sqrt{\bar{\rho}^{\varepsilon}}d\xi^{K(\varepsilon)}).
\end{align*}
we have $\bar{q}^{\varepsilon}(\partial_{\xi}\psi)=F_{\psi}(\bar{\rho}^{\varepsilon}, g^{\varepsilon}, (W^{k},B^{k})_{k\in \mathbb{N}})$.
Since $\tilde{\rho}^{\varepsilon}$ is a kinetic solution to (\ref{s-3}) with initial data $\rho_0$ and control $\tilde{g}^{\varepsilon}$ on the probability space $(\tilde{\Omega},\tilde{\mathcal{F}},\tilde{\mathbb{P}},(\tilde{W}^{k,\varepsilon},\tilde{B}^{k,\varepsilon} )_{k\in \mathbb{N}})$, we define $\tilde{q}^{\varepsilon}(\partial_{\xi}\psi):=F_{\psi}(\tilde{\rho}^{\varepsilon}, \tilde{g}^{\varepsilon}, (\tilde{W}^{k,\varepsilon},\tilde{B}^{k,\varepsilon})_{k\in \mathbb{N}})$.
Clearly, $\tilde{q}^{\varepsilon}$ has the same law as $\bar{q}^{\varepsilon}$.

 Proceed similarly as in \cite[Theorem 6.2]{WWZ22}, it follows that
\begin{align*}
  \tilde{Y}=(\tilde{\rho},\nabla\sqrt{\tilde{\rho}},\tilde{g},
  (\tilde{W}^k,\tilde{B}^k)_{k\in \mathbb{N}}).
\end{align*}
As a result, it holds that
\begin{align}\label{t-22}
\tilde{\rho}^{\varepsilon}\rightarrow  \tilde{\rho} \ {\rm{strongly \ in \ }} L^1([0,T];L^1(\mathbb{T}^d)),\  \tilde{\mathbb{P}}-a.s.,
\end{align}
and
\begin{align}\label{t-23}
\nabla\sqrt{\tilde{\rho}^{\varepsilon}}\rightarrow  \nabla\sqrt{\tilde{\rho}}, \ \tilde{g}^{\varepsilon}\rightarrow \tilde{g}\ {\rm{weakly \ in \ }} L^2([0,T];L^2(\mathbb{T}^d;\mathbb{R}^d)),\ \tilde{\mathbb{P}}-a.s..
\end{align}

For $r\in \mathbb{N}$, define $K_r:=\mathbb{T}^d\times [0,T]\times [0,r]$. Let $\mathcal{M}_r$ be the space of bounded Borel measures over $K_r$ (with norm given by the total variation of measures). Clearly, $\mathcal{M}_r$ is the topological dual of $C(K_r)$, which is the set of continuous functions on $K_r$.
Regarding the kinetic measures, by (\ref{k-32}) and (\ref{kk-2}), $\tilde{q}^{\varepsilon}$ is bounded in $L^2(\tilde{\Omega}; \mathcal{M}_r)$ for any $r>0$ uniformly on $\varepsilon\in (0,1)$.
Similarly to Theorem \ref{thm-n-1}, there exists a subsequence (still denoted by $\tilde{q}^{\varepsilon}$) and a kinetic measure $\tilde{q}$ such that $\tilde{q}^{\varepsilon}\rightharpoonup\tilde{q}$ weakly
star in $L^2(\tilde{\Omega}; \mathcal{M}_r)$ for every $r\in \mathbb{N}$. Then, with the aid of (\ref{t-22})-(\ref{t-23}), (\ref{k-28})  and (\ref{k-29}), we have $\tilde{q}\geq 4\delta_0(\xi-\tilde{\rho})\xi|\nabla\sqrt{\tilde{\rho}}|^2$, and
\begin{align}\label{k-16}
       \lim_{M\rightarrow \infty}M^{-1}\tilde{\mathbb{E}}\Big[\tilde{q}(\mathbb{T}^d\times [0,T]\times [M,2M])\Big]=&0,\\
       \label{k-17}
      \lim_{M\rightarrow \infty}M\tilde{\mathbb{E}}\Big[\tilde{q}(\mathbb{T}^d\times [0,T]\times [1/M,2/M])\Big]=&0.
     \end{align}
Similarly, for the measure $\tilde{p}^{\varepsilon}$, in view of (\ref{kk-33}) and (\ref{kk-2}), there exists a subsequence (still denoted by $\tilde{p}^{\varepsilon}$) and a kinetic measure $\tilde{p}$ such that $\tilde{p}^{\varepsilon}\rightharpoonup\tilde{p}$ weakly
star in $L^2(\tilde{\Omega}; \mathcal{M}_r)$ for every $r\in \mathbb{N}$. The condition at infinity (\ref{kk-33}) shows that $\tilde{p}$ defines an element of $L^2(\tilde{\Omega};\mathcal{M})$, where $\mathcal{M}$ denotes the space of bounded Borel measures over $\mathbb{T}^d\times [0,T]\times [0,\infty)$.

In the following, we aim to derive the equation satisfied by $\tilde{\rho}$. To achieve it, we will pass to the limit $\varepsilon\rightarrow 0$ in (\ref{r-2}). However,
as explained in \cite[Theorem 28]{FG23}, it is impossible to identify the limit
\begin{align*}
\lim_{\varepsilon\rightarrow 0}  \Big( 2\int^t_0\int_{\mathbb{T}^d}\tilde{\rho}^{\varepsilon}\nabla \sqrt{\tilde{\rho}^{\varepsilon}}P_{K(\varepsilon)}\tilde{g}^{\varepsilon}(x,s)\partial_{\xi}\psi(x,\tilde{\rho}^{\varepsilon})dxds\Big)
\end{align*}
due to the product of the weakly convergent gradient and control terms. The key observation is that this term does not appear in the weak solution of skeleton equation. Hence, we aim to prove that
 $\tilde{\mathbb{P}}-$almost surely, the kinetic solution $\tilde{\rho}^{\varepsilon}$ of (\ref{r-2}) converges to a weak solution of the skeleton equation (\ref{skeleton}) in the sense of Definition \ref{dfn-3}. To recover the weak formulation of the skeleton equation, test functions that are compactly in the velocity variable are needed.
For every $M>2$, let $\phi_{M}: [0,\infty)\rightarrow [0,1]$ be a smooth function satisfying that for some $c\in (0,\infty)$ independent of $M$,
\begin{align*}
  \phi_{M}=0\ {\rm{on}} \ [0,1/M]\cup[2M,\infty),\ \phi_{M}=1\ {\rm{on}} \ [2/M,M],
\end{align*}
and
\begin{align*}
  |\phi'_{M}|\leq cMI_{\{\xi\in [1/M,2/M]\}}+cM^{-1}I_{\{\xi\in [M,2M]\}}.
\end{align*}
Let $\eta\in C^{\infty}(\mathbb{T}^d)$ be arbitrary. Taking $\psi_M(x,\xi)=\phi_{M}(\xi)\eta(x)$ in (\ref{r-2}), we will pass first to the limit $\varepsilon\rightarrow0$ and then $M\rightarrow \infty$ to recover the weak solution of the skeleton equation.  Compared with \cite[Theorem 28]{FG23}, we need to deal with the kinetic measure term and the kernel term.
We firstly proceed with the kinetic measure term. In view of $\tilde{q}^{\varepsilon}\rightharpoonup\tilde{q}$ weakly
star in $L^2(\tilde{\Omega}; \mathcal{M}_r)$ for every $r>0$, we get
\begin{align*}\notag
\lim_{\varepsilon\rightarrow 0}\tilde{\mathbb{E}}\int^t_0\int_{\mathbb{R}}\int_{\mathbb{T}^d}|\phi'_M(\xi)\eta(x)|d\tilde{q}^{\varepsilon}
=\tilde{\mathbb{E}}\int^t_0\int_{\mathbb{R}}\int_{\mathbb{T}^d}|\phi'_M(\xi)\eta(x)|d\tilde{q}.
\end{align*}
Then,
by the definition of $\phi_{M}$ and using (\ref{k-16})-(\ref{k-17}), we have
\begin{align}\notag
&\lim_{M\rightarrow \infty}\lim_{\varepsilon\rightarrow 0}\tilde{\mathbb{E}}\int^t_0\int_{\mathbb{R}}\int_{\mathbb{T}^d}|\phi'_M(\xi)\eta(x)|d\tilde{q}^{\varepsilon}
=\lim_{M\rightarrow \infty}\tilde{\mathbb{E}}\int^t_0\int_{\mathbb{R}}\int_{\mathbb{T}^d}|\phi'_M(\xi)\eta(x)|d\tilde{q}\\
\label{k-33}
&\lesssim \lim_{M\rightarrow \infty}M\tilde{\mathbb{E}}\tilde{q}([0,T]\times \mathbb{T}^d\times [1/M,2/M])+M^{-1}\lim_{M\rightarrow \infty}\tilde{\mathbb{E}}\tilde{q}([0,T]\times \mathbb{T}^d\times [M,2M])=0.
\end{align}
Similarly, by the finite of $\tilde{p}$, the definition of $\phi_{M}$ and the boundedness of $\eta$, we get
\begin{align*}
&\lim_{M\rightarrow \infty}\lim_{\varepsilon\rightarrow 0}2\tilde{\mathbb{E}}\int^t_0\int_{\mathbb{T}^d}\int_{\mathbb{R}}\xi\phi'_M(\xi)\eta(x)d\tilde{p}^{\varepsilon}\lesssim\lim_{M\rightarrow \infty}\tilde{\mathbb{E}}\int^t_0\int_{\mathbb{T}^d}\int_{[M^{-1},2M^{-1}]\cup [M,2M]}d\tilde{p}=0.
\end{align*}
For the kernel term,  it follows from \cite[Theorem 6.2, step 4]{WWZ22} that
\begin{align*}
\lim_{\varepsilon\rightarrow 0}\int_0^t\int_{\mathbb{T}^d}\phi_{M}(\tilde{\rho}^{\varepsilon})\eta(x)\nabla\cdot(\tilde{\rho}^{\varepsilon}V\ast\tilde{\rho}^{\varepsilon})dxds
=\int_0^t\int_{\mathbb{T}^d}
\phi_{M}(\tilde{\rho})\eta(x)\nabla\cdot(\tilde{\rho}V\ast\tilde{\rho})dxds, \quad \tilde{\mathbb{P}}-a.s..
\end{align*}
Then, by the definition of $\phi_{M}$, (\ref{t-17-1}) and the dominated  convergence theorem, we get $\tilde{\mathbb{P}}-$a.s.,
\begin{align}\label{s-40}
\lim_{M\rightarrow \infty}\lim_{\varepsilon\rightarrow 0}\int_0^t\int_{\mathbb{T}^d}\phi_{M}(\tilde{\rho}^{\varepsilon})\eta(x)\nabla\cdot(\tilde{\rho}^{\varepsilon}V\ast\tilde{\rho}^{\varepsilon})dxds
=\int_0^t\int_{\mathbb{T}^d}
\eta(x)\nabla\cdot(\tilde{\rho}V\ast\tilde{\rho})dxds.
\end{align}


Based on (\ref{k-33}), (\ref{s-40}) and  \cite[Theorem 28]{FG23}, we conclude that $\tilde{\mathbb{P}}-$a.s. for almost every $t\in [0,T]$,
 \begin{align*}
 \int_{\mathbb{T}^d}\tilde{\rho}\eta(x)dx=&\int_{\mathbb{T}^d}\rho_0\eta(x)dx
 -2\int^t_0\int_{\mathbb{T}^d}\sqrt{\tilde{\rho}}\nabla \sqrt{\tilde{\rho}}\cdot \nabla\eta(x)dxds\\
 +&\int^t_0\int_{\mathbb{T}^d}\sqrt{\tilde{\rho}}\tilde{g}\cdot \nabla \eta(x)dxds-\int_0^t\int_{\mathbb{T}^d}
\eta(x)\nabla\cdot(\tilde{\rho}V\ast\tilde{\rho})dxds.
\end{align*}
Thus, $\tilde{\rho}$ is a weak solution of the skeleton equation with control $\tilde{g}$ in the sense of Definition \ref{dfn-1}. Further, by Proposition \ref{L1uniq} and Theorem \ref{thm-1}, it follows that $\tilde{\rho}$ is the unique weak solution of the skeleton equation (\ref{skeleton}).
 Since $\tilde{\rho}^{\varepsilon}$ has the same law as $\bar{\rho}^{\varepsilon}$ and $\tilde{g}$ has the same law as $g$, we get $\bar{\rho}^{\varepsilon}$ converges to $\rho$ in distribution on $L^1([0,T];L^1(\mathbb{T}^d))$, which is the desired result.
\end{proof}

\section{Applications to fluctuations of Ising-Kac-Kawasaki dynamics}\label{sec-7}
In this section, we will show the large deviations of a conservative SPDE called fluctuating  Ising-Kac-Kawasaki equation. This equation was
 proposed by \cite{GJE99} when studying nonlinear fluctuations of Kawasaki dynamical Ising-Kac process in a neighborhood of the critical temperature.
As explained in subsection \ref{subsec-1.1}, (2), this equation can be regarded as a continuum phenomenological model of the Kawasaki dynamical Ising-Kac spin system due to it shares the same fluctuations feature with Kawasaki dynamical Ising-Kac process in terms of Gaussian fluctuations, large deviations and scaling limit near criticality when the spatial dimension $d=1$.

For any spatial dimension $d\ge1$, the fluctuating  Ising-Kac-Kawasaki equation can be written as
\begin{equation}\label{ising-kac-2}
d\rho^{\varepsilon}=\Delta\rho^{\varepsilon} dt-\nabla\cdot[(1-(\rho^{\varepsilon})^2) \nabla J\ast\rho^{\varepsilon}]dt-\varepsilon^{1/2}\nabla\cdot(\sqrt{1-(\rho^{\varepsilon})^2}\circ d\xi^{K(\varepsilon)}),
\end{equation}
where $\xi^K$ is defined by (\ref{s-39-1}).
{\bf{In the sequel, we always assume that $ J\in C^{\infty}(\mathbb{T}^d;\mathbb{R}^d)$.}}

Clearly, (\ref{ising-kac-2}) can be reformulated into It\^o's form:
\begin{equation}\label{ising-kac-3}
d\rho^{\varepsilon}=\Delta\rho^{\varepsilon} dt-\nabla\cdot[(1-(\rho^{\varepsilon})^2) \nabla J\ast\rho^{\varepsilon}]dt-\varepsilon^{1/2}\nabla\cdot(\sqrt{1-(\rho^{\varepsilon})^2}d\xi^{K(\varepsilon)})+\frac{N_{K(\varepsilon)}}{2}\nabla\cdot\Big(\frac{\rho^{\varepsilon}}{1-(\rho^{\varepsilon})^2}\nabla\rho^{\varepsilon}\Big).
\end{equation}

For any $g\in L^2([0,T]\times\mathbb{T}^d;\mathbb{R}^d)$, the skeleton equation associated with (\ref{ising-kac-3}) can be written as
\begin{equation}\label{skeleton-ising}
\partial_t\rho=\Delta\rho-\nabla\cdot[(1-\rho^2)\nabla J\ast\rho]-\nabla\cdot(\sqrt{1-\rho^2}g). 	
\end{equation}
As an important of the proof of large deviations, we need to prove the well-posedness of skeleton equation and weak-strong uniqueness.

\subsection{Well-posedness of skeleton equation and weak-strong continuity}
First of all, we define the weak solution of (\ref{skeleton-ising}).
\begin{definition}\label{dfn-weak-ising}
	Let $g\in L^2(\mathbb{T}^d\times[0,T];\mathbb{R}^d)$, and $\rho_0\in L^2(\mathbb{T}^d;[-1,1])$. A weak solution of (\ref{skeleton-ising}) is a
  function
  $\rho\in C([0,T];L^2(\mathbb{T}^d;[-1,1])) \bigcap L^{2}([0,T];H^1(\mathbb{T}^d))$ that satisfies, for $\psi\in C^{\infty}(\mathbb{T}^d)$,
\begin{align}\notag
  \int_{\mathbb{T}^d}\rho(x,t)\psi(x)dx&=\int_{\mathbb{T}^d}\rho_0\psi(x)dx-\int^t_0\int_{\mathbb{T}^d}\nabla \rho\cdot \nabla\psi dxds+\int^t_0\int_{\mathbb{T}^d}\nabla\psi(x)\cdot[(1-\rho^2) \nabla J\ast \rho] dxds\\
\label{ss-5}
& +\int^t_0\int_{\mathbb{T}^d}\sqrt{1-\rho^2}g\cdot \nabla \psi dxds.
\end{align}

\end{definition}

\begin{proposition}\label{L2uniq}
 Let $\rho_0\in L^2(\mathbb{T}^d;[-1,1])$ and $g\in L^2([0,T];L^2(\mathbb{T}^d;\mathbb{R}^d))$, then the weak solution of (\ref{skeleton-ising}) is unique.
\end{proposition}
\begin{proof}
Let $\rho^1_0, \rho^2_0\in L^2(\mathbb{T}^d;[-1,1])$ be two initial values. Assume that $\rho_1, \rho_2$  are weak solutions of (\ref{skeleton-ising}) in the sense of Definition \ref{dfn-weak-ising} with  initial data $\rho^1_0, \rho^2_0$. For any $\delta\in(0,1)$, let $\eta^{\delta}_d$ be the standard convolution kernel on $\mathbb{T}^d$ and $\eta^{\delta}=:\eta^{\delta}_1$ be a standard convolution kernel on $\mathbb{R}$. For $i\in \{1,2\}$, let $\rho_i^{\delta}=\rho_i\ast\eta^{\delta}_d$. Using the definition of the weak solution to see that
	\begin{align*}
	\partial_t\rho_i^{\delta}=-\nabla\eta^{\delta}_d\ast\rho_i+\nabla\eta_{\delta}\ast[(1-(\rho_i)^2)\nabla J\ast\rho_i]+\nabla\eta^{\delta}_d\ast(\sqrt{1-(\rho_i)^2}g). 	
	\end{align*}
Denote by $a: \mathbb{R}\rightarrow [0,\infty)$ the absolute value function $a(x)=|x|$. For every $\gamma\in(0,1)$, we set $a^{\gamma}:=a\ast\eta^{\gamma}$ and $\sgn^{\gamma}=\sgn\ast\eta^{\gamma}$.
For every $\delta,\gamma\in(0,1)$, applying the chain rule to $\partial_t\int_{\mathbb{T}^d}a^{\gamma}(\rho_1^{\delta}-\rho_2^{\delta})$, by integration by parts formula and the boundedness of $\rho_i$, then
passing to the limits  $\delta\rightarrow0$ and $\gamma\rightarrow0$ to see that
\begin{align*}
	\partial_t\|\rho_1-\rho_2\|_{L^1(\mathbb{T}^d)}
\leq
C(\|\nabla J\|_{L^{\infty}(\mathbb{T}^d)},\|\Delta J\|_{L^{\infty}(\mathbb{T}^d)})(\|\nabla \rho_2\|_{L^2(\mathbb{T}^d)}+1)\|\rho_1-\rho_2\|_{L^1(\mathbb{T}^d)}.
\end{align*}
Thus, we get the uniqueness by applying Gronwall inequality.
\end{proof}

In the following, we provide the existence of weak solutions for (\ref{skeleton-ising}). Since the proof is entirely similar to the proof of Proposition \ref{prp-2}, thus we omit it.
\begin{proposition}\label{prp-2-ising}
 Let
 $\rho_0\in L^2(\mathbb{T}^d;[-1,1])$ and $g\in L^2([0,T]\times\mathbb{T}^d;\mathbb{R}^d)$. Then there exists a weak solution  to (\ref{skeleton-ising})
 in the sense of Definition \ref{dfn-weak-ising}. Moreover,

\end{proposition}

Regarding the weak-strong continuity of (\ref{skeleton-ising}), it can be proved by the same argument as in Proposition \ref{prp-3}. It reads as follows.

\begin{proposition}\label{prp-3-ising}
 Let $\rho_0\in L^2(\mathbb{T}^d;[-1,1])$.
For any $N>0$, assume that $\{g_n\}_{n\in\mathbb{N}_+}, g\subseteq L^2([0,T]\times\mathbb{T}^d;\mathbb{R}^d)$ satisfying
\begin{align*}
\sup_{n\geq 1}\int^T_0 \|g_n(s)\|^2_{L^2(\mathbb{T}^d)}ds\leq N,
\quad g_n\rightharpoonup g\ {\rm{weakly\ in\ }} L^2([0,T]\times\mathbb{T}^d;\mathbb{R}^d).
\end{align*}
For every $n\in \mathbb{N}$, let $\rho_n$ be the weak solution of the skeleton equation (\ref{skeleton-ising}) with control $g_n$ and initial data $\rho_0$. Let $\rho$ be the weak solution of the skeleton equation to (\ref{skeleton-ising}) with control $g$ and initial data $\rho_0$. Then
\begin{align*}
\rho_n\rightarrow\rho\ {\rm{strongly\ in\ }} L^2([0,T];L^2(\mathbb{T}^d)),\quad {\rm{as}}\ n\rightarrow \infty.
\end{align*}
\end{proposition}

\subsection{Proof of large deviations}

 For any $N<\infty$, let $\{g^{\varepsilon}: 0<\varepsilon<1\}$ $\subset \mathcal{A}_N$.
Then, we have
\begin{eqnarray}\label{ss-8}
  \sup_{\varepsilon\in (0,1)}\|g^{\varepsilon}\|^2_{L^{\infty}(\Omega;L^2(\mathbb{T}^d\times [0,T];\mathbb{R}^d) )}\leq N.
\end{eqnarray}
With the above $g^{\varepsilon}$, we consider the following stochastic control equation
\begin{align}\notag
 \partial_t \bar{\rho}^{\varepsilon,K,g^{\varepsilon}}=&\Delta \bar{\rho}^{\varepsilon,K,g^{\varepsilon}}-\nabla\cdot  [(1-(\bar{\rho}^{\varepsilon,K,g^{\varepsilon}})^2) \nabla J\ast \bar{\rho}^{\varepsilon,K,g^{\varepsilon}}]-\sqrt{\varepsilon}\nabla\cdot  (\sqrt{1-(\bar{\rho}^{\varepsilon,K,g^{\varepsilon}})^2} \xi^K)\\
 \label{s-3-ising}
 & +\frac{\varepsilon N_K}{2}\nabla\cdot\Big(\frac{\bar{\rho}^{\varepsilon,K,g^{\varepsilon}}}{1-(\bar{\rho}^{\varepsilon,K,g^{\varepsilon}})^2}\nabla\bar{\rho}^{\varepsilon,K,g^{\varepsilon}}\Big)-\nabla\cdot (\sqrt{1-(\bar{\rho}^{\varepsilon,K,g^{\varepsilon}})^2} P_K g^{\varepsilon}).
 \end{align}
  For simplicity, denote by $\bar{\rho}^{\varepsilon}:=\bar{\rho}^{\varepsilon,K,g^{\varepsilon}}$.

Similarly to Dean-Kawasaki equation, we also define kinetic solution and
make entropy estimates for (\ref{s-3-ising}).
   For a given bounded solution $\bar{\rho}^{\varepsilon}$ of (\ref{s-3-ising}), we define the kinetic function $\bar{\chi}^{\varepsilon}: \mathbb{T}^{d}\times\mathbb{R}\times[0,T]\to\{-1,1\}$ of $\bar{\rho}^{\varepsilon}$ as
\begin{align*}
\bar{\chi}^{\varepsilon}(x,\xi,t):=\mathbf{1}_{\{0<\xi<\bar{\rho}^{\varepsilon}(x,t)\}}-I_{\{\bar{\rho}^{\varepsilon}(x,t)<\xi<0\}}.
\end{align*}
Formally, we have the following identities
\begin{align*}
\nabla\bar{\chi}^{\varepsilon}=2\delta_{0}(\xi-\bar{\rho}^{\varepsilon})\nabla\bar{\rho}^{\varepsilon}, \quad \partial_{\xi}\bar{\chi}^{\varepsilon}=2\delta_{0}(\xi)-2\delta_{0}(\xi-\bar{\rho}^{\varepsilon})\quad {\rm{and}}\ \bar{\rho}^{\varepsilon}=\frac12\int_{\mathbb{R}}\bar{\chi}^{\varepsilon} \mathrm{d}\xi.
\end{align*}

Define $\psi(\xi)=\frac{1}{2}\log\frac{1+\xi}{1-\xi}$ and $\Psi(\xi)=\frac{1}{2}[(1+\xi)\log(1+\xi)-(\xi+1)+(1-\xi)\log(1-\xi)-(1-\xi)]$. Clearly, $\Psi'(\xi)=\psi(\xi)$.
Let the initial data be subject to the space
\begin{align}\label{eqq3}
\overline{\text{{\rm{Ent}}}}(\mathbb{T}^{d})=\Big\{\rho: -1\leq\rho\leq1\ a.e.,\ {\rm{and}}\ \int_{\mathbb{T}^{d}}\Psi(\rho(x))dx<\infty\Big\}.
\end{align}

When the control term $g^{\varepsilon}\equiv0$ and the initial value belongs to $\overline{\text{{\rm{Ent}}}}(\mathbb{T}^{d})$, the well-posedness of (\ref{s-3-ising}) in the framework of stochastic renormalized kinetic solution has been proved by \cite[Theorem 7.12]{WWZ22}.
Similar to \cite[Definition 7.2]{WWZ22}, the definition of (\ref{s-3-ising}) is formulated as follows.
 \begin{definition}\label{dfn-n-1-ising}
Let $\rho_0\in \overline{\text{{\rm{Ent}}}}(\mathbb{T}^{d})$, $\varepsilon\in (0,1)$ and $K\in \mathbb{N}$. A stochastic renormalized kinetic solution of (\ref{s-3-ising}) with initial datum $\bar{\rho}^{\varepsilon}(\cdot,0)=\rho_0$ is an almost surely continuous $L^2(\mathbb{T}^d;[-1,1])$-valued $\mathcal{F}_t$-predictable function $\bar{\rho}^{\varepsilon}\in L^2\left(\Omega\times[0,T];L^2(\mathbb{T}^d;[-1,1])\right)$ that satisfies the following properties.
\begin{enumerate}
  \item Essentially bounded: almost surely for every $t\in[0,T]$,
		\begin{equation}\label{eqq4}
		\bar{\rho}^{\varepsilon}(\cdot,t)\in[-1,1],\ a.e.
		\end{equation}
  \item Regularity of $\sqrt{1-(\bar{\rho}^{\varepsilon})^2}$: there exists a constant $c\in(0,\infty)$ such that
		\begin{equation}\label{eqq5}
		\mathbb{E}\int^T_0\int_{\mathbb{T}^d}\Big[|\nabla\sqrt{1-(\bar{\rho}^{\varepsilon})^2}|^2+|\nabla \bar{\rho}^{\varepsilon}|^2\Big]\mathrm{d}x\mathrm{d}s\le c(T,d,\rho_0,J).
		\end{equation}
		Furthermore, there exists a finite nonnegative kinetic measure $\bar{q}^{\varepsilon}$ satisfying the following items.
		\item Regularity: almost surely
		\begin{align}\label{eqq6}
		\delta_{0}(\xi-\bar{\rho}^{\varepsilon})|\nabla \bar{\rho}^{\varepsilon}|^{2}\le \bar{q}^{\varepsilon}\quad {\rm{on}}\ \mathbb{T}^d\times[-1,1]\times[0,T].
		\end{align}
\item Optimal regularity: the measure $\mu^{\varepsilon}$ defined by
	\begin{align}\label{mu}
		\mathrm{d}\mu^{\varepsilon}=\left(1-\xi^2\right)^{-1}\mathrm{d}\bar{q}^{\varepsilon}\ \text{is finite on}\ \mathbb{T}^d\times(-1,1)\times[0,T].
	\end{align}
\item The equation: the pair $(\bar{\chi}^{\varepsilon}(x,\xi,t),\bar{q}^{\varepsilon})$ satisfies, for every $\psi\in C^{\infty}_c(\mathbb{T}^d\times (-1,1))$, for almost every $t\in [0,T]$, $\mathbb{P}-$a.s.,
     \begin{align}\notag
      &\int_{\mathbb{R}}\int_{\mathbb{T}^d}\bar{\chi}^{\varepsilon}(x,\xi,t)\psi(x,\xi)dxd\xi
      =\int_{\mathbb{R}}\int_{\mathbb{T}^d}\bar{\chi}(\rho_0(x))\psi(x,\xi)dxd\xi
      -2\int^t_0\int_{\mathbb{T}^d}\nabla\bar{\rho}^{\varepsilon}\cdot\nabla\psi(x,\bar{\rho}^{\varepsilon}) dx ds\\ \notag
      -&2\int^t_0\int_{\mathbb{T}^d}\psi(x,\bar{\rho}^{\varepsilon})\nabla\cdot[(1-(\bar{\rho}^{\varepsilon})^2) \nabla J\ast\bar{\rho}^{\varepsilon}]dxds
      +2\int^t_0\int_{\mathbb{T}^d}\sqrt{1-(\bar{\rho}^{\varepsilon})^2}\nabla \bar{\rho}^{\varepsilon}\cdot P_Kg^{\varepsilon} (\partial_{\xi}\psi)(x,\bar{\rho}^{\varepsilon})dxds\\ \notag
      +& 2\int^t_0\int_{\mathbb{T}^d}\sqrt{1-(\bar{\rho}^{\varepsilon})^2}P_Kg^{\varepsilon}(x,s)\cdot (\nabla  \psi)(x,\bar{\rho}^{\varepsilon})dxds
      -\varepsilon N_K\int^t_0\int_{\mathbb{T}^d}\Big(\frac{\bar{\rho}^{\varepsilon}}{1-(\bar{\rho}^{\varepsilon})^2}\Big)\nabla \bar{\rho}^{\varepsilon}
\cdot \nabla \psi(x,\bar{\rho}^{\varepsilon})dxds\\
\notag
-&2\int^t_0\int_{\mathbb{R}}\int_{\mathbb{T}^d}\partial_{\xi}\psi(x,\xi)d\bar{q}^{\varepsilon}+\varepsilon M_K\int^t_0\int_{\mathbb{T}^d} (1-(\bar{\rho}^{\varepsilon})^2) \partial_{\xi}\psi(x,\bar{\rho}^{\varepsilon})dxds\\
\label{s-5-ising}
-&2\sqrt{\varepsilon}\int^t_0\int_{\mathbb{T}^d}\psi(x,\bar{\rho}^{\varepsilon})\nabla\cdot(\sqrt{1-(\bar{\rho}^{\varepsilon})^2}d\xi^{K}).
\end{align}
\end{enumerate}
\end{definition}
We point out that the skeleton equation (\ref{skeleton}) can also be equipped with kinetic solution similar to (\ref{s-3-ising}). Then, by the same method as Theorem \ref{thm-1}, we can show the equivalence of renormalized kinetic solutions and weak solutions to (\ref{skeleton}).

Thanks to \cite[Proposition 7.11, Theorem 7.12]{WWZ22}, by using the boundedness of weak solution and (\ref{ss-8}), we  readily deduce the following well-posedness and tightness of (\ref{s-3-ising}).
\begin{theorem}\label{thm-8}
Assume that $J\in C^{\infty}(\mathbb{T}^d)$. Let $\rho_0\in \overline{\rm{Ent}} (\mathbb{T}^d)$.  Then (\ref{s-3-ising}) admits a unique renormalized kinetic solution $\bar{\rho}^{\varepsilon,K,g^{\varepsilon}}$ in the sense of Definition \ref{dfn-n-1-ising} which is a probabilistically strong solution. Furthermore, assume that $\varepsilon\in (0,1)$, $K\in \mathbb{N}$ satisfying $\varepsilon K^{d+2}(\varepsilon)< \infty$, then the laws of the solutions $\{\bar{\rho}^{\varepsilon,K(\varepsilon),g^{\varepsilon}}\}_{\varepsilon\in (0,1)}$ of (\ref{s-3-ising}) are tight on $L^2([0,T];L^2(\mathbb{T}^d; [-1,1]))$.
\end{theorem}

The result of large deviations of solutions of (\ref{ising-kac-3}) is formulated as follows.
\begin{theorem}\label{thm-LDP-ising}
Assume that $J\in C^{\infty}(\mathbb{T}^d)$. Let $\rho_0\in  \overline{\rm{Ent}} (\mathbb{T}^d)$, $\varepsilon\in (0,1)$, and $K\in \mathbb{N}$ satisfying $K(\varepsilon)\rightarrow \infty$ and $\varepsilon K^{d+2}(\varepsilon)\rightarrow 0$. Then the solutions $\{\rho^{\varepsilon,K(\varepsilon)}(\rho_0)\}_{\varepsilon\in (0,1)}$ of (\ref{ising-kac-3}) satisfy large deviation principles with rate function $I$ on $L^2([0,T];L^2(\mathbb{T}^d;[-1,1]))$, where $I$ is given by
\begin{align}\label{k-7-ising}
I(\rho)=\frac{1}{2}\inf\Big\{\|g\|^2_{L^2(\mathbb{T}^d\times[0,T];\mathbb{R}^d)}:
\partial_t \rho=\Delta\rho-\nabla\cdot[(1-\rho^2) \nabla J\ast\rho]-\nabla\cdot(\sqrt{1-\rho^2} g),\  \rho(0)=\rho_0
\Big\}.
\end{align}
\end{theorem}
\begin{proof}
Based on Proposition \ref{prp-3-ising}, it remains to verify {\textbf{(a)}} in {\textbf{Condition A}}.
That is, for every $N<\infty$, let $\{g^{\varepsilon}: 0<\varepsilon<1\}$ $\subset \mathcal{A}_N$, if $g^{\varepsilon}\rightarrow g$ in distribution weakly in $L^2([0,T]\times \mathbb{T}^d;\mathbb{R}^d)$, we need to show that
the solution $\bar{\rho}^{\varepsilon,K(\varepsilon),g^{\varepsilon}}$ of (\ref{s-3-ising}) with initial data $\rho_0$ satisfies that, as $\varepsilon\rightarrow 0$,
\begin{align*}
  \bar{\rho}^{\varepsilon,K(\varepsilon),g^{\varepsilon}}\rightarrow \rho^g\quad {\rm{in\ distribution\ on\ }} L^2([0,T];L^2(\mathbb{T}^d)),
\end{align*}
where $\rho^g$ is the unique solution of the skeleton equation (\ref{skeleton-ising}) with initial data $\rho_0$.

Due to the tightness of $\{\bar{\rho}^{\varepsilon,K(\varepsilon),g^{\varepsilon}}\}_{\varepsilon\in (0,1)}$ on $L^2([0,T];L^2(\mathbb{T}^d; [-1,1]))$, (\ref{eqq5}) and (\ref{mu}), by a simplified version of Theorem \ref{thm-3}, Prokhorov's theorem and the Skorokhod representation theorem, there exists a stochastic basis still denoted by $(\Omega,\mathcal{F},\mathbb{P})$, a subsequence $\varepsilon\rightarrow 0$ and $\rho^*\in L^2(\Omega\times [0,T];H^1(\mathbb{T}^d))$ such that
$\mathbb{P}-$a.s.,
\begin{align}\label{ss-10}
  \rho^{\varepsilon}\rightarrow \rho^*\ {\rm{strongly\ in}} \ L^2([0,T];L^2(\mathbb{T}^d;[-1,1]))\ {\rm{and\ weakly\ in}}\ L^2([0,T];H^1(\mathbb{T}^d)),
\end{align}
and that there exists a finite, nonnegative measure $q$ on $\mathbb{T}^d\times [0,T]\times [-1,1]$ such that the corresponding kinetic measure $\bar{q}^{\varepsilon}$ of $\bar{\rho}^{\varepsilon,K(\varepsilon),g^{\varepsilon}}$
fulfills
\begin{align}\label{ss-12}
  (1-\xi^2)^{-1}\bar{q}^{\varepsilon}\rightharpoonup q\ {\rm{weakly\ in}}\ \mathcal{M}_{+}(\mathbb{T}^d\times [0,T]\times [-1,1]),
\end{align}
where $\mathcal{M}_{+}(\mathbb{T}^d\times [0,T]\times [-1,1])$ is the space of nonnegative Radon measures on $\mathbb{T}^d\times [0,T]\times [-1,1]$.

In view of (\ref{ss-10}), it suffices to show that $\rho^*$ is the unique weak solution of  (\ref{skeleton-ising}). Similar to Theorem \ref{thm-3}, we need to introduce test functions
that are compactly supported in the velocity variable.
For every $\beta\in (0,\frac{1}{4})$, let $\phi_{\beta}\in C^{\infty}_c((-1,1);[0,1])$ satisfy that $\phi_{\beta}(\xi)=0$ if $\xi\geq 1-\beta$ and
$\phi_{\beta}(\xi)=1$ if $\xi\in [-1+2\beta, 1-2\beta]$ and that $|\phi'_{\beta}(\xi)|\leq \frac{\bar{c}}{\beta}(I_{-1+\beta<\xi<-1+2\beta}+I_{1-2\beta<\xi<1-\beta})$ for some $\bar{c}$ independent of $\beta$. A direct calculation shows that
\begin{align}\label{ss-11}
  |\phi'_{\beta}(\xi)|\leq \frac{c}{1-\xi^2}\Big(I_{-1+\beta<\xi<-1+2\beta}+I_{1-2\beta<\xi<1-\beta}\Big),
\end{align}
where $c$ is a constant independent of $\beta$.
 Let $\eta(x)\in C^{\infty}_c(\mathbb{T}^d)$. Taking $\psi(x,\xi)=\eta(x)\phi_{\beta}(\xi)$ in (\ref{s-5-ising}), we get
\begin{align}\notag
      &\int_{\mathbb{R}}\int_{\mathbb{T}^d}\bar{\chi}^{\varepsilon}(x,\xi,t)\phi_{\beta}(\xi)\eta(x)dxd\xi
      =\int_{\mathbb{R}}\int_{\mathbb{T}^d}\bar{\chi}(\rho_0(x))\phi_{\beta}(\xi)\eta(x)dxd\xi
      -2\int^t_0\int_{\mathbb{T}^d}\nabla\bar{\rho}^{\varepsilon}\cdot\nabla \eta(x)\phi_{\beta}(\bar{\rho}^{\varepsilon}) dx ds\\ \notag
      -&2\int^t_0\int_{\mathbb{T}^d}\eta(x)\phi_{\beta}(\bar{\rho}^{\varepsilon})\nabla\cdot[(1-(\bar{\rho}^{\varepsilon})^2) \nabla J\ast\bar{\rho}^{\varepsilon}]dxds
      +2\int^t_0\int_{\mathbb{T}^d}\sqrt{1-(\bar{\rho}^{\varepsilon})^2}\nabla \bar{\rho}^{\varepsilon}\cdot P_Kg^{\varepsilon} \eta(x)\phi'_{\beta}(\bar{\rho}^{\varepsilon}) dxds\\ \notag
      +& 2\int^t_0\int_{\mathbb{T}^d}\sqrt{1-(\bar{\rho}^{\varepsilon})^2}P_Kg^{\varepsilon}(x,s)\cdot \nabla \eta(x) \phi_{\beta}(\bar{\rho}^{\varepsilon})dxds
      -\varepsilon N_K\int^t_0\int_{\mathbb{T}^d}\Big(\frac{\bar{\rho}^{\varepsilon}}{1-(\bar{\rho}^{\varepsilon})^2}\Big)\nabla \bar{\rho}^{\varepsilon}
\cdot \nabla \eta(x) \phi_{\beta}(\bar{\rho}^{\varepsilon})dxds\\
\notag
-&2\int^t_0\int_{\mathbb{R}}\int_{\mathbb{T}^d}\eta(x)\phi'_{\beta}(\xi)d\bar{q}^{\varepsilon}+\varepsilon M_K\int^t_0\int_{\mathbb{T}^d} (1-(\bar{\rho}^{\varepsilon})^2) \eta(x)\phi'_{\beta}(\bar{\rho}^{\varepsilon})dxds\\  \label{ss-9}
-&2\sqrt{\varepsilon}\int^t_0\int_{\mathbb{T}^d}\eta(x)\phi_{\beta}(\bar{\rho}^{\varepsilon})\nabla\cdot(\sqrt{1-(\bar{\rho}^{\varepsilon})^2}d\xi^{K})
=: \sum^9_{i=1}K_i .
\end{align}
Since the compact support of $\phi_{\beta}$ is on $(-1,1)$, by the same method as \cite[(4.17)-(4.19)]{DFG}, we have $\mathbb{P}-$a.s.,
\begin{align*}
  \lim_{\varepsilon\rightarrow 0}|K_6+K_8|=0,
\quad
  \lim_{\varepsilon\rightarrow 0}\sup_{t\in [0,T]}|K_9(t)|=0,
\end{align*}
and
\begin{align*}
\Big|\int_{\mathbb{R}}\int_{\mathbb{T}^d}\bar{\chi}^{\varepsilon}(x,\xi,t)\phi_{\beta}(\xi)\eta(x)dxd\xi
- 2\int_{\mathbb{T}^d}\rho^*(x,t)\eta(x)dx\Big|\leq c\beta,\quad \Big|K_1-2\int_{\mathbb{T}^d}\rho_0(x)\eta(x)dx\Big|\leq c\beta.
\end{align*}
For $K_2$ and $K_5$, by (\ref{ss-10}) and the boundedness of weak solution, it follows that
\begin{align*}
  &\lim_{\varepsilon\rightarrow 0}K_2=-2\int^t_0\int_{\mathbb{T}^d}\nabla\rho^*\cdot\nabla \eta(x)\phi_{\beta}(\rho^*) dx ds,\\
  &\lim_{\varepsilon\rightarrow 0}K_5= 2\int^t_0\int_{\mathbb{T}^d}\sqrt{1-(\rho^*)^2}g(x,s)\cdot \nabla \eta(x) \phi_{\beta}(\rho^*)dxds.
\end{align*}
Regarding the kernel term, due to (\ref{ss-10}), (\ref{eqq5}) and the boundedness of the weak solution, we deduce that
\begin{align*}
  \lim_{\varepsilon\rightarrow 0}K_3=-2\int^t_0\int_{\mathbb{T}^d}\eta(x)\phi_{\beta}(\rho^*)\nabla\cdot[(1-(\rho^*)^2) \nabla J\ast\rho^*]dxds.
\end{align*}
Based on the above, by taking $\varepsilon\rightarrow 0$ on both sides of (\ref{ss-9}),
we get
\begin{align}\notag
&2\Big|\int_{\mathbb{T}^d}\rho(x,t)\eta(x)dx-\int_{\mathbb{T}^d}\rho_0(x)\eta(x)dx
+\int^t_0\int_{\mathbb{T}^d}\nabla\bar{\rho}\cdot\nabla \eta(x)\phi_{\beta}(\bar{\rho}) dx ds\\
\notag
+& \int^t_0\int_{\mathbb{T}^d}\eta(x)\phi_{\beta}(\rho^*)\nabla\cdot[(1-(\rho^*)^2) \nabla J\ast\rho^*]dxds-\int^t_0\int_{\mathbb{T}^d}\sqrt{1-(\rho^*)^2}g(x,s)\cdot \nabla \eta(x) \phi_{\beta}(\rho^*)dxds\Big|\\
\label{ss-16}
\leq&
c\beta+ \limsup_{\varepsilon\rightarrow 0}(K_4+K_7).
\end{align}

Due to (\ref{ss-12}), we have
\begin{align*}
\limsup_{\varepsilon\rightarrow 0}\int^T_0\int_{\mathbb{R}}\int_{\mathbb{T}^d}\eta(x)\phi'_{\beta}(\xi)d\bar{q}^{\varepsilon}
\leq& \|\eta\|_{L^{\infty}(\mathbb{T}^d)}\limsup_{\varepsilon\rightarrow 0}\int^T_0\int_{\mathbb{R}}\int_{\mathbb{T}^d}(1-\xi^2)\phi'_{\beta}(\xi)(1-\xi^2)^{-1}d\bar{q}^{\varepsilon}\\
=& \|\eta\|_{L^{\infty}(\mathbb{T}^d)}\int^T_0\int_{\mathbb{R}}\int_{\mathbb{T}^d}(1-\xi^2)\phi'_{\beta}(\xi)dq.
\end{align*}
As a result of (\ref{ss-11}) and the dominated convergence theorem, it gives
\begin{align}\label{ss-14}
\lim_{\beta\rightarrow 0}\limsup_{\varepsilon\rightarrow 0}K_7=0.
\end{align}
The term $K_4$ can be handled by the same method as \cite[(4.23)-(4.24)]{DFG}.
It follows from H\"{o}lder inequality, (\ref{eqq6}) and (\ref{ss-8}) that
\begin{align*}
\limsup_{\varepsilon\rightarrow 0} |K_4|
\leq& 2\|\eta\|_{L^{\infty}(\mathbb{T}^d)} \|g^{\varepsilon}\|_{L^2([0,T]\times \mathbb{T}^d)}\limsup_{\varepsilon\rightarrow 0}\Big(\int^t_0\int_{\mathbb{T}^d}(1-(\bar{\rho}^{\varepsilon})^2)|\phi'_{\beta}(\bar{\rho}^{\varepsilon})|^2|\nabla \bar{\rho}^{\varepsilon}|^2dxds\Big)^{\frac{1}{2}}\\
\leq& 2\|\eta\|_{L^{\infty}(\mathbb{T}^d)} \|g^{\varepsilon}\|_{L^2([0,T]\times \mathbb{T}^d)}\limsup_{\varepsilon\rightarrow 0}\Big(\int^t_0\int_{\mathbb{T}^d}\int_{\mathbb{R}}(1-\xi^2)|\phi'_{\beta}(\xi)|^2\frac{d\bar{q}^{\varepsilon}}{1-\xi^2}\Big)^{\frac{1}{2}}\\
\leq& 2\sqrt{N}\|\eta\|_{L^{\infty}(\mathbb{T}^d)}\Big(\int^t_0\int_{\mathbb{T}^d}\int_{\mathbb{R}}(1-\xi^2)|\phi'_{\beta}(\xi)|^2dq\Big)^{\frac{1}{2}},
\end{align*}
where we have used (\ref{ss-12}) in the last inequality. Then, by (\ref{ss-11}), it yields
\begin{align}\label{ss-15}
\lim_{\beta\rightarrow 0}\limsup_{\varepsilon\rightarrow 0}K_4=0.
\end{align}
Passing to the limit $\beta\rightarrow 0$ on both sides of (\ref{ss-16}) and by (\ref{ss-14})-(\ref{ss-15}), we get
\begin{align*}
\int_{\mathbb{T}^d}\rho^*(x,t)\eta(x)dx=&\int_{\mathbb{T}^d}\rho_0(x)\eta(x)dx
-\int^t_0\int_{\mathbb{T}^d}\nabla\rho^*\cdot\nabla \eta(x)dx ds
- \int^t_0\int_{\mathbb{T}^d}\eta(x)\nabla\cdot[(1-(\rho^*)^2) \nabla J\ast\rho^*]dxds\\
&+\int^t_0\int_{\mathbb{T}^d}\sqrt{1-(\rho^*)^2}g(x,s)\cdot \nabla \eta(x) dxds.
\end{align*}
Thus, $\rho^*$ is the weak solution to skeleton equation (\ref{skeleton-ising}) with initial data $\rho_0$. By the uniqueness of (\ref{skeleton-ising}), we have $\rho^*=\rho^g$. We complete the proof.

\end{proof}

\noindent{\bf  Acknowledgements}\quad Zhengyan Wu acknowleges support by the Deutsche Forschungsgemeinschaft (DFG, German Research Foundation) via IRTG 2235 - Project Number 282638148. Rangrang Zhang acknowledges support by National Natural Science Foundation of China (No. 12171032), Beijing Institute of Technology Research Fund Program for Young Scholars and MIIT Key Laboratory of Mathematical Theory, Computation in Information Security.


%

\newcommand{\etalchar}[1]{$^{#1}$}

\end{document}